\documentclass[12pt]{article}
\usepackage{amsmath,enumerate,amsfonts,amssymb,color,graphicx,amsthm}

\usepackage{multirow,caption,subcaption}

\usepackage{hyperref}
\usepackage[normalem]{ulem}
\usepackage{tikz}
\usetikzlibrary{patterns}

\usepackage{cite}

\numberwithin{equation}{section}

\def\RR{{\mathbb R}}

\def\SSphere{{\mathbb S}}

\newcounter{marnote}

\newtheorem{theorem}{Theorem}[section]
\newtheorem{proposition}[theorem]{Proposition}
\newtheorem{corollary}[theorem]{Corollary}
\newtheorem{lemma}[theorem]{Lemma}

\newtheorem{remark}[theorem]{Remark}

\newtheorem{example}[theorem]{Example}

\newtheorem{manuallemmainner}{Lemma}
\newenvironment{manuallemma}[1]{%
  \renewcommand\themanuallemmainner{#1}%
  \manuallemmainner
}{\endmanuallemmainner}

\def\sign{{\rm sign}\,}
\def\deg{{\rm deg}\,}

\def\mcL{{\mycal L}}

\DeclareFontFamily{OT1}{rsfs}{}
\DeclareFontShape{OT1}{rsfs}{m}{n}{ <-7> rsfs5 <7-10> rsfs7 <10-> rsfs10}{}
\DeclareMathAlphabet{\mycal}{OT1}{rsfs}{m}{n}

\def\mcO{{\mycal O}}
\def\mcS{{\mycal S}}
\def\mcN{{\mycal N}}
\def\mcA{{\mycal A}}

\def\eps{{\varepsilon}}
\def\ringg{{\mathring{g}}}

\begin{document}
\title{The axisymmetric $\sigma_k$-Nirenberg problem}
\author{YanYan Li \thanks{Department of Mathematics, Rutgers University, Hill Center, Busch Campus, 110 Frelinghuysen Road, Piscataway, NJ 08854, USA. Email: yyli@math.rutgers.edu.}~\thanks{Partially supported by NSF Grants DMS-1501004, DMS-2000261, and Simons Fellows Award 677077. }~ and Luc Nguyen \thanks{Mathematical Institute and St Edmund Hall, University of Oxford, Andrew Wiles Building, Radcliffe Observatory Quarter, Woodstock Road, Oxford OX2 6GG, UK. Email: luc.nguyen@maths.ox.ac.uk.}~ and Bo Wang \thanks{School of Mathematics and Statistics, Beijing Institute of Technology, No. 5, Zhongguancunnan Street, Haidian District, Beijing 100081, China. Email: wangbo89630@bit.edu.cn.}~\thanks{Partially supported by NSF of China (11701027 and 11971061).}}

\date{}

\maketitle

\begin{abstract}
We study the problem of prescribing $\sigma_k$-curvature for a conformal metric on the standard sphere $\SSphere^n$ with $2 \leq k < n/2$ and $n \geq 5$ in axisymmetry. Compactness, non-compactness, existence and non-existence results are proved in terms of the behaviors of the prescribed curvature function $K$ near the north
and the south poles. For example, consider the case when the north and the south poles are local maximum points of $K$ of flatness order $\beta \in [2,n)$. We  prove among other things the following statements. (1) When
$\beta>n-2k$, the solution set is compact, has a nonzero total degree counting and is therefore  non-empty. (2) When $  \beta = n-2k$, there is an explicit positive constant $C(K)$ associated with $K$. If $C(K)>1$, the solution set is compact with a nonzero total degree counting and is therefore non-empty. If $C(K)<1$, the solution set is compact but the total degree counting is $0$, and the solution set is sometimes empty and sometimes non-empty. (3) When $\frac{2}{n-2k}\le \beta < n-2k$, the solution set is compact, but the total degree counting is zero, and the solution set is sometimes empty and sometimes non-empty. (4) When $\beta < \frac{n-2k}{2}$, there exists $K$ for which there exists a blow-up sequence of solutions with unbounded energy. In this same range of $\beta$, there exists also some $K$ for which the solution set is empty.
\end{abstract}

\tableofcontents
\section{Introduction}

We consider the $\sigma_k$-Nirenberg problem on the $n$-sphere $\SSphere^n$ ($n \geq 3$): Find a metric conformal to the standard metric on $\SSphere^n$ such that its $\sigma_k$-curvature is equal to a prescribed positive function on $\SSphere^n$.

Recall that, for a metric $g$ on $\SSphere^{n}$, the $\sigma_k$-curvature of $g$ is defined as follows. Let $Ric_{g}$, $R_{g}$ and $A_{g}$ denote respectively the Ricci curvature, the scalar curvature and the Schouten tensor of $g$:
\begin{equation*}
A_{g}=\frac{1}{n-2}\left(Ric_{g}-\frac{R_{g}}{2(n-1)}g\right).
\end{equation*}
Let $\lambda(A_g)$ denote the eigenvalues of $A_g$ with respect to $g$. For $1 \leq k \leq n$, the $\sigma_k$-curvature of $g$ is then the function $\sigma_k(\lambda(A_g))$ where $\sigma_k$ is the $k$-elementary symmetric function, $\sigma_k(\lambda)=\sum\limits_{i_{1}<\cdots<i_{k}}\lambda_{i_{1}}\cdots\lambda_{i_{k}}$. Our equation of interest is thus
\begin{equation}
\sigma_k(\lambda(A_{g})) = K \text{ and } \lambda(A_{g}) \in \Gamma_k \text{ on } \SSphere^n
	\label{Eq:29X20-NP}
\end{equation}
where $g$ is the unknown metric which is conformal to the standard metric, $K$ is a prescribed positive function on $\SSphere^n$, and $\Gamma_k$ is the connected component of $\{\lambda \in \RR^n: \sigma_k(\lambda) > 0\}$ which contains the positive cone $\{\lambda \in \RR^n: \lambda_1, \ldots, \lambda_n > 0\}$.

Let $\ringg$ denote the standard metric on $\SSphere^n$ and write the metric $g$ as $g_v =  v^{\frac{4}{n-2}}\ringg$ for some positive function $v$. Note that
\begin{equation*}
A_{g_{v}}=A_{\ringg}-\frac{2}{n-2}v^{-1}\nabla^{2}_{\ringg}v+\frac{2n}{(n-2)^{2}}v^{-2}dv\otimes  d v-\frac{2}{(n-2)^{2}}v^{-2}|dv|^{2}_{\ringg}\ringg.
\end{equation*}
Therefore, for $2 \leq k \leq n$, \eqref{Eq:29X20-NP} is a fully nonlinear elliptic equation for $v$. Similar equations involving eigenvalues of the Hessian of a function were first considered in \cite{C-N-S-Acta}.

In a recent paper \cite{LiNgWang-NPkLarge}, we started our study of the $\sigma_k$-Nirenberg problem. We proved an existence and compactness result in the case $k \geq n/2$ under the assumption that the prescribed curvature function $K$ satisfies certain non-degeneracy condition at its critical points, which generalized a result of Chang, Han and Yang \cite{CHY11-CVPDE} for $k = 2$ in dimension $4$. We refer the readers to \cite{LiNgWang-NPkLarge} for a discussion of related works. 

The compactness issue for the $\sigma_k$-Nirenberg problem as well as for the related $\sigma_k$-Yamabe problem on compact manifolds when $2 \leq k < n/2$ is a challenging open problem. In the present paper, we study this issue in the restrictive setting of axisymmetry. Namely, we view $(\SSphere^n, \ringg) = \{(x^1)^2 + \ldots + (x^{n+1})^2 = 1\}$ as the unit sphere embedded in $\RR^{n+1}$ and suppose that the functions $K$ and $v$ depend only on $\theta = \arccos x^{n+1}$. In addition, we assume that $K$ has the following behaviors at the north and south pole: there exist $a_{1}, a_{2}\neq 0$ and $\beta_1, \beta_2 > 1$ such that if we write
\begin{align*}
K(\theta) 
	=K(0)+a_{1} \theta^{\beta_{1}}+R_{1}(\theta) 
	=K(\pi)+a_{2}(\pi-\theta)^{\beta_{2}}+R_{2}(\theta)
\end{align*}
then
\begin{equation}
\lim_{\theta \rightarrow 0} \frac{|R_1(\theta)| + |\theta||R_1'(\theta)|}{|\theta|^{\beta_{1}}} = \lim_{\theta \rightarrow \pi} \frac{|R_2(\theta)| + |\pi - \theta||R_2'(\theta)|}{|\pi - \theta|^{\beta_{2}}}  = 0.
	\label{K condition}
\end{equation}

Our study is motivated by earlier works in the case $k = 1$ by Bianchi and Egnell \cite{BianchiEgnellMZ}, Chen and Lin \cite{ChenLin99-CPDE, CL00-ADE}, and Li \cite{Li95-JDE, Li96-CPAM}, where there is a qualitative difference in the analysis when the exponents $\beta_1, \beta_2$ belong to $(1,\frac{n-2}{2})$, $[\frac{n-2}{2},n-2)$, $\{n-2\}$ or $(n-2,n)$. To keep things simple and without losing depth, we focus our discussion in this paragraph to the case $a_1, a_2 < 0$ and $\beta_1 = \beta_2 = \beta$. When $n-2 < \beta < n$, the solution set of \eqref{Eq:29X20-NP} is compact, and the total Leray--Schauder degree of all solutions is $-1$. When $\frac{n-2}{2} \leq \beta  < n-2$, the solution set of \eqref{Eq:29X20-NP} is compact, and the total Leray--Schauder degree of all solutions is $0$. When $\beta = n-2$, the solution set of \eqref{Eq:29X20-NP} is compact provided $c \neq 1$ for certain explicit positive number $c$ depending only on $a_1, a_2, K(0)$ and $K(\pi)$, and the total Leray--Schauder degree is $-1$ when $c > 1$ and $0$ when $c < 1$. When $\beta < \frac{n-2}{2}$, there exist functions $K$ for which \eqref{Eq:29X20-NP} has a blow-up sequence of solutions with unbounded energy.

\begin{table}
\begin{center}
\caption{A summary of results for $\max\{\frac{n-2k}{2},2\} \leq \beta_1, \beta_2 < n$.}\label{Table1}

T = True. F = False. 
\\
T/F = Sometimes True and Sometimes False.\\
* = Sometimes True. ? = Unknown.

\bigskip

\begin{tabular}{cc}
Table \ref{Table1}(a): $a_1, a_2 > 0$ &  Table \ref{Table1}(b): $a_1 > 0 > a_2$\\
\begin{tabular}{l|c}
\hline
\multicolumn{2}{c}{Thm. \ref{main}, \ref{Thm:DegForm}}\\
\hline
Compactness & T   \\
Degree & $ -1$ \\
Existence & T\\
\hline
\hline
\end{tabular}
&
\begin{tabular}{l|c}
\hline
\multicolumn{2}{c}{Thm. \ref{main}, \ref{Thm:DegForm}, \ref{Thm:Perturb}} \\
\hline
Compactness & T  \\
Degree & $0$  \\
Existence & T/F \\
\hline
\hline
\end{tabular}
\end{tabular}

\bigskip
Table \ref{Table1}(c): $a_1, a_2 <  0$\\
\begin{tabular}{l|c|l|c|l|c}
\hline
\multicolumn{2}{c|}{$\frac{1}{\beta_1} + \frac{1}{\beta_2} < \frac{2}{n-2k}$} & \multicolumn{2}{c|}{$\frac{1}{\beta_1} + \frac{1}{\beta_2} = \frac{2}{n-2k}$} & \multicolumn{2}{c}{$\frac{1}{\beta_1} + \frac{1}{\beta_2} > \frac{2}{n-2k}$}\\
\multicolumn{2}{c|}{Thm. \ref{main}, \ref{Thm:DegForm}} &\multicolumn{2}{c|}{Thm. \ref{main}, \ref{Thm:DegForm}, \ref{Thm:Perturb},\ref{nonexist}} & \multicolumn{2}{c}{Thm. \ref{main}, \ref{Thm:Perturb}, \ref{nonexist}}\\
\hline
Compactness & T &Compactness & * &Compactness & T \\
 Degree & $ -1$ &  Degree & $-1/0/?$ &   Degree & $0$\\
 Existence & T & Existence & T/F &   Existence & T/F\\
\hline
\hline
\end{tabular}
\end{center}
\end{table}

Our present work extends the above results to the case $k \geq 2$. When $2 \leq \beta_1, \beta_2 < \frac{n-2k}{2}$ there exist functions $K$ for which \eqref{Eq:29X20-NP} has a blow-up sequence of solutions with unbounded energy; see Theorem \ref{noncompact}. For $\max\{\frac{n-2k}{2},2\} \leq \beta_1, \beta_2 < n$, our results are summarized in Table \ref{Table1}. In Table \ref{Table1}(a), when $a_1, a_2 > 0$, we have that the solution set is compact, the total degree for second order nonlinear elliptic operators is equal to $-1$ and \eqref{Eq:29X20-NP} has a positive solution. In Table \ref{Table1}(b), when $a_1$ and $a_2$ are of different signs, the solution set is compact, but the total degree is equal to $0$. In this case sometimes \eqref{Eq:29X20-NP} does not have a solution and sometimes it has a solution. If $K$ is strictly monotone, \eqref{Eq:29X20-NP} has no solution in view of the Kazdan--Warner-type identity. We also give examples of $K$'s for which \eqref{Eq:29X20-NP} has positive solutions. Let us describe Table \ref{Table1}(c) which concerns the case $a_1, a_2 < 0$ in more details. The analysis is split according to how $\frac{1}{\beta_1} + \frac{1}{\beta_2}$ compares to $\frac{2}{n-2k}$.
\begin{itemize}
\item When $\frac{1}{\beta_1} + \frac{1}{\beta_2} < \frac{2}{n-2k}$, the solution set is compact, the total degree is equal to $-1$ and \eqref{Eq:29X20-NP} has a positive solution.
\item When $\frac{1}{\beta_1} + \frac{1}{\beta_2} > \frac{2}{n-2k}$, the solution set is compact, the total degree is zero, and the existence of positive solution to \eqref{Eq:29X20-NP} depends on the particular $K$ at hand: there are examples of $K$'s which give existence as well as those which give non-existence for \eqref{Eq:29X20-NP}. 
\item When $\frac{1}{\beta_1} + \frac{1}{\beta_2} = \frac{2}{n - 2k}$, there exist functions $K$ for which the solution set is compact where the total degree can be $-1$ or $0$. Clearly, when the degree is $-1$, \eqref{Eq:29X20-NP} has a positive solution. There are examples of $K$'s for which \eqref{Eq:29X20-NP} has no solution. It is not known if the compactness of the solution set holds for every $K$.

\end{itemize}

For any integer $m\geq0$ and $0<\alpha\leq1$, let $C^{m}_{r}(\SSphere^{n})$ and $C^{m,\alpha}_{r}(\SSphere^{n})$ denote the spaces of $C^m$ and $C^{m,\alpha}$ axisymmetric functions on $\SSphere^n$, respectively.

In the statement of the next two theorems, let $C_{(1)} = C_{n,k}(\beta_1,a_1,K(0))$, $C_{(2)} = C_{n,k}(\beta_2,a_2,K(\pi))$ when $a_1, a_2 < 0$, where
\[
C_{n,k}(\beta, a, s)
	:= \frac{1}{2} \left[\frac{2\Gamma(n)s^{\frac{n-\beta}{2k}} }{  |a| \beta \Gamma(\frac{n-\beta}{2})\Gamma(\frac{n+\beta}{2}) } \right]^{\frac{1}{\beta}} \text{ for } \frac{n(n-2k)}{n+2k} < \beta < n, a < 0, s > 0.
\]

\begin{theorem}\label{main}
Let $n\geq 5$, $2\leq k < n/2$, $0 < \alpha \leq 1$, $K  \in C_{r}^{2,\alpha}(\SSphere^{n})$ be positive and satisfy \eqref{K condition} for some $a_1, a_2 \neq 0$ and $2 \leq \beta_1,\beta_2 < n$. Assume that
\begin{enumerate}[(i)]
\item if $\beta_i < \frac{n-2k}{2}$ for some $i \in \{1,2\}$, then $a_i > 0$, and
\item if $\frac{1}{\beta_1} + \frac{1}{\beta_2} = \frac{2}{n-2k}$, $\frac{n(n-2k)}{n+2k} < \beta_1,\beta_2 < n$, and $a_1,a_2 < 0$, then
\begin{equation}
C_{(1)}C_{(2)} \neq 1.
	\label{Eq:BalCond}
\end{equation}
\end{enumerate}
Then there exists a constant $C > 0$ such that all positive solutions of \eqref{Eq:29X20-NP} in $C_{r}^{2}(\SSphere^{n})$ satisfy  
\[
 \|\ln v\|_{C^{4,\alpha}(\SSphere^{n})}< C.
\]
\end{theorem}

See Remark \ref{Rem:27II21-R1} for detailed statement on how $C$ depends on the function $K$. See Subsection \ref{SSec:CptExt} for further compactness results involving a family of $K$'s.

We make a comment on condition \eqref{Eq:BalCond}. In the case $k =1$ and $\beta_1 = \beta_2 = n-2$, a similar condition was given in \cite{Li96-CPAM}. The relevance of this condition in the study of compactness issues is shown more clearly when one considers a family of $K$'s in \eqref{Eq:29X20-NP}. More precisely, for any positive $K \in C^{2,\alpha}_r(\SSphere^n)$ satisfying \eqref{K condition} with $\frac{1}{\beta_1} + \frac{1}{\beta_2} = \frac{2}{n-2k}$, $\frac{n(n-2k)}{n+2k} < \beta_1,\beta_2 < n$, $a_1,a_2 < 0$ for which $C_{(1)}C_{(2)} = 1$, there exists a sequence of positive functions $\{K_i\} \subset C^{2,\alpha}_r(\SSphere^n)$ which satisfies \eqref{K condition} with $\beta_{1,i} = \beta_1$, $\beta_{2,i} = \beta_2$, $C_{(1,i)} \rightarrow C_{(1)}$, $C_{(2,i)} \rightarrow C_{(2)}$ and which converges in $C^{2,\alpha}(\SSphere^n)$ to $K$  such that there exists a blow-up sequence of positive solutions to \eqref{Eq:29X20-NP} with $K$ replaced by $K_i$. This is a consequence of the homotopy invariance property of the degree and the degree counting formula in Theorem \ref{Thm:DegForm} below. The proof is similar to that of \cite[Corollary 0.24]{Li96-CPAM} in the case $k = 1$. Our analysis also shows that such sequence of solutions blow up at both the north and south poles; see the proof of Theorem \ref{Thm:27II21-T1} or Lemma \ref{Lem:05VI21-L1}.

For $k = 1$, analogous compactness results were proved by Li \cite{Li95-JDE, Li96-CPAM} and by Chen and Lin \cite{CL00-ADE}. Roughly speaking, compactness of the solution set was obtained in \cite[Theorem 0.19]{Li96-CPAM} when $n -2 < \beta_1,\beta_2 < n$, in \cite[Theorem 0.20]{Li96-CPAM} for $\beta_1 = \beta_2 = n-2$, and in \cite[Theorem 1.2]{CL00-ADE} for $\frac{n-2}{2} < \beta_1, \beta_2 <n $ satisfying $\frac{1}{\beta_1} + \frac{1}{\beta_2} \neq \frac{2}{n-2}$. We remark that when $k = 1$ and $\frac{1}{\beta_1} + \frac{1}{\beta_2} = \frac{2}{n - 2}$ the corresponding compactness result also holds but is not available in the literature except for the case $\beta_1 = \beta_2 = n-2$ mentioned above. We will publish this result elsewhere.

As a direct application of the above compactness result and available degree computation (see \cite{CHY11-CVPDE, Li95-JDE, LiNgWang-NPkLarge}), we have:

\begin{theorem}\label{Thm:DegForm}
Assume that $n, k, \alpha, K, a_1, a_2, \beta_1$, $\beta_2$ and $C$ be as in Theorem \ref{main}. Then
\begin{align*}
&\deg\Big(\sigma_k(\lambda(A_{g_v}))-K,\Big\{v\in C_{r}^{4,\alpha}(\SSphere^{n}): v > 0, \lambda(A_{g_v}) \in \Gamma_k, \|\ln v\|_{C^{4,\alpha}(\SSphere^{n})}<C\Big\},0\Big)\nonumber\\
&\qquad= \left\{\begin{array}{ll}
-1 &\text{ if } a_1, a_2 > 0,\\
0 & \text{ if } a_1 < 0 <  a_2 \text{ or } a_2 < 0 < a_1,\\
-1&\text{ if } a_1, a_2 < 0 \text { and } \frac{1}{\beta_1} + \frac{1}{\beta_2} < \frac{2}{n-2k},\\
-1&\text{ if } a_1, a_2 < 0, \frac{1}{\beta_1} + \frac{1}{\beta_2} = \frac{2}{n-2k}, \text { and } C_{(1)}C_{(2)} > 1,\\
0 & \text{ if } a_1, a_2 < 0, \frac{1}{\beta_1} + \frac{1}{\beta_2} = \frac{2}{n-2k}, \text { and }  C_{(1)}C_{(2)} < 1,\\
0&\text{ if } a_1, a_2 < 0 \text { and } \frac{1}{\beta_1} + \frac{1}{\beta_2} > \frac{2}{n-2k}.
\end{array}\right.
\end{align*}
Here $\deg$ is the degree for second order nonlinear elliptic operators as defined in \cite{Li89-CPDE}. In particular, in the cases where the resulting degree is non-zero, \eqref{Eq:29X20-NP}  has at least one positive solution in $C_{r}^{4,\alpha}(\SSphere^{n})$.
\end{theorem}

Our next two results concern the case the total degree is zero, i.e. when $a_1$ and $a_2$ are of opposite signs, or $a_1, a_2 < 0$ but $\frac{1}{\beta_1} + \frac{1}{\beta_2} > \frac{2}{n-2k}$ or $\frac{1}{\beta_1} + \frac{1}{\beta_2} = \frac{2}{n-2k}$ and $C_{(1)}C_{(2)} < 1$. In these situations, the existence of solutions depends on the particular $K$ at hand. Our next result shows that for any given signs of $a_1$ and $a_2$ and any given values of $\beta_1, \beta_2 \geq 2$, there exists $K$ for which \eqref{Eq:29X20-NP} has a solution.

\begin{theorem}\label{Thm:Perturb}
Assume that $n\geq 5$ and $2\leq k < n/2$. For any given signs $\varepsilon_1, \varepsilon_2 \in \{-1,1\}$ and constants $\beta_1,\beta_2 \geq 2$, there exist some non-zero constants $a_1, a_2$ with $sign(a_i) = \varepsilon_i$ and a positive function $K \in C_{r}^2(\SSphere^{n}) \cap C^\infty_{loc}(\SSphere^n \setminus \{\theta = 0,\pi\})$ satisfying \eqref{K condition} with the above $a_i$'s and $\beta_i$'s such that \eqref{Eq:29X20-NP} has at least one positive solution in $C_{r}^4(\SSphere^{n})$.
\end{theorem}

On the other hand, as mentioned earlier, if $K$ is monotone in one direction (which implies that $a_1$ and $a_2$ are of opposite signs), \eqref{Eq:29X20-NP} has no solution in view of the Kazdan--Warner-type identity (see \cite{Viac00-ams, Han06}, and also Section \ref{Sec:Prelims}). Our next result asserts that for any $2 \leq \beta_1, \beta_2< n$ satisfying $\frac{1}{\beta_1} + \frac{1}{\beta_2} \geq \frac{2}{n-2k}$, there exists a function $K$ with $a_1, a_2 < 0$ for which \eqref{Eq:29X20-NP} has no solution.

\begin{theorem}\label{nonexist}
Let $n\geq5$ and $2\leq k<\frac{n}{2}$. For any given $2\leq \beta_{1}$, $\beta_{2}< n$ with $\frac{1}{\beta_1} + \frac{1}{\beta_2} \geq \frac{2}{n-2k}$, there exists a positive function $K \in C_{r}^{[\beta],\beta -[ \beta]}(\SSphere^{n}) \cap C^\infty_{loc}(\SSphere^n \setminus \{\theta = 0,\pi\})$, $\beta = \min\{\beta_1,\beta_2\}$ satisfying (\ref{K condition}) with the above $\beta_1, \beta_2$ and some $a_{1}$, $a_{2}<0$ such that (\ref{Eq:29X20-NP}) admits no positive solution in $C^2(\SSphere^{n})$, with or without axisymmetry.
\end{theorem}

When $k = 1$, a similar non-existence result of axisymmetric solutions was proved by Bianchi-Egnell \cite[Theorem 0.3]{BianchiEgnellMZ} under the assumption $\frac{n(n-2)}{n+2}<\beta_{1}$, $\beta_{2} <  n$ and $\frac{1}{\beta_1}+\frac{1}{\beta_2}\geq\frac{2}{n-2}$, and by Chen-Lin \cite[Theorem 1.3]{CL00-ADE} under the assumption $\beta_{1}$, $\beta_{2}>1$ and $\frac{1}{\min\{\beta_{1},n\}}+\frac{1}{\min\{\beta_{2},n\}}>\frac{2}{n-2}$. Under certain monotonicity of $K$, it was shown in Bianchi \cite{Bianchi96-CPDE} that the axisymmetry of $K$ implies that of solutions for the prescribed scalar curvature equation. These results together give the counterpart of Theorem \ref{nonexist} for $k = 1$.

Our next theorem is a non-compactness result when $2 \leq \beta_{1} = \beta_{2}<\frac{n-2k}{2}$.

\begin{theorem}\label{noncompact}
Let $n\geq5$ and $2\leq k<\frac{n}{2}$. For any given $2 \leq \beta_{1} = \beta_{2}<\frac{n-2k}{2}$, there exists a positive function $K \in C^{[\beta_1],\beta_1 - [\beta_1]}(\SSphere^{n}) \cap C^\infty_{loc}(\SSphere^n \setminus \{\theta = 0,\pi\}) $ satisfying (\ref{K condition}) with $a_{1}$ = $a_{2}<0$ and a sequence of positive solutions $\{v_i\} \subset C_r^2(\SSphere^n)$ of \eqref{Eq:29X20-NP} such that, for some constant $C > 0$ depending only on $n$ and $\beta$,
\[
C\ln\ln \max_{\SSphere^n} v_i \geq \int_{\SSphere^n} v_i^{\frac{2n}{n-2}}\,dv_{\ringg} \geq \frac{1}{C} \ln\ln \max_{\SSphere^n} v_i\rightarrow \infty.
\]
\end{theorem}

For $k = 1$, the existence of blow-up sequences of solutions was proved by Chen and Lin \cite[Theorem 1.1]{ChenLin99-CPDE}, though without an estimate on the rate of blow-up for $\int_{\SSphere^n} v_i^{\frac{2n}{n-2}}\,dv_{\ringg}$ as in our result above.  

An ingredient in the proof of Theorems \ref{main}--\ref{noncompact}  is a fine analysis near a blow-up point in rotational symmetry. Consider in $B_2 \subset \RR^n$ the equation
\begin{equation}
\sigma_k(\lambda(A_{u_i^{\frac{4}{n-2}}g_{\rm flat}}))=K_{Euc},~~\lambda(A_{u_i^{\frac{4}{n-2}}g_{\rm flat}})\in\Gamma_k \text{ in } B_2,
	\label{Eq:25XII20-B1Eq}
\end{equation}
where $u_i(0) \rightarrow \infty$ and $K_{Euc} \in C^{2,\alpha}(B_2)$ satisfies for some $2 \leq \beta < n$ the condition
\begin{equation}
K_{Euc}(r) = K_{Euc}(0) + a r^\beta + R(r)
	\label{Eq:25XII20-B1KCond}
\end{equation}
with $|R(r)| + r|R'(r)| = o(r^\beta)$ as $r \rightarrow 0$. We give in Theorem \ref{Prop:LocEst} a description of $u_i$ as a `sum of bubbles' as $i \rightarrow \infty$. To keep things simple in this introduction, let us state here a consequence of it instead of the full result.

\begin{theorem}\label{Prop:LocEnergy}
Let $n\geq5$ and $2\leq k<\frac{n}{2}$. Suppose that $K_{Euc} \in C^{2,\alpha}(B_2)$, $0 < \alpha \leq 1$, is positive, rotationally symmetric and satisfies \eqref{Eq:25XII20-B1KCond} for some $a \neq 0$ and $\beta \geq 2$. Suppose that $u_i \in C^2(B_2)$ are positive, rotationally symmetric and satisfy \eqref{Eq:25XII20-B1Eq} and that $u_i(0)  \rightarrow \infty$. Then:
\begin{enumerate}[(i)]
\item When $\frac{n-2k}{2} \leq \beta < n$, the integral $\int_{B_1} u_i^{\frac{2n}{n-2}}\,dx$ is bounded as $i \rightarrow \infty$.
\item When $2 \leq \beta < \frac{n-2k}{2}$,
\[
\lim_{i\rightarrow \infty} \frac{1}{\ln\ln u_i(0)} \int_{B_1} u_i^{\frac{2n}{n-2}}\,dx = \frac{C(n,k)}{|\ln (1 - \frac{2\beta}{n-2k})|} K_{Euc}(0)^{-\frac{n}{2k}}
\]
for some constant $C(n,k) > 0$ depending only on $n$ and $k$.
\end{enumerate}
\end{theorem}

We note that, when $\frac{n-2k}{2} < \beta < n$, the sequence $\{u_i\}$ contains exactly one bubble, i.e. $\int_{B_1} u_i^{\frac{2n}{n-2}}\,dx$ converges to $C(n,k) K_{Euc}(0)^{-\frac{n}{2k}}$ for some positive constant $C(n,k)$ depending only on $n$ and $k$. When $\beta = \frac{n-2k}{2}$, we know that $\{u_i\}$ contains at least one bubble. (See Theorem \ref{Prop:LocEst}.) It is interesting to understand whether $\{u_i\}$ can contain two or more bubbles.

When $k=1$, statement (i) in Theorem \ref{Prop:LocEnergy} was proved by Li \cite{Li95-JDE} for $\beta \geq n-2$ and by Chen and Lin \cite{ChenLin98-JDG} for $\frac{n-2}{2} \leq \beta < n-2$.

The rest of the paper is structured as follows. In Section \ref{Sec:Prelims} we derive some useful integral identities for the $\sigma_k$-Nirenberg problem in axisymmetry. These integral identities contain the well-known Pohozaev identity as well as some other identities which we refer to as mass-type identities (see subsection \ref{SSec:MTIds}). In Section \ref{Sec:LocAnal}, we give a fine analysis of near a blow-up point for the $\sigma_k$-Yamabe problem on Euclidean balls. In Sections \ref{Sec:Comp}--\ref{Sec:NonComp}, we use the local analysis above to prove Theorems \ref{main}--\ref{noncompact}. We include also an appendix where certain integrals used in the body of the paper are computed in terms of the gamma function.


\section{Preliminaries}\label{Sec:Prelims}

In this section we give some equivalent forms of \eqref{Eq:29X20-NP} for positive $v \in C^2_r(\SSphere^n)$ and derive some useful integral identities, among which is the Pohozaev identity. 

We let $r = \cot\frac{\theta}{2}$, $t = \ln \cot\frac{\theta}{2}$ and express $g_v$ as a metric conformal to the Euclidean metric or the round cylinder metric:
\[
g_v = v(\theta)^{\frac{4}{n-2}} (d\theta^2 + \sin^2\theta \ringg_{\SSphere^{n-1}}) = u(r)^{\frac{4}{n-2}} (dr^2 + r^2\ringg_{\SSphere^{n-1}}) = e^{-2\xi(t)} (dt^2 + \ringg_{\SSphere^{n-1}}).
\]
Define $K_{Euc}(r) := K(\theta) =: K_{cyl}(t)$.

We will use a prime and a dot to mean differentiation with respect to $r$ and $t$ respectively.

One can explicitly express $u$ in terms of $v$ as follows. Let $\Phi: \RR^n \rightarrow \SSphere^n $ be the inverse of the stereographic projection:
\[
x^{i} =\frac{2y^{i}}{1+|y|^{2}} \text{ for }i=1,\ldots,n, 
\text{ and }
x^{n+1} =\frac{|y|^{2}-1}{|y|^{2}+1}.
\]
Then 
\begin{equation}
u(y)=\left(\frac{2}{1+|y|^{2}}\right)^{\frac{n-2}{2}}v(x), \qquad K_{Euc} = K \circ \Phi,
	\label{Eq:19XI20-E1}
\end{equation}
and \eqref{Eq:29X20-NP} is equivalent to
\begin{equation}\label{Eq:19XI20-E2}
\sigma_k(\lambda(A^u))=K_{Euc},~~\lambda(A^u)\in\Gamma_k \text{ in } \RR^n.
\end{equation}
where $A^u$ is the matrix
\[
A^u = -\frac{2}{n-2} u^{-\frac{n+2}{n-2}} \nabla^2 u + \frac{2n}{(n-2)^2} u^{-\frac{2n}{n-2}} du \otimes du - \frac{2}{(n-2)^2} u^{-\frac{2n}{n-2}} |du|^2\,I.
\]

Likewise, with $r = |y|$ and $t = \ln |y|$, we have 
\begin{equation}
\xi = -\frac{2}{n-2}\ln u - \ln r, \qquad K_{cyl}(t) = K \circ \Phi(e^t, 0, \ldots, 0),
	\label{Eq:12XII20-E2}
\end{equation}
and \eqref{Eq:29X20-NP} gives
\begin{equation}
F_k[\xi] = K_{cyl}
\text{ and }  |\dot \xi| < 1 \text{ in } (-\infty,\infty),
	\label{Eq:11XI20-E1}
\end{equation}
where
\begin{equation}
F_k[\xi] := \frac{1}{2^{k-1}} \binom{n-1}{k-1} e^{2k\xi}(1 - \dot \xi^2)^{k-1}\Big(\ddot \xi + \frac{n-2k}{2k}(1 - \dot \xi^2)\Big).
	\label{Eq:FkDef}
\end{equation}

The condition \eqref{K condition} is equivalent to the condition that 
\begin{align*}
K_{cyl} 
	&= K_{cyl}(\infty) + 2^{\beta_1} a_1 e^{-\beta_1 t} + o(e^{-\beta_1 t} ) \text{ as }t \rightarrow \infty,
	\\
K_{cyl} 
	&= K_{cyl}(-\infty) + 2^{\beta_2} a_2 e^{\beta_2 t} + o(e^{\beta_2 t}) \text{ as } t \rightarrow -\infty,
\end{align*}
with the error terms being controlled up to and including first order derivatives.

We note a simple property of the equation \eqref{Eq:11XI20-E1} which we will make use later on: There exists a constant $\bar x$ depending only on $n, k$ and an upper bound for $K_{cyl}$ such that
\begin{equation}
\text{If } \xi(t) \geq \bar x \text{ and } \dot\xi(t) = 0, \text{ then } \ddot \xi(t) < 0.
	\label{Eq:12III21-LMCrit}
\end{equation}

\subsection{Pohozaev-type identities}

For $\xi \in C^2(\RR)$, following \cite{Viac00-Duke}, define
\[
\bar H(\xi,\dot \xi) := \frac{1}{2^k} \binom{n}{k} e^{(2k-n)\xi}(1- \dot \xi^2)^k -  e^{-n\xi}.
\]
Then $\bar H$ has the property that, for $-\infty < t_1 \leq t_2 < \infty$,
\begin{equation}
\bar H(\xi(t_2),\dot\xi(t_2)) - \bar H(\xi(t_1),\dot\xi(t_1)) = -n \int_{t_1}^{t_2} (F_k[\xi] - 1)e^{-n\xi}\dot\xi\,dt.
	\label{Eq:12III21-BHP1}
\end{equation}

We will also consider the quantity
\begin{equation}
H(t,\xi,\dot \xi) := \frac{1}{2^k} \binom{n}{k} e^{(2k-n)\xi}(1- \dot \xi^2)^k -  K_{cyl}(t)\,e^{-n\xi}.
	\label{Eq:HDef}
\end{equation}
As a consequence of \eqref{Eq:12III21-BHP1}, we have, for $-\infty < t_1 \leq t_2 < \infty$,
\begin{equation}
 H(t,\xi(t),\dot \xi(t))\Big|_{t=t_1}^{t=t_2}
	= \int_{t_1}^{t_2} \Big[-n (F_k[\xi] - K_{cyl}) e^{-n\xi} \dot\xi - \dot K_{cyl}\,e^{-n\xi}\Big]\,dt.
	\label{Eq:13XI20-M2pre}
\end{equation}

If $\xi$ satisfies \eqref{Eq:11XI20-E1} and $\xi(t) - |t|$ is bounded in $(-\infty,\infty)$ (e.g. if $\xi$ is related to a solution to \eqref{Eq:29X20-NP} via  \eqref{Eq:19XI20-E1} and \eqref{Eq:12XII20-E2}), we have $H(t,\xi,\dot \xi) \rightarrow 0$ as $t \rightarrow \pm \infty$ and \eqref{Eq:13XI20-M2pre} gives
\begin{equation}
H(t,\xi,\dot \xi) = -\int_{-\infty}^t \dot K_{cyl}(\tau)\,e^{-n\xi(\tau)}\,d\tau = \int_t^{\infty} \dot K_{cyl}(\tau)\,e^{-n\xi(\tau)}\,d\tau.
	\label{Eq:13XI20-M3}
\end{equation}
Equivalently, if we let $u$ be related to $\xi$ via \eqref{Eq:12XII20-E2} and define
\begin{multline*}
H_{Euc}(r,u,u') =
	 \frac{(-1)^k 2^{k}}{(n-2)^{2k}} \binom{n}{k} r^{n-2k} u^{\frac{2(n-2k)}{n-2}}  \Big[\frac{r u'}{u} \Big(\frac{ru'}{u} + n-2\Big)\Big]^k \\
	 - K_{Euc}(r) r^{n} u^{\frac{2n}{n-2}},
\end{multline*}
then
\begin{equation}
H_{Euc}(r,u,u') = - \int_0^r K_{Euc}'(s)\,u(s)^{\frac{2n}{n-2}}\,s^n\,ds = \int_r^{\infty} K_{Euc}'(s)\,u(s)^{\frac{2n}{n-2}}\,s^n\,ds.
	\label{Eq:13XI20-M4}
\end{equation}

The identities \eqref{Eq:12III21-BHP1}, \eqref{Eq:13XI20-M2pre}, \eqref{Eq:13XI20-M3} and \eqref{Eq:13XI20-M4} are known as Pohozaev identities for the $\sigma_k$-Yamabe equation. See \cite{Pohozaev65, BrezisNirenberg83-CPAM} for the case $k=1$, \cite{Viac00-ams,Han06} for $k \geq 2$.

\subsection{Mass-type identities}\label{SSec:MTIds}

More generally, one is interested in finding smooth functions $B, P: \RR \times \RR \times (-1,1) \rightarrow \RR$ such that 
\begin{equation}
B(t,\xi,\dot \xi)\Big|_{t = t_1}^{t = t_2} = \int_{t_1}^{t_2} F_k[\xi] P(\tau,\xi,\dot \xi)\,d\tau \text{ for all } -\infty < t_1 \leq t_2 < \infty,
	\label{Eq:BPId}
\end{equation}
i.e.
\[
\frac{d}{dt} B(t,\xi,\dot \xi) = F_k[\xi] P(t,\xi,\dot \xi).
\]

Let $A(t,x,y)$ be such that $\partial_{y} A(t,x,y) = e^{2kx}(1- y^2)^{k-1}P(t,x,y)$ where $x$ and $y$ are dummy variables standing for $\xi$ and $\dot\xi$. We compute, using \eqref{Eq:FkDef} and then integrating by parts,
\begin{align*}
&\int_{t_1}^{t_2} F_k[\xi] P(\tau,\xi,\dot \xi)\,d\tau\\
	&\quad= \frac{1}{2^{k-1}} \binom{n-1}{k-1} \int_{t_1}^{t_2}\Big[ \partial_{y} A(\tau,\xi,\dot\xi) \ddot\xi + \frac{n-2k}{2k}\partial_{y} A(\tau,\xi,\dot\xi) (1-\dot\xi^2)\Big]\,d\tau\\
	&\quad= \frac{1}{2^{k-1}} \binom{n-1}{k-1} A(t,\xi, \dot\xi)\Big|_{t=t_1}^{t = t_2}\\
		&\qquad + \frac{1}{2^{k-1}} \binom{n-1}{k-1} \int_{t_1}^{t_2}\Big[ - \partial_{t} A(\tau,\xi,\dot\xi) - \partial_x A(\tau,\xi,\dot\xi) \dot \xi + \frac{n-2k}{2k}\partial_{y} A(\tau,\xi,\dot\xi) (1-\dot\xi^2)\Big]\,d\tau.
\end{align*}
Hence, to obtain \eqref{Eq:BPId}, we impose that $A$ satisfies the first order PDE
\begin{equation}
- \partial_{t} A(t,x,y) - y \partial_x A(t,x,y) + \frac{n-2k}{2k} (1- y^2)\partial_{y} A(t,x,y)
	= 0 \text{ in } \{|y| < 1\}.
	\label{Eq:ADefPDE} 
\end{equation}
The projected characteristic curves of \eqref{Eq:ADefPDE} are given by $\{t + \frac{k}{n-2k} \ln \frac{1+y}{1-y} = \textrm{const}, x - \frac{k}{n-2k} \ln (1 - y^2) = \textrm{const} \}$. The general solution to \eqref{Eq:ADefPDE} thus takes the form
\[
A(t,x,y) =G\Big(t + \frac{k}{n-2k} \ln \frac{1+y}{1-y}, x - \frac{k}{n-2k} \ln (1 - y^2)\Big)
\]
for an arbitrary smooth function $G: \RR^2 \rightarrow \RR$. Putting things together we have
\begin{align}
&\frac{n-2k}{2^{k}n} \binom{n}{k} G\Big(t + \frac{k}{n-2k}\ln \frac{1+\dot\xi}{1-\dot\xi}, \xi - \frac{k}{n-2k} \ln (1 - \dot \xi^2)\Big)\Big|_{t=t_1}^{t = t_2}
	\nonumber\\
	&\quad =  \int_{t_1}^{t_2} \frac{F_k[\xi]}{ e^{2k\xi} (1- \dot\xi^2)^{k}} \Big\{\partial_t G\Big(\tau + \frac{k}{n-2k}\ln \frac{1+\dot\xi}{1-\dot\xi}, \xi - \frac{k}{n-2k} \ln (1 - \dot \xi^2)\Big)  + \nonumber\\
		&\qquad + \dot\xi \partial_x G\Big(\tau + \frac{k}{n-2k}\ln \frac{1+\dot\xi}{1-\dot\xi}, \xi - \frac{k}{n-2k} \ln (1 - \dot \xi^2)\Big)\Big\}  \,d\tau.
	\label{Eq:BPIdDetail}
\end{align}
We have therefore proved that, for any smooth function $G: \RR^2 \rightarrow \RR$, the following $B$ and $P$ satisfy \eqref{Eq:BPId}:
\begin{align*}
B(t,x,y) 
	&= \frac{n-2k}{2^{k}n} \binom{n}{k} G\Big(t + \frac{k}{n-2k}\ln \frac{1+y}{1- y}, \xi - \frac{k}{n-2k} \ln (1 - y^2)\Big),\\
P(t,x,y)
	&= \frac{1}{ e^{2kx} (1- y^2)^{k}} \Big\{\partial_t G\Big(t + \frac{k}{n-2k}\ln \frac{1+y}{1-y}, x - \frac{k}{n-2k} \ln (1 - y^2)\Big)  + \nonumber\\
		&\qquad + \dot\xi \partial_x G\Big(t + \frac{k}{n-2k}\ln \frac{1+y}{1-y}, x - \frac{k}{n-2k} \ln (1 - y^2)\Big)\Big\}.
\end{align*}

\begin{example}
It is readily seen that the choice $G(t,x) = \frac{n}{n-2k} e^{-(n-2k)x}$ in \eqref{Eq:BPIdDetail} implies the Pohozaev identities \eqref{Eq:12III21-BHP1} and \eqref{Eq:13XI20-M2pre}. 
\end{example}

\begin{example}
We will also use in the proof of Theorem \ref{main} the choice $G(t,x) = \frac{2}{n-2k}e^{\frac{n-2k}{2}(t-x)}$. This gives the quantity
\[
	m(t,\xi,\dot\xi) := \frac{1}{2^{k-1} n} \binom{n}{k} (1 + \dot\xi(t))^ke^{\frac{n-2k}{2}(-\xi(t) + t)}
\]
and the identity, for $-\infty < t_1 \leq t_2 < \infty$,
\begin{equation}
m(t,\xi,\dot\xi)\Big|_{t=t_1}^{t=t_2} = \int_{t_1}^{t_2} F_k[\xi](1 - \dot\xi)^{-(k-1)} e^{-\frac{n+2k}{2}\xi} e^{ \frac{n-2k}{2}\tau}\,d\tau.
	\label{Eq:12III21-mId}
\end{equation}

If $\xi$ satisfies \eqref{Eq:11XI20-E1} and $\xi(t) + |t|$ is bounded as $t \rightarrow - \infty$, we have $m(t,\xi,\dot\xi) \rightarrow 0$ as $t \rightarrow - \infty$ and \eqref{Eq:12III21-mId} implies
\begin{equation}
m(t,\xi,\dot\xi) 
	= \int_{-\infty}^t  F_k[\xi] (1 - \dot\xi)^{-(k-1)} e^{-\frac{n+2k}{2}\xi} e^{ \frac{n-2k}{2}\tau}\,d\tau.
	\label{Eq:21II21-Mass1}
\end{equation}

We will refer to identities \eqref{Eq:12III21-mId} and \eqref{Eq:21II21-Mass1} as mass-type identities.
\end{example}

\begin{example}
For further reference, we also note that the separable ansatz $P(t,x,y) = P_1(t)P_2(x)P_3(y)$ leads to the choice $G(t,x) = \frac{2}{n-2k} e^{(n-2k)(bx + ct)}$. This gives the quantity
\[
	m_{b,c}(t,\xi,\dot\xi) := \frac{1}{2^{k-1} n} \binom{n}{k} (1 - \dot\xi)^{-k(b+c)}(1 + \dot\xi)^{-k(b-c)} e^{(n-2k)(b\xi + ct)}
\]
and the identity, for $-\infty < t_1 \leq t_2 < \infty$,
\begin{multline*}
m_{b,c}(t,\xi,\dot\xi)\Big|_{t=t_1}^{t=t_2} = 2\int_{t_1}^{t_2} F_k[\xi](1 - \dot\xi)^{-k(b+c+1)}(1+\dot\xi)^{-k(b-c+1)}(b\dot\xi + c) \times\\
	\times e^{((n-2k)b - 2k)\xi + (n-2k)c\tau}\,d\tau.
\end{multline*}

It is readily seen that taking $b = -1$ and $c = 0$ gives the Pohozaev identities, and taking $b = -1$ and $c = 1$ gives the mass-type identities.

\end{example}


\section{Local blow-up analysis}\label{Sec:LocAnal}

Consider in $B_2 \subset \RR^n$ the equation \eqref{Eq:25XII20-B1Eq}, i.e.
\[
\sigma_k(\lambda(A^{u}))=K_{Euc},~~\lambda(A^{u})\in\Gamma_k \text{ in } B_2
\]
where $K_{Euc} \in C^{2,\alpha}(B_2)$ satisfies \eqref{Eq:25XII20-B1KCond} for some $2 \leq \beta < n$. In this section, we study the behavior of a sequence of positive rotationally symmetric solutions $\{u_i\}$ of \eqref{Eq:25XII20-B1Eq} with $u_i(0) \rightarrow \infty$.

As in the previous section, we work with cylindrical coordinates. Let $t = \ln r$, $\xi(t) = -\frac{2}{n-2} \ln u(r) - \ln r$, and $K_{cyl}(t) = K_{Euc}(r)$. Then $\xi(t) + t$ is bounded as $t \rightarrow -\infty$,
\[
F_k[\xi] = K_{cyl} \text{ and } |\dot\xi| < 1  \text{ in } (-\infty,\ln 2),
\]
and
\begin{align}
K_{cyl} 
	&= K_{cyl}(-\infty) + 2^{\beta} a e^{\beta t} + o(e^{\beta t} ) \text{ as }t \rightarrow -\infty,
	\label{Eq:25XII20-B1KcylCond}
\end{align}
with the error terms being controlled up to and including first order derivatives.

Throughout the section, let
\[
\Xi(t) := -\ln \frac{ e^t}{1 + e^{2t}} - \ln\Big(2^{\frac{1}{2}}\binom{n}{k}^{\frac{1}{2k}}\Big) .
\]
Note that solutions to $F_k[E] = 1$ and $|\dot E| < 1$ in $(-\infty,\infty)$ satisfying $H(t,E,\dot E) \equiv 0$ are given by
\[
E(t) = \Xi(t + \ln\lambda) = -\ln \frac{\lambda e^t}{1 + \lambda^2 e^{2t}} - \ln\Big(2^{\frac{1}{2}}\binom{n}{k}^{\frac{1}{2k}}\Big) \text{ for some } \lambda > 0.
\]

\begin{theorem}\label{Prop:LocEst}
Let $n \geq 5$ and $2 \leq k < n/2$. Suppose that $K_{Euc} \in C^{2,\alpha}(B_2)$, $0 < \alpha \leq 1$, is positive, rotationally symmetric and satisfies \eqref{Eq:25XII20-B1KCond} for some $a \neq 0$ and $\beta \geq 2$. Suppose that $u_i \in C^2(B_2)$ are positive, rotationally symmetric and satisfy \eqref{Eq:25XII20-B1Eq} and that $u_i(0)  \rightarrow \infty$. Let $t = \ln r$, $\xi_i(t) = -\frac{2}{n-2} \ln u_i(r) - \ln r$ and $\lambda_i = 2^{-\frac{1}{2}}\binom{n}{k}^{-\frac{1}{2k}}K_{Euc}(0)^{\frac{1}{2k}} u_i(0)^{\frac{2}{n-2}}$. 

\begin{enumerate}[(a)]

\item One has for some $C$ depending only on $n$ and $K_{Euc}$ that
\begin{equation}
\xi_i \geq - C  \text{ and } |\dot \xi_i| + |\ddot \xi_i| \leq C \text{ in } (-\infty,\ln \frac{3}{2}).
	\label{Eq:30IV21-E1}
\end{equation}
Furthermore, for every $\varepsilon_i \rightarrow 0^+$, $R_i \rightarrow \infty$, after passing to a subsequence, one has that $\frac{R_i}{\lambda_i} \rightarrow 0$ and, $0 \leq \ell \leq 2$,
\begin{equation}
\Big|\frac{d^\ell }{dt^\ell}\Big[ \xi_i(t) - \Xi(t + \ln \lambda_i) - \frac{1}{2k} \ln K_{Euc}(0)\Big]\Big| \leq \varepsilon_i \lambda_i^\ell\, e^{\ell t} \text{ in } (-\infty, \ln \frac{R_i}{\lambda_i}).
	\label{Eq:30IV21-E2}
\end{equation}
In particular, $\xi_i(\ln\frac{R_i}{\lambda_i}) = \ln R_i + O(1) \rightarrow \infty$ and there exists $t_{1,i} = - \ln\lambda_i  + o(1)$ such that $\dot \xi_i < 0$ in $(-\infty, t_{1,i})$ and $\dot \xi_i > 0$ in $(t_{1,i} , \ln \frac{R_i}{\lambda_i})$.

\item Let
\[
t_{2,i} = \sup\Big\{t \in [t_{1,i},0]: \dot\xi_i > 0 \text{ in } (t_{1,i},t)\Big\}.
\]
Then, for large $i$,
\begin{equation}
t_{2,i} =  -\max\Big\{1 - \frac{\beta}{n-2k},0\Big\} \ln\lambda_i + O(1) > t_{1,i}
	\label{Eq:30IV21-E3}
\end{equation}
$\dot \xi_i > 0$ in $(t_{1,i},t_{2,i})$ and
\begin{equation}
\Big| \xi_i(t) - \Xi(t + \ln \lambda_i) \Big| \leq O(1) \text{ in } (-\infty, t_{2,i}).
	\label{Eq:30IV21-E4}
\end{equation}
Furthermore, if $\beta \leq n-2k$, then $a < 0$.

\item Suppose $2 \leq \beta < n-2k$. Then, for large $i$, $t_{2,i} < 0$, $\dot\xi_i(t_{2,i}) = 0$, and $\ddot\xi_i(t_{2,i}) < 0$. Let
\[
t_{3,i} = \sup\Big\{t \in [t_{2,i},0]: \dot\xi_i < 0 \text{ in } (t_{2,i},t)\Big\}.
\]
Then
\begin{equation}
t_{3,i} =  -\max\Big\{1 -\frac{2\beta}{n-2k},0\Big\} \ln \lambda_i + O(1) > t_{2,i}
	\label{Eq:30IV21-E5}
\end{equation}
$\dot \xi_i < 0$ in $(t_{2,i},t_{3,i})$, and 
\begin{equation}
\Big|\xi_i(t) - \Xi\Big(t + (1 -\frac{2\beta}{n-2k})\ln\lambda_i\Big)\Big| \leq O(1) \text{ in } (t_{2,i},t_{3,i}).
	\label{Eq:30IV21-E6}
\end{equation}

\item If $2 \leq \beta < \frac{n-2k}{2}$, then, for large $i$, there exist $N_i = \left\lfloor \frac{\ln \ln \lambda_i + O(1)}{|\ln (1 - \frac{2\beta}{n-2k})|}\right\rfloor \geq 2$ and $2N_i$ critical points of $\xi_i$,
\[
t_{1,i} < t_{2,i} < t_{3,i}  < t_{4,i} < \ldots < t_{2N_i,i} = O(1)
 \]
with
\begin{align*}
t_{2\ell,i} 
	&= -(1-\frac{\beta}{n-2k})(1-\frac{2\beta}{n-2k})^{\ell-1}\ln \lambda_i  + O(1), \\
 \quad t_{2\ell+1,i} 
 	&= -(1-\frac{2\beta}{n-2k})^{\ell}\ln\lambda_i + O(1),
\end{align*}
such that 
$\dot \xi_i < 0$ in $(t_{2\ell, i},t_{2\ell+1,i})$, $\dot \xi_i > 0$ in $(t_{2\ell + 1, i},t_{2\ell+2,i})$ and
\[
\Big|\xi_i(t) - \Xi(t - t_{2\ell + 1,i})\Big| \leq O(1) \text{ in } (t_{2\ell,i},t_{2\ell + 2,i}) \text{ for } 1 \leq \ell \leq N_i-1.
\]
Furthermore, for every $\varepsilon > 0$, there exists $R_\varepsilon > \frac{1}{\varepsilon}$ independent of $i$, such that, for any $\ell$ satisfying $|t_{2\ell+1,i}| \geq R_\varepsilon$, we have
\[
\|\xi_i(t) - \Xi(t - t_{2\ell + 1,i}) - \frac{1}{2k} \ln K_{Euc}(0)\|_{C^2[t_{2\ell+1,i} - 1/\varepsilon,t_{2\ell + 1,i} + 1/\varepsilon]} \leq \varepsilon.
\]

\end{enumerate}
Here $|O(1)| \leq C$, independent of $i$ and $\ell$ and $\varepsilon$,  and $o(1)$ denotes a term which goes to $0$ as $i \rightarrow \infty$.
\end{theorem}

A schematic sketch for the conclusion in (d) is given in Figure \ref{Fig1}. 

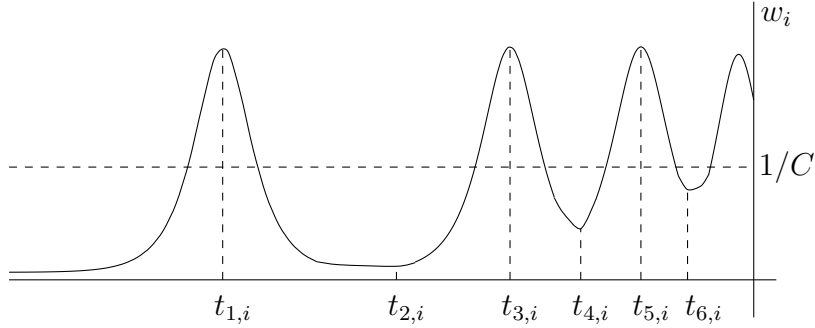
\begin{figure}
\caption{A profile of $w_i = r^{\frac{n-2}{2}}u_i$ vs. $t = \ln r$ with $N_i = 3$ when $\beta < \frac{n-2k}{2}$. The gap between the peaks decrease exponentially to $O(1)$ in $N_i$ steps. The pieces above the line $w = 1/C$ are close to the standard bubbles.}\label{Fig1}
\begin{center}
\begin{tikzpicture}
\clip (0.1,-2.2) rectangle (11,2.2);
\draw (.1,-1.5)--(10.3,-1.5);
\draw (10,-2)--(10,2.2);
\draw (10.3,2) node { $w_i$};
\draw[smooth,domain=0.1:4.2,variable=\t] plot ({\t},{3/cosh(3*(\t-2.94))-1.4});
\draw plot [smooth] coordinates {(4.2,-1.26) (4.5,-1.3) (5.25, -1.32) (5.5,-1.27)};
\draw[smooth,domain=5.5:7.35,variable=\t] plot ({\t},{3/cosh(3*(\t-6.76))-1.4});
\draw plot [smooth] coordinates {(7.35,-.41) (7.46,-.6) (7.7, -.82) (7.9,-.43)};
\draw[smooth,domain=7.9:9,variable=\t] plot ({\t},{3/cosh(3*(\t-8.5))-1.4});
\draw plot [smooth] coordinates {(9,-.12) (9.12, -.3) (9.3,-.25) (9.4,-.1)};
\draw[smooth,domain=9.4:10,variable=\t] plot ({\t},{3/cosh(3.5*(\t-9.8))-1.5});

\draw[dashed] (.1,0)--(10,0);
\draw (10.3,0) node {{ } $1/C$};
\draw[dashed] (2.94,-1.5) -- (2.94,1.6);
\draw (2.94,-1.9) node {{ } $t_{1,i}$};
\draw[dashed] (5.25,-1.5) -- (5.25,-1.32);
\draw (5.25,-1.9) node {{ } $t_{2,i}$};
\draw[dashed] (6.76,-1.5) -- (6.76,1.6);
\draw (6.76,-1.9) node {{ } $t_{3,i}$};
\draw[dashed] (7.7,-1.5) -- (7.7,-.82);
\draw (7.7,-1.9) node {{ } $t_{4,i}$};
\draw[dashed] (8.5,-1.5) -- (8.5,1.6);
\draw (8.50,-1.9) node {{ } $t_{5,i}$};
\draw[dashed] (9.12,-1.5) -- (9.12,-.3);
\draw (9.2,-1.9) node {{ } $t_{6,i}$};
\end{tikzpicture}
\end{center}
\end{figure}

\begin{corollary}\label{Cor:31XII20-InftyEnergy}
Under the hypotheses of Theorem \ref{Prop:LocEst}, when $\beta < \frac{n-2k}{2}$ we have
\[
\frac{1}{\ln \ln \lambda_i} \int_{B_1} u_i^{\frac{2n}{n-2}}\,dx  \rightarrow \frac{C(n,k)}{|\ln (1 - \frac{2\beta}{n-2k})|} K_{Euc}(0)^{-\frac{n}{2k}}\text{ as } i \rightarrow \infty.
\]
where $C(n,k) > 0$ depends only on $n$ and $k$. Furthermore, along a subsequence, $\xi_i$ converges in $C^2_{loc}(-\infty,\ln 2)$ to some $\xi_\infty \in C^2(-\infty,\ln 2)$ satisfying $F_k[\xi_\infty] = K_{cyl}$ and $|\dot\xi_\infty| < 1$ in $(-\infty,\ln 2)$ and there exist critical points of $\xi_\infty$,
\[
0 > t_{0,\infty} > t_{1,\infty} > \ldots  \rightarrow -\infty
\]
with
\begin{align*}
&\Big|t_{2j,\infty} - (1-\frac{2\beta}{n-2k})^{-j} t_{0,\infty}\Big| 
	\leq C,\\
&\Big|t_{2j+1,\infty} - (1-\frac{\beta}{n-2k})^{-1}(1-\frac{2\beta}{n-2k})^{-j} t_{0,\infty}\Big| 
	\leq C,
\end{align*}
such that $\dot \xi_\infty \leq 0$ in $(t_{2j+2, \infty},t_{2j,\infty})$, $\dot \xi_\infty \geq 0$ in $(t_{2j + 1, \infty},t_{2j,\infty})$ and
\[
\Big|\xi_\infty(t) - \Xi(t - t_{2j + 1,\infty})\Big| \leq C \text{ in } (t_{2j+2,\infty},t_{2j,\infty}) \text{ for } j \geq 0
\]
for some constant $C > 0$. Finally, for every $\varepsilon > 0$, there exists $R_\varepsilon > \frac{1}{\varepsilon}$, such that, for any $j$ satisfying $|t_{2j+1,\infty}| \geq R_\varepsilon$, we have
\[
\|\xi_\infty(t) - \Xi(t - t_{2j + 1,\infty}) -  \frac{1}{2k} \ln K_{Euc}(0)\|_{C^2[t_{2j+1,\infty} - 1/\varepsilon,t_{2j + 1,\infty} + 1/\varepsilon]} \leq \varepsilon.
\]
\end{corollary}

\subsection{An oscillation estimate}

We will make use of the following oscillation estimate for sub-solutions to $\sigma_k$-equation.

\begin{lemma}
\label{Lem:06XII20-FOCyl}
Assume that $n \geq 5$ and $2 \leq k < n/2$. There exist large constants $\xi_0 > 0$ and $C_0 > 0$ depending only on $n$ such that if $\xi$ is $C^2$ and monotone in some interval $[t_{1},t_{2}] \subset (0,\infty)$ and satisfies 
\[
0 \leq F_k[\xi] \leq 1, |\dot \xi| \leq 1, \text{ and } \xi \geq \xi_0 \text{ in } [t_1, t_2],
\]
then
\begin{align}
t_2 - t_1 \geq |\xi(t_2) - \xi(t_1)| \geq t_2 - t_1 - C_0.
\label{Eq:14IX20-E1Cyl}
\end{align}

\end{lemma}

For future reference, we state also here an equivalent version in Euclidean setting.

\begin{manuallemma}{\ref{Lem:06XII20-FOCyl}'}
\label{Lem:06XII20-FO}
Assume that $n \geq 5$ and $2 \leq k < n/2$. There exist a small constant $\varepsilon_0 > 0$ and a large constant $C$ depending only on $n$ such that if $u \in C^2(\bar B_{r_2} \setminus B_{r_1})$ is positive, rotationally symmetric,
\[
\sigma_k(\lambda(A^{u})) \leq 1, \quad \lambda(A^u) \in \Gamma_k  \text{ in }B_{r_2} \setminus \bar B_{r_1},
\]
$r^{\frac{n-2}{2}}u(r)$ is non-increasing (or non-decreasing, resp.), and $r_1^{\frac{n-2}{2}}u(r_1) \leq \varepsilon_0$, then
\[
1\leq\frac{r_{2}^{n-2}u(r_{2})}{r_{1}^{n-2}u(r_{1})}\leq C  \qquad \Big(\text{or } \frac{1}{C} \leq\frac{u(r_{2})}{u(r_{1})}\leq 1 ,\text{ resp.}\Big).
\]
\end{manuallemma}

Using cylindrical coordinates, the above result can be equivalently restated as:

\begin{proof}
By considering $\tilde \xi(t) = \xi(-t)$ instead of $\xi$ if necessary, it suffices to consider the case $\xi$ is non-decreasing. 

The first inequality in \eqref{Eq:14IX20-E1Cyl} holds due to the fact that $\dot \xi \leq 1$, so we only need to prove the second inequality.

By \eqref{Eq:12III21-LMCrit} and the fact that $F_k[\xi] \leq 1$, there exists a constant $\bar x$ depending only on $n$ and $k$ such that, whenever $\xi(t) > \bar x$ and $\dot \xi(t) = 0$, it holds that $\ddot \xi(t)  < 0$. Without loss of generality, we assume that $\xi > \bar x$ in $[t_1,t_2]$. Since $\dot\xi \geq 0$, this implies that
\begin{equation}
\dot \xi(t) > 0 \text{ and } \xi(t) < \xi(t_2) \text{ for }  t \in [t_1,t_2).
	\label{Eq:09III21-X1}
\end{equation}

For $x,y \in \RR$, let
\[
\bar H(x,y) = c\, e^{(2k-n)x}(1- y^2)^k - e^{-nx} \text{ where } c = \frac{1}{2^k} \binom{n}{k}.
\]
By \eqref{Eq:12III21-BHP1} and the fact that $F_k[\xi] \leq 1$ and $\dot \xi \geq 0$ in $[t_1, t_2]$, we have
\[
\frac{d}{dt} \bar H(\xi,\dot \xi) = - n e^{-n\xi}  (F_k[\xi]  - 1) \dot\xi  \geq 0 \text{ in } [t_1, t_2].
\]
Therefore
\begin{equation}
\bar H(\xi(t),\dot \xi(t)) \leq \bar H(\xi(t_2),\dot \xi(t_2)) \leq \bar H(\xi(t_2),0) \text{ for } t \in [t_1,t_2].
\label{12-1}
\end{equation}

By the explicit expression of $\bar H$, by increasing $\bar x$ if necessary, we may assume that $\bar H(\cdot,0)$ is decreasing and positive in $(\bar x, \infty)$. For $\bar x \leq x \leq a$, define
\[
g_a(x) = 1 - c^{-\frac{1}{k}} e^{\frac{n-2k}{k}x} (\bar H(a,0) + e^{-nx})^{\frac{1}{k}}.
\]
Then $g_a(a) = 0$, and by the monotonicity of $\bar H(x,0)$,
\[
g_a(x) > 1 - c^{-\frac{1}{k}} e^{\frac{n-2k}{k}x} (\bar H(x,0) + e^{-nx})^{\frac{1}{k}} = 0 \text{ for } x \in [\bar x, a).
\]

Using the explicit expression of $\bar H$ and the fact that $0 \leq \dot\xi \leq 1$, we can rewrite \eqref{12-1} as
\begin{equation}
\dot \xi \geq \sqrt{g_{\xi(t_2)}(\xi)} \text{ in }[t_1,t_2] \text{ provided } \xi(t_1) > \bar x.
	\label{Eq:04III21-DE1}
\end{equation}

\medskip
\noindent
Claim: There exist constants $C > 0$ and $\xi_0 > \bar x$ depending only on $n$ and $k$ such that
\begin{equation}
\int_x^{a} \frac{d\mu}{\sqrt{g_a(\mu)}} \leq a - x + C  \text{ for } \xi_0 \leq x \leq a.
\label{Eq:25XII20-R5}
\end{equation}
Clearly \eqref{Eq:09III21-X1}, \eqref{Eq:04III21-DE1} and \eqref{Eq:25XII20-R5} imply, for $\xi \geq \xi_0$ in $[t_1, t_2]$, that
\[
t_2 - t_1 \leq \int_{\xi(t_1)}^{\xi(t_2)} \frac{d\mu}{\sqrt{g_{\xi(t_2)}(\mu)}} \leq \xi(t_2) - \xi(t_1) + C,
\]
which gives the right half of \eqref{Eq:14IX20-E1Cyl}.

Using the inequality $(1 + z)^{\frac{1}{k}} \leq 1 + z^{\frac{1}{k}}$ for $z \geq 0$, we have for all $\bar x \leq x \leq a$ that
\begin{align*}
1 - g_a(x) 
	&= e^{\frac{n-2k}{k}(x-a)} \Big(1 - c^{-1} e^{-2ka} + c^{-1} e^{-nx + (n-2k)a}\Big)^{\frac{1}{k}} \nonumber\\
	&\leq e^{\frac{n-2k}{k}(x-a)} \Big(1 + c^{-1} e^{-nx + (n-2k)a}\Big)^{\frac{1}{k}} 
		\leq  e^{\frac{n-2k}{k}(x-a)}\Big(1 + c^{-\frac{1}{k}} e^{\frac{-nx + (n-2k)a}{k}}\Big)\\
	&= e^{\frac{n-2k}{k}(x-a)} + c^{-\frac{1}{k}} e^{-2x}.
\end{align*}
In particular, we can choose $\xi_0 = \xi_0(n,k)$ such that, for $\xi_0 \leq x \leq a - 1$,
\begin{align*}
1 - g_a(x) 
	\leq e^{\frac{n-2k}{k}(x-a)} + c^{-\frac{1}{k}} e^{-2x} \leq e^{-\frac{n-2k}{k}} + c^{-\frac{1}{k}} e^{-2\xi_0} < 1.
\end{align*}
This implies that there exists a constant $C = C(n,k,\xi_0)$ such that
\[
\frac{1}{\sqrt{g_a(x)}} = \frac{1}{\sqrt{1 - (1 - g_a(x))}}  \leq 1 + C(e^{\frac{n-2k}{k}(x-a)} +  e^{-2x}) \text{ for } \xi_0 \leq x \leq a - 1.
\]
On the other hand, by enlarging $\xi_0$ and $C$ if necessary, we have for $a - 1\leq x\leq a$ that
\begin{align*}
g_a(x) 
	&= 1 - e^{\frac{n-2k}{k}(x-a)} \Big(1 + c^{-1} e^{-2ka}(e^{n(a-x)} - 1)\Big)^{\frac{1}{k}}\\
	&\geq 1 - e^{\frac{n-2k}{k}(x-a)} \Big(1 + c^{-1} e^{-2k\xi_0}(e^{n(a-x)} - 1)\Big)^{\frac{1}{k}} \geq \frac{1}{C} (a - x).
\end{align*}

Combining the above estimates, we have for $\xi_0 \leq x \leq a$ that
\begin{align*}
\int_x^{a} \frac{d\mu}{\sqrt{g_a(\mu)}} 
	&\leq \int_{x}^{a-1} \Big[1 + C(e^{\frac{n-2k}{k}(\mu-a)} +  e^{-2\mu})\Big]d\mu
		+ \int_{a-1}^a \frac{Cd\mu}{\sqrt{a - \mu}}\\
	&\leq a - x + C(n,k),
\end{align*}
which gives the claim, and hence completes the proof.
\end{proof}

\subsection{Proof of Statement (a) of Theorem \ref{Prop:LocEst}}\label{SSec:LE(a)}

By first and second derivative estimates for the $\sigma_k$-Yamabe equation (see e.g. \cite[Theorem 1.1]{GW03-IMRN}, \cite[Theorem 1.10]{Li09-CPAM}), to prove \eqref{Eq:30IV21-E1}, we only need to show
\begin{equation}
u_i(r) r^{\frac{n-2}{2}}\leq C,
	\label{Eq:30IV21-A1}
\end{equation}
where here and below $C$ denotes a constant depending only on $n, k$ and $K_{Euc}$. 

The proof of \eqref{Eq:30IV21-A1} is a standard argument using the Liouville-type theorem \cite[Theorem 1.3]{LiLi05} and the symmetry of $u_i$.

Suppose by contradiction that \eqref{Eq:30IV21-A1} does not hold. Then, we can find $y_i \in B_{3/2}$ such that $|y_i|^{\frac{n-2}{2}} u_i(y_i) \rightarrow \infty$.

Let $r_i = |y_i|/2$, $\bar y_i \in \bar B_{r_i}(y_i)$ be a point where $(r_i - |y - y_i|)^{\frac{n-2}{2}} u_i(y)$ attains its maximum in $\bar B_{r_i}(y_i)$, and $s_i = (r_i - |\bar y_i - y_i|)/2 \in (0,r_i/2]$. It is clear that
\[
s_i^{\frac{n-2}{2}} u_i(\bar y_i) 
\geq 2^{-\frac{n-2}{2}}r_i^{\frac{n-2}{2}} u_i(y_i) \rightarrow \infty \text{ and } \max_{\bar B_{s_i}(\bar y_i)} u_i \leq 2^{\frac{n-2}{2}} u_i(\bar y_i). 
\]

Let
\[
\hat u_i(z) = \frac{1}{u_i(\bar y_i)} u_i\Big(\bar y_i + u_i(\bar y_i)^{-\frac{2}{n-2}}z\Big) \text{ for } |z| \leq s_i\,u_i(\bar y_i)^{\frac{2}{n-2}}.
\]
Then $\hat u_i$ satisfies
\[
\sigma_k(\lambda(A^{\hat u_i}(z)))=K_{Euc}(\bar y_i + u_i(\bar y_i)^{-\frac{2}{n-2}}z),~~\lambda(A^{\hat u_i})\in\Gamma_k \text{ in } \big\{|z| \leq s_i\,u_i(\bar y_i)^{\frac{2}{n-2}}\big\}.
\]

By first and second derivative estimates for the $\sigma_k$-Yamabe equation and the Liouville-type theorem \cite[Theorem 1.3]{LiLi05}, we thus have, after passing to a subsequence, that $\hat u_i$ converges in $C^{2}_{loc}(\RR^n)$ to a limit $\hat u_*$ of the form
\[
\hat u_*(z) = b_* (a_* + |z - z_*|^2)^{-\frac{n-2}{2}}
\]
for some positive constants $a_*$, $b_*$ and some $z_* \in \RR^n$.

On the other hand, the rotational symmetry of $u_i$ implies that, for every ball $B_r(y) \subset B_{s_i}(\bar y_i)$, the level set $\{u_i = u_i(y)\}$ intersects $\partial B_r(y)$ non-trivially. Applying this to balls centered at $\bar y_i + u_i(\bar y_i)^{-\frac{2}{n-2}}z_*$ and sending $i \rightarrow \infty$, we obtain that the level set $\{\hat u_* = \hat u_*(z_*)\}$ intersects every spheres centered at $z_*$. This is impossible as $z_*$ is a strict maximum point of $u_*$. Estimate \eqref{Eq:30IV21-A1} is proved.

Now, define 
\[
\tilde u_i(z) = \frac{1}{u_i(0)} u_i(\lambda_i^{-1} z) \text{ for } z \in \RR^n.
\]
By first and second derivative estimates for the $\sigma_k$-Yamabe equation and the Liouville-type theorem \cite[Theorem 1.3]{LiLi05}, we may assume after passing to a subsequence if necessary that $\tilde u_i$ converges in $C^2_{loc}(\RR^n)$ to 
\[
U(r) = (1 + r^2)^{-\frac{n-2}{2}}.
\]
Furthermore, for every $\varepsilon_i \rightarrow 0^+$ and every $R_i \rightarrow \infty$, after passing to a subsequence, we have 
\[
\|\tilde u_i - U\|_{C^2(B_{R_i})} \leq \varepsilon_i.
\]
This gives precisely estimate \eqref{Eq:30IV21-E2}. The last assertion of (a) also follows.

\subsection{Proof of Statement (b) of Theorem \ref{Prop:LocEst}}

Note that by (a) with $\varepsilon_i R_i \rightarrow 0$, we have $\dot\xi_i(t) > 0$ in $(t_{1,i}, \ln\frac{R_i}{\lambda_i})$. Clearly, by the definition of $t_{2,i}$, if $t_{2,i} < 0$ is finite, $\dot\xi_i(t_{2,i}) = 0$. Furthermore, we have for $\ln\frac{R_i}{\lambda_i} < t < t_{2,i}$ that
\begin{equation}
\xi_i(t)
	\geq \xi_i(\ln(R_i/\lambda_i))  \stackrel{\eqref{Eq:30IV21-E2}}{\geq} \ln R_i - O(1).
		\label{Eq:30IV21-B1}
\end{equation}

It follows from property \eqref{Eq:12III21-LMCrit} that $\ddot \xi_i < 0$ at every critical point of $\xi_i$ in $[\ln\frac{R_i}{\lambda_i},t_{2,i}]$ for large $i$. In particular, for large $i$, $\xi_i$ is strictly increasing in $[\ln\frac{R_i}{\lambda_i}, t_{2,i})$ and, if $t_{2,i}< 0$ is finite, then as $\dot\xi_i(t_{2,i}) = 0$, $\ddot\xi_i(t_{2,i}) < 0$.

Estimate \eqref{Eq:30IV21-E4} follows from \eqref{Eq:30IV21-E2}, \eqref{Eq:30IV21-B1}, the monotonicity of $\xi_i$ and Lemma \ref{Lem:06XII20-FOCyl}.

Let us now prove \eqref{Eq:30IV21-E3} and, when $\beta \leq n-2k$, the negativity of $a$. For $t \in (\ln\frac{R_i}{\lambda_i}, t_{2,i})$, we have by the Pohozaev identity \eqref{Eq:13XI20-M3} that
\[
\frac{1}{2^k} \binom{n}{k} \Big[1 - \dot  \xi_i(t)^2\Big]^k - e^{-2k\xi_i(t)} K_{cyl}(t)
	= e^{(n-2k)\xi_i(t)} \int_{-\infty}^{t} \dot K_{cyl}(\tau)e^{ - n\xi_i(\tau)}d\tau.
\]
Recalling \eqref{Eq:25XII20-B1KCond}, we see that $\dot K_{cyl}(\tau) = -a 2^{\beta} e^{\beta\tau}  L(\tau)$ for some bounded function $L$ satisfying $L(\tau) \rightarrow 1$ as $\tau \rightarrow -\infty$. Hence, using \eqref{Eq:30IV21-E2} with $\ell = 0$ in the interval $(-\infty, \ln \frac{R_i}{\lambda_i})$, \eqref{Eq:30IV21-E4} in the interval $(\ln \frac{R_i}{\lambda_i}, t)$ and noting that $\xi_i(t) = \ln\lambda_i + t + O(1) \geq \ln R_i + O(1) \rightarrow \infty$ as $i \rightarrow \infty$, we have 
\begin{align}
&\Big[1 - \dot \xi_i(t)^2\Big]^k + o(1)\nonumber\\
	&\quad= -(1 + o(1)) e^{(n-2k)\xi_i(t)} \lambda_i^{-\beta_2}  a\beta 2^{\beta + \frac{n + 2k}{2}}\binom{n}{k}^{\frac{n-2k}{2k}} K_{Euc}(0)\int_0^\infty  \frac{r^{n + \beta_2 - 1}}{(1+r^2)^{n}}\,dr.
	\label{Eq:30IV21-B3}
\end{align}
Since the right hand side of \eqref{Eq:30IV21-B3} is $- e^{O(1)}  a   e^{(n-2k) t}\lambda_i^{n-2k -\beta}$ (by \eqref{Eq:30IV21-E4}) and \eqref{Eq:30IV21-B3} holds for all $t \in (\ln \frac{R_i}{\lambda_i},t_{2,i})$, we have that $t_{2,i} = O(1)$ when $\beta \geq n - 2k$ and $t_{2,i} = -(1-\frac{\beta}{n-2k}) \lambda_i + O(1)$ for $\beta < n - 2k$, which gives \eqref{Eq:30IV21-E3}. In addition, by considering the sign of the left and right sides of \eqref{Eq:30IV21-B3}, we have that $a < 0$ when $\beta \leq n-2k$.

\subsection{Proof of Statement (c) of Theorem \ref{Prop:LocEst}}

Note that when $\beta < \frac{n-2k}{2}$, Statement (c) is contained in Statement (d). We consider here only the case $\frac{n-2k}{2} \leq \beta < n-2k$ and leave the case $\beta < \frac{n-2k}{2}$ to the proof of Statement (d).

We have seen that $t_{2,i} < 0$, $\dot\xi_i(t_{2,i}) = 0$ and $\ddot\xi_i(t_{2,i}) < 0$. 

Since $\dot\xi_i > -1$, we deduce from \eqref{Eq:30IV21-E4} that
\begin{equation}
\xi_i(t)
	\geq \xi_i(t_{2,i}) - (t - t_{2,i}) 
	\geq -\Big(1 - \frac{2\beta}{n-2k} \Big) \ln \lambda_i  - t - O(1) \text{ for } t \geq t_{2,i}.
	\label{Eq:30IV21-C1}
\end{equation}

As $\beta \geq \frac{n-2k}{2}$, estimate \eqref{Eq:30IV21-C1} implies that, for every $\xi_0 > 0$, there exists $\tilde t_{3,i} = O(1) \gg t_{2,i}$ such that $\xi_i > \xi_0$ in $(t_{2,i},\tilde t_{3,i})$. In view of \eqref{Eq:11XI20-E1} and \eqref{Eq:12III21-LMCrit}, when $\xi_0$ is sufficiently large, the function $\dot\xi_i$ in the interval $(t_{2,i},\tilde t_{3,i})$ has the property that, whenever $\dot\xi_i(t) = 0$, it holds that $\ddot \xi_i(t) < 0$. Note also that $\dot\xi_i(t) < 0$ for $t > t_{2,i}$ and close to $t_{2,i}$ (because $\dot\xi_i(t_{2,i}) = 0$ and $\ddot\xi_i(t_{2,i}) < 0$). These two properties imply that $\dot\xi_i < 0$ in $(t_{2,i},\tilde t_{3,i})$. In particular, $t_{3,i} \geq \tilde t_{3,i}$ and so $t_{3,i} = O(1)$, which gives \eqref{Eq:30IV21-E5}. Estimate \eqref{Eq:30IV21-E6} follows from Lemma \ref{Lem:06XII20-FOCyl} applied in the interval $(t_{2,i},\tilde t_{3,i})$ and the fact that $|\dot\xi| < 1$ in $[\tilde t_{3,i},t_{3,i}]$.

\subsection{Proof of Statement (d) of Theorem \ref{Prop:LocEst}}

\begin{proof}[\unskip\nopunct]

Suppose $2 \leq \beta < \frac{n-2k}{2}$. Recall from (b) that $a < 0$. For simplicity, we assume instead of \eqref{Eq:25XII20-B1KcylCond} that $K_{cyl}(t) = 1 - |a| e^{\beta t}$ in $(-\infty,0)$. The proof in the general case where $K_{cyl}$ satisfies only \eqref{Eq:25XII20-B1KcylCond} can be done by restricting attention to a small ball $B_{\delta}$ and an easy accommodation for error terms.

Note that $K_{cyl}(t) = 1 - |a|e^{\beta t}$ is decreasing and so by the Pohozaev identity \eqref{Eq:13XI20-M2pre}, the functions $H(t,\xi_i,\dot\xi_i)$ are increasing. Noting that $H(t,\xi_i,\dot\xi_i) \rightarrow 0$ as $t \rightarrow -\infty$ (since $\xi_i(t) + t$ is bounded as $t \rightarrow -\infty$), we thus have that $H(t,\xi_i,\dot\xi_i) > 0$ in $(-\infty,0]$.

Hence, by estimate \eqref{Eq:30IV21-E1} and the Pohozaev identity \eqref{Eq:13XI20-M3},
\begin{equation}
0 \leq \lim_{t \rightarrow -\infty} \sup_i H(t,\xi_i,\dot\xi_i) \leq \lim_{t \rightarrow -\infty} C\int_{-\infty}^t |\dot K_{cyl}|\,d\tau = 0.
	\label{Eq:09IV21-P1}
\end{equation}
where here and below $C$ denotes a constant which remains independent of $i$.

Recall that all solutions to $F_k[E] = 1$ in $(-\infty,\infty)$ satisfying $H(t,E,\dot E) \equiv 0$ are of the form
\[
E(t) = \Xi(t + \ln\lambda)  \text{ for some } \lambda > 0.
\] 
This implies:
\begin{lemma}\label{Eq:23IV21-M1}
Let $\xi_i$ satisfy $F_k[\xi_i] = K_i$ and $|\dot\xi_i| < 1$ in $(-\infty,\ln 2)$ where $K_i \in C^2(-\infty,\ln 2)$ satisfies\
\[
\sup_i \sup_{ t \in (-\infty,\ln2)} (|\ln K_i(t)| + |\dot K_i(t)| + |\ddot K_i(t)|) < \infty.
\]
For every $c_0 \in \RR$, if $s_i \leq 0$, $\xi_i(s_i) \rightarrow c_0$, $H(s_i,\xi_i(s_i),\dot\xi_i(s_i)) \rightarrow 0$ and, for some $0 \leq S_i \leq |s_i|$,
\[
\lim_{i\rightarrow \infty} K_i(t + s_i) = 1 \text{ for all } t \leq \liminf S_i,
\]
then $c_0 \geq \min \Xi = \ln\Big(2^{-\frac{1}{2}}\binom{n}{k}^{\frac{1}{2k}}\Big)$ and there exists $T_i \rightarrow \infty$ such that, after passing to a subsequence, 
\[
\|\xi_i(t + s_i) - \Xi(t + \bar t)\|_{C^{2}([-T_i,\tilde T_i])} \leq \delta_i, \qquad \tilde T_i = \left\{\begin{array}{ll}
T_i &\text{ if } S_i \rightarrow \infty,\\
 S_i &\text{ if } S_i \text{ is bounded},
 \end{array}\right.
\]
where $\bar t$ are one of the two solutions of $\Xi(\bar t) = c_0$ if $c_0 >  \min \Xi$ and $\bar t = 0$ if $c_0 =  \min \Xi$.
\end{lemma}

In the sequel, we fix some $\bar\xi_0 > \min \Xi$ which is larger than the constant $\xi_0$ in Lemma \ref{Lem:06XII20-FOCyl} and the constant $\bar x$ in \eqref{Eq:12III21-LMCrit}, and has the additional property that
\begin{equation}
	\begin{minipage}{.7\textwidth}
For any $C^2$ functions $\xi$, if $t$ satisfies $\xi(t) > \bar \xi_0$ and $\dot \xi(t) = 0$, then $\frac{1}{2} e^{-(n-2k)\xi(t)} \leq 2^{k} \binom{n}{k}^{-1} H(t,\xi(t),\dot\xi(t)) \leq e^{-(n-2k)\xi(t)}$.
	\end{minipage}
\label{Eq:07IV21-XP1}
\end{equation}

By Lemma \ref{Eq:23IV21-M1} and in view of \eqref{Eq:09IV21-P1}, there exists $m_0 > 0$ depending only on $(n,K,\bar \xi_0)$ such that, for each $i$, the number $\tilde N_i$ of points $s < -m_0$ such that $\xi_i(s) = \bar\xi_0$ and $\dot \xi_i(s) < 0$ is non-zero and finite. We label these points as $s_{1,i} < s_{2,i} < \ldots < s_{\tilde N_i,i}$. By the same lemma, if we let $m_0' > 0$ be the solution to $\Xi(-m_0' ) = \bar \xi_0$ with $\dot\Xi(-m_0') < 0$, then for every $\varepsilon > 0$, there exists $\tilde R_\varepsilon > \frac{2}{\varepsilon}$ independent of $i$ such that for any $\ell$ satisfying $|s_{\ell,i}| > \tilde R_\varepsilon$,
\begin{equation}
\|\xi_i(t + s_{\ell,i}) - \Xi(t - m_0')\|_{C^2[-2/\varepsilon,2/\varepsilon]} \leq \varepsilon .
	\label{Eq:06V21-X1}
\end{equation}
It is readily seen from \eqref{Eq:06V21-X1} and \eqref{Eq:12III21-LMCrit} that $\xi_i^{-1}(\bar \xi_0) \cap (-\infty, s_{\tilde N_i,i} + 9m_0'/4]$ comprises of $s_{1,i} < s_{1,i}'' < s_{2,i} < s_{2,i}'' < \ldots < s_{\tilde N_i,i} < s_{\tilde N_i,i}''$, and $\xi_i\big|_{(-\infty, s_{\tilde N_i,i}'']}$ has critical points $t_{1,i}, \ldots, t_{2\tilde N_i-1,i}$ such that
\[
s_{1,i} < t_{1,i} < s_{1,i}'' < t_{2,i} < s_{2,i}  < \ldots < s_{\tilde N_i,i}'' < t_{2\tilde N_i-2,i} < s_{\tilde N_i,i} < t_{2\tilde N_i-1,i} < s_{\tilde N_i,i}'',
\] 
$\dot\xi_i < 0$ in $(-\infty,t_{1,i})$ and $(t_{2\ell,i},t_{2\ell +1,i})$, and $\dot\xi_i > 0$ in $(t_{2\ell - 1,i},t_{2\ell,i})$, for $1 \leq \ell \leq \tilde N_i-1$. Furthermore, \eqref{Eq:30IV21-E6} holds.

By Statements (a) and (b), we have that $a < 0$, $t_{1,i} = -\ln \lambda_i + o(1)$, $t_{2,i} = -(1 - \frac{\beta}{n-2k})\ln\lambda_i + O(1)$, and $\xi_i(t_{2,i}) = \frac{\beta}{n-2k}\ln\lambda_i + O(1)$. 

To conclude, we need to show that there exists $1 \leq N_i \leq \tilde N_i - 1$, $N_i = \lfloor \frac{\ln\ln\lambda_i + O(1)}{|\ln(1-\frac{2\beta}{n-2k})|}\rfloor$ such that
\begin{enumerate}[(i)]
\item $t_{2\ell,i} = -\alpha_\ell \ln\lambda_i + O(1)$ for $2 \leq \ell \leq N_i$,
\item $t_{2\ell + 1,i} = -(\alpha_\ell - \gamma_\ell)\ln \lambda_i + O(1)$ for $1 \leq \ell \leq N_i - 1$,
\end{enumerate}
where $|O(1)| \leq C$, independent of $i$ and $\ell$, $\alpha_\ell = (1 -\frac{\beta}{n-2k})(1 -\frac{2\beta}{n-2k})^{\ell-1}$ and $\gamma_\ell = \frac{\beta}{n-2k}(1 -\frac{2\beta}{n-2k})^{\ell-1}$.  Note that by applying Lemma \ref{Lem:06XII20-FOCyl} to the intervals $[t_{2\ell,i},t_{2\ell+1,i}]$ and $[t_{2\ell+1,i},t_{2\ell+2,i}]$, we obtain from the above that 
\begin{align*}
\xi_i(t) 
	&= \xi_i(t_{2\ell+1,i}) - t + t_{2\ell+1,i} + O(1) =  -(\alpha_\ell - \gamma_\ell) \ln\lambda_i - t + O(1) \text{ in } [t_{2\ell,i},t_{2\ell+1,i}],\\
\xi_i(t) 
	&= \xi_i(t_{2\ell+1,i}) + t - t_{2\ell+1,i} + O(1) =  (\alpha_\ell - \gamma_\ell)\ln \lambda_i + t  + O(1) \text{ in } [t_{2\ell+1,i},t_{2\ell+2,i}].
\end{align*}
In other words 
\[
\xi_i(t) = \Xi(t - t_{2\ell+1,i}) + O(1) \text{ in }[t_{2\ell,i},t_{2\ell+2,i}].
\]

To prove (i)-(ii), we use the following lemma, which is of independent interest and can be applied in a situation more general that what is described above. (Note that no assumption at $-\infty$ is assumed in the lemma.) Recall that $\bar\xi_0 $ is a constant larger than the constant $\xi_0$ in Lemma \ref{Lem:06XII20-FOCyl} and the constant $\bar x$ in \eqref{Eq:12III21-LMCrit}, and has the property \eqref{Eq:07IV21-XP1}.

\begin{lemma}\label{Lem:ConnectingBubble}
Let $a < 0$ and $\beta \in (0,n-2k)$ and suppose $K_{cyl} \in C^{2,\alpha}(-\infty,\ln 2)$, $0 < \alpha \leq 1$, satisfies \eqref{Eq:25XII20-B1KcylCond}. For every given constant $D \geq 0$, there exists some large $M = M(n,K_{cyl},D,\bar\xi_0) > 1$ such that if $\xi \in C^2(-\infty,\ln 2)$ satisfies $F_k[\xi] = K_{cyl}(t)$ and $|\dot \xi| < 1$ in $(-\infty,\ln 2)$, and if $t_* < 0$ is a critical point of $\xi$ satisfying
\[
-(n-2k)\xi(t_*) - D \leq \beta(t_* + \xi(t_*)) \leq -M,
\]
 then $\xi(t_{*}) > \bar \xi_0$, $\ddot\xi(t_{*}) < 0$, and there exist critical points $t_* < t_{*+1} < t_{*+2} < 0$ of $\xi$ such that $\xi(t_{*+1}) < \ln\Big(2^{-\frac{1}{2}}\binom{n}{k}^{\frac{1}{2k}}\Big) + \frac{1}{M}$, $\xi(t_{*+2}) > \bar \xi_0$, $\ddot\xi(t_{*+1}) > 0$, $\ddot\xi(t_{*+2}) < 0$, $\dot\xi < 0$ in $(t_*,t_{*+1})$, $\dot\xi > 0$ in $(t_{*+1},t_{*+2})$, and
\begin{align}
&|t_{*+1} - (t_* + \xi(t_*))| \leq M,\label{Eq:08IV21-A1}\\
&\Big|t_{*+2} - \big(1 - \frac{\beta}{n-2k}\big)(t_* + \xi(t_*))\Big| \leq M,\label{Eq:08IV21-A3}\\
&|\xi(t) - \Xi(t - t_{*+1})| \leq M \text{ in } [t_*,t_{*+2}],\label{Eq:08IV21-AS}\\
&- (n-2k)\xi(t_{*+2}) \leq \beta(t_{*+2} + \xi(t_{*+2})) ,\label{Eq:08IV21-A5}\\
&  \Big|(t_{*+2} + \xi(t_{*+2})) - \big(1 - \frac{2\beta}{n-2k}\big)(t_* + \xi(t_*))\Big| \leq  M.\label{Eq:08IV21-A6}
\end{align}
\end{lemma}

Once this lemma is proved, we can obtain the conclusion as follows. Take $D = 0$ and fix $M$ as in the lemma. First, we have for all large $i$ that
\[
- (n-2k)\xi_i(t_{2,i}) \leq \beta(t_{2,i} + \xi_i(t_{2,i}))  \leq -  M.
\]
Let $N_i$ be the largest number in $\{2, \ldots,  \tilde N_i - 1\}$ such that $\beta(t_{2\ell,i} + \xi_i(t_{2\ell,i}) \leq -M$ for $1 \leq \ell \leq N_i $. Applying the lemma repeatedly with $t_* = t_{2\ell,i} < 0$ for $1 \leq \ell \leq N_i$, we have
\begin{equation}
 \Big| (t_{2\ell+2,i} + \xi_i(t_{2\ell+2,2})) - \Big(1 - \frac{2\beta}{n-2k}\Big)(t_{2\ell,i} + \xi_i(t_{2\ell,i}))\Big| \leq M. 
 \label{Eq:16IV21-X4}
\end{equation}
(Note that if $N_i = \tilde N_i - 1$, the lemma also gives the existence of another local maximum point $t_{2\tilde N_i,i} \in ( s_{\tilde N_i,i}'',0)$ of $\xi_i$.) This implies for $2 \leq \ell \leq N_i + 1$ that
\[
\Big|(t_{2\ell,i} + \xi_i(t_{2\ell,i}))  - \big(1 - \frac{2\beta}{n-2k}\big)^{\ell-1}(t_{2,i} + \xi_i(t_{2,i})) \Big| \leq M \sum_{j = 0}^{\ell - 2} (1 - \frac{2\beta}{n-2k})^j \leq \frac{n-2k}{2\beta}M.
\]
Since $t_{2,i} + \xi_i(t_{2,i}) + \big(1 - \frac{2\beta}{n-2k}\big)\ln\lambda_i$ is bounded as $i \rightarrow \infty$, we thus have for $1 \leq \ell \leq N_i + 1$ that
\begin{equation}
\Big|t_{2\ell,i} + \xi_i(t_{2\ell,i})  + \big(1 - \frac{2\beta}{n-2k}\big)^\ell \ln\lambda_i\Big|  \leq C,
	\label{Eq:18IV21-M1}
\end{equation}
where $C$ is independent of $i$ and $\ell$. Returning to \eqref{Eq:08IV21-A1} and \eqref{Eq:08IV21-A3} (still with $t_* = t_{2\ell,i}$), we see that the declared properties (i) and (ii) hold.

To finish the proof, we show that $N_i \geq  \lfloor \frac{\ln\ln\lambda_i + O(1)}{|\ln(1-\frac{2\beta}{n-2k})|}\rfloor =: \hat N_i$. (Note that $t_{2\hat N_i,i} \geq -C$ for some $C$ independent of $i$ and $t_{2\ell + 2,i} - t_{2\ell + 1,i} \geq m_0'/4$ for all $\ell$, this estimate gives $N_i = \hat N_i + O(1) = \lfloor \frac{\ln\ln\lambda_i + O(1)}{|\ln(1-\frac{2\beta}{n-2k})|}\rfloor$.)

 In view of \eqref{Eq:18IV21-M1} with $\ell = N_i$ and the fact that $\beta(t_{2N_i,i} + \xi_i(t_{2N_i,i})) \leq -M$, we only need to show that $t_{2N_i,i} + \xi_i(t_{2N_i,i}) \geq -C$ for some $C$ independent of $i$. By \eqref{Eq:16IV21-X4}, it suffices to show that $t_{2N_i+2,i} + \xi_i(t_{2N_i+2,i}) \geq -C$. To this end, we may assume without loss of generality that $\beta(t_{2N_i+2,i} + \xi_i(t_{2N_i+2,i})) \leq - M$, as otherwise there is nothing to prove. By the lemma, we can find critical points $t_{2 N_i+2,i} < t_{2 N_i+3,i} < t_{2 N_i+4,i} < 0$ of $\xi$ where $\xi(t_{2N_i+2,i}) > \bar \xi_0 >  \xi(t_{2N_i+3,i})$ and $\dot \xi < 0$ in $(t_{2N_i+2,i}, t_{2N_i+3,i})$. In particular, there exists $s_{N_i + 2,i} \in (t_{2N_i+2,i}, t_{2N_i+3,i})$ such that $\xi(s_{N_i + 2,i}) = \bar\xi_0$ and $\dot\xi(s_{N_i + 2,i}) < 0$. By construction of the sequence $\{s_{\ell,i}\}$, we have $s_{N_i + 2,i} \geq -m_0$. It follows that $t_{2N_i+3,i} \geq -m_0$. Recalling \eqref{Eq:08IV21-A1} with $t_{*} = t_{2N_i+2,i}$, we thus have $t_{2N_i+2,i} + \xi_i(t_{2N_i+2,i}) \geq -C$ as wanted. Theorem \ref{Prop:LocEst} follows.
\end{proof}

\begin{proof}[Proof of of Lemma \ref{Lem:ConnectingBubble}]
In the proof we will frequently use the function $H$ defined in \eqref{Eq:HDef} and the Pohozaev identity \eqref{Eq:13XI20-M2pre}. For convenience, we write $H(t) := H(t,\xi(t),\dot\xi(t))$. 

For simplicity, we consider again only the case $K_{cyl}(t) = 1 - |a|e^{\beta t}$.

As in \eqref{Eq:30IV21-E1}, there exists $C' = C'(n,K_{cyl})$ such that
\[
\xi \geq - C'  \text{ and } |\dot \xi| + |\ddot \xi| \leq C' \text{ in } (-\infty,\ln \frac{3}{2}).
\]
By Lemma \ref{Eq:23IV21-M1}, there exist $m_0 > 10m_0' > 0$ depending only on $(n,K,\bar \xi_0)$ such that 
\begin{equation}
\begin{minipage}{.84\textwidth}
If $s \leq -m_0$ satisfies $\xi(s) =  \bar \xi_0$,  $\dot\xi(s)  < 0 $ and $|H(s)| \leq 1/m_0$ then there exist $s'' > s' > s$ such that $\xi(s'') = \bar \xi_0$, $\dot\xi(s') = 0$, $\ddot\xi(s') > 0$, $\dot\xi < 0$ in $[s, s')$, $\dot\xi > 0$ in $(s', s'']$, $3m_0'/4 \leq s' - s\leq 5m_0'/4$ and $7m_0'/4 \leq s'' - s\leq 9m_0'/4$.
\end{minipage}
	\label{Eq:07IV21-S1X}
\end{equation}

In the sequel, $M$ is a large constant which may need to be enlarged at a few instances in the proof but will depend only on $n, a, \beta, D, C'$ and $\bar\xi_0$. 

Since $-(n-2k)\xi(t_*) -D \leq -M$, we may take $M$ sufficiently large so that $\xi(t_*) > \bar\xi_0$. As $\xi(t_*) > \bar \xi_0$ and $\dot\xi(t_*) = 0$, we have by \eqref{Eq:12III21-LMCrit} that $\ddot \xi(t_*) < 0$ and $t_*$ is a local maximum point of $\xi$. We will show the existence of $t_{*+1}$ by showing that $\xi$ will decrease to the value $\bar \xi_0$ and appeal to \eqref{Eq:07IV21-S1X}.

Define
\[
s_0 = \sup \Big\{t \in [t_*, 0): \xi(t) > \bar \xi_0 \text{ in } [t_*, t]\Big\}.
\]
Since $\dot\xi(t) < 0$ for $t > t_*$ and close to $t_*$, we have by \eqref{Eq:12III21-LMCrit} that $\dot \xi < 0$ in $(t_*, s_0)$. Applying Lemma \ref{Lem:06XII20-FOCyl}, we have
\begin{align}
\xi(t_*) -  (t - t_*) \leq \xi(t) \leq \xi(t_*) - (t - t_*) + C_0
	\text{ for } t \in [t_*, s_0],
		\label{Eq:07IV21-E2}
\end{align}
where $C_0 > 0$ is the constant in Lemma \ref{Lem:06XII20-FOCyl}. Taking $t = s_0$ in \eqref{Eq:07IV21-E2} and using the fact that $\beta( t_* + \xi(t_*))  \leq - M$, we obtain, after possibly enlarging $M$, that
\[
s_0 \leq t_* + \xi(t_*) - \xi(s_0) + C_0 \leq -\frac{1}{\beta} M - \bar \xi_0 + C_0 \leq -m_0 < 0,
\] 
which implies that $\xi(s_0) = \bar \xi_0$ and
\begin{align}
(t_* + \xi(t_*))   - \bar\xi_0 \leq s_0  \leq (t_* + \xi(t_*))  - \bar \xi_0 + C_0.
		\label{Eq:07IV21-E3}
\end{align}

To use \eqref{Eq:07IV21-S1X}, we need to estimate $H(s_0)$. On one hand, by \eqref{Eq:07IV21-XP1} and the relation $-(n-2k)\xi(t_*) - D \leq \beta(t_* + \xi(t_*))$, we have 
\[
0 < H(t_*) \leq \frac{1}{2^k} \binom{n}{k} e^{-(n-2k)\xi(t_*)}  \leq  \frac{1}{2^k} \binom{n}{k} e^D e^{\beta(t_* + \xi(t_*))}.
\]
On the other hand, we have
\[
0 <  - \int_{t_*}^{s_0} \dot K_{cyl}e^{-n\xi}\,d\tau
 \stackrel{\eqref{Eq:07IV21-E2}, \eqref{Eq:07IV21-E3}}{\leq} \frac{|a|\beta e^{(n + \beta)(\bar \xi_0 + C_0)}}{n+\beta} e^{\beta(t_* + \xi(t_*))}.
\]
Thus, by the Pohozaev identity \eqref{Eq:13XI20-M2pre} and the fact that $\beta(t_* + \xi(t_*)) \leq -M$ and by possibly enlarging $M$,
\begin{eqnarray}
0 &<& H(s_0)
	= H(t_*) - \int_{t_*}^{s_0} \dot K_{cyl}e^{-n\xi}\,d\tau\nonumber\\
	& \leq& \Big( \frac{1}{2^k}\binom{n}{k}e^D  + \frac{|a|\beta\, e^{(n + \beta)(\bar \xi_0 + C_0)}}{n+\beta}\Big)e^{\beta (t_* + \xi(t_*))}
	\leq \frac{1}{m_0}.
	\label{Eq:07IV21-E4}
\end{eqnarray}
Therefore, by \eqref{Eq:07IV21-S1X}, there exist $s_1 > t_{*+1} > s_0$ 
such that $\xi(s_1) = \bar \xi_0$, $\dot\xi_i(t_{*+1}) = 0$, $\dot\xi < 0$ in $[s_0, t_{*+1})$, $\dot\xi > 0$ in $(t_{*+1}, s_1]$, $3m_0'/4 \leq t_{*+1} - s_0 \leq 5m_0'/4$ and $7m_0'/4 \leq s_1 - s_0 \leq 9m_0'/4$.

Clearly, \eqref{Eq:08IV21-A1} follows from  \eqref{Eq:07IV21-E3} and the bound $s_0 + 3m_0'/4 < t_{*+1} < s_0 + 5m_0'/4$.

From the above, we know that $\dot \xi > 0$ in $(t_{*+1},s_1)$. Define
\[
t_{*+2} = \sup \Big\{t \in [s_1, 0): \dot\xi(t) > 0 \text{ in } [s_1, t]\Big\}.
\]
Note that $\xi \geq \bar\xi_0$ in $[s_1, t_{*+2}]$, and so by \eqref{Eq:12III21-LMCrit}, $\dot \xi > 0$ in $[s_1,t_{*+2})$. We will show that when $M$ is suitably large, $t_{*+2} < 0$ and hence $t_{*+2}$ is a critical point of $\xi$.

By Lemma \ref{Lem:06XII20-FOCyl}, \eqref{Eq:07IV21-E3} and the fact that $\xi(s_1) = \bar\xi_0$ and $s_0 + 7m_0'/4 < s_1 < s_0 + 9m_0'/4$, we have 
\begin{equation}
t - (t_* + \xi(t_*)) - C
	\leq \xi(t) 
		\leq  t - (t_* + \xi(t_*)) + C  \text{ in } [s_1, t_{*+2}],
	\label{Eq:07IV21-E6}
\end{equation}
where here and below $C$ denotes a positive constant depending only on $n,a,\beta, D, C', \bar \xi_0, C_0$ and $m_0$. 
This together with \eqref{Eq:08IV21-A1} and \eqref{Eq:07IV21-E3} gives \eqref{Eq:08IV21-AS} after possibly enlarging $M$.

Let us now estimate $H(t_{*+2})$ in terms of $t_* + \xi(t_*)$. By the Pohozaev identity \eqref{Eq:13XI20-M2pre}, we have $H(t_{*+2}) = H(s_0)  -  \int_{s_0}^{t_{*+2}} \dot K_{cyl}e^{-n\xi}\,d\tau$. Using \eqref{Eq:07IV21-E3} and the inequalities $-C' \leq \xi \leq \bar\xi_0$ in $[s_0,s_1]$ and $7m_0'/4 \leq s_1 - s_0 \leq 9m_0'/4$, we have that
\begin{equation}
\frac{1}{C}e^{\beta (t_* + \xi(t_*))} \leq \int_{s_0}^{s_1} \dot K_{cyl}e^{-n\xi}\,d\tau \leq Ce^{\beta (t_* + \xi(t_*))}.
	\label{Eq:07IV21-E5}
\end{equation}
By \eqref{Eq:07IV21-E6}, we have
\[
\frac{1}{C}   e^{(\beta-n) t} e^{- n(t_* + \xi(t_*))} \leq -\dot K_{cyl}(t)e^{-n\xi(t)} \leq C e^{(\beta-n) t} e^{- n(t_* + \xi(t_*))} \text{ in } [s_1,t_{*+2}],
\]
and so, as $\beta < n$,
\begin{equation}
0 \leq -   \int_{s_1}^{t_{*+2}} \dot K_{cyl}e^{-n\xi}\,d\tau
	\leq Ce^{\beta(t_* + \xi(t_*))}.
	\label{Eq:07IV21-E7}
\end{equation}
Putting together \eqref{Eq:07IV21-E4}, \eqref{Eq:07IV21-E5} and \eqref{Eq:07IV21-E7}, we thus have
\begin{equation}
\frac{1}{C}e^{\beta(t_* + \xi(t_*))} \leq H(t_{*+2}) \leq  Ce^{\beta(t_* + \xi(t_*))}.
	\label{Eq:07IV21-E8}
\end{equation}
Recalling the expression of $H$ in \eqref{Eq:HDef} and using $t = t_{*+2}$ in \eqref{Eq:07IV21-E6}, we obtain
\[
1  \stackrel{\eqref{Eq:HDef}}{\geq} 2^k \binom{n}{k}^{-1} e^{(n-2k)\xi(t_{*+2})}H(t_{*+2})
	\stackrel{\eqref{Eq:07IV21-E6},\eqref{Eq:07IV21-E8}}{\geq}
	 \frac{1}{C} e^{(n-2k)t_{*+2} }  e^{-(n-2k-\beta)(t_* + \xi(t_*))},
\]
which, in view of the fact $\beta(t_* + \xi(t_*)) \leq -M$, leads to
\begin{equation}
t_{*+2} \leq \frac{n-2k-\beta}{n-2k}(t_* + \xi(t_*)) + C \leq -\frac{n-2k-\beta}{\beta(n-2k)}M + C.
	\label{Eq:07IV21-E9}
\end{equation}

As $\beta < n-2k$, the right hand side of \eqref{Eq:07IV21-E9} can be made negative by enlarging $M$. Recalling the definition of $t_{*+2}$, we thus have $\dot\xi(t_{*+2}) = 0$ and, by \eqref{Eq:12III21-LMCrit}, $\ddot \xi(t_{*+2}) < 0$. 

As $\xi(t_{*+2}) > \xi(s_1) = \bar\xi_0$ and $\dot\xi(t_{*+2}) = 0$ and in view of \eqref{Eq:07IV21-XP1} and \eqref{Eq:07IV21-E8}, we have 
\[
-\frac{\beta}{n-2k}(t_* + \xi(t_*)) - C \leq \xi(t_{*+2}) \leq -\frac{\beta}{n-2k}(t_* + \xi(t_*)) + C,
\]
and, in view of \eqref{Eq:07IV21-E6} with $t = t_{*+2}$,
\[
\frac{n-2k-\beta}{n-2k}(t_* + \xi(t_*)) - C \leq t_{*+2} \leq  \frac{n-2k-\beta}{n-2k}(t_* + \xi(t_*)) + C.
\]
These give \eqref{Eq:08IV21-A3}. They also give
\begin{align*}
(n-2k)\xi(t_{*+2}) + \beta(t_{*+2} + \xi(t_{*+2})) 
	&\geq  -  \frac{2\beta^2}{n-2k}(t_* + \xi(t_*)) - C,\\
\Big(1 -  \frac{2\beta}{n-2k}\Big)(t_* + \xi(t_*)) - C
	&\leq
t_{*+2} + \xi(t_{*+2}) 
	\leq \Big(1 -  \frac{2\beta}{n-2k}\Big)(t_* + \xi(t_*)) + C.
\end{align*}
In view of the fact that $\beta(t_* + \xi(t_*)) \leq - M$, by enlarging $M$ one final time, we obtain \eqref{Eq:08IV21-A5} and \eqref{Eq:08IV21-A6} as desired.
\end{proof}


\section{Compactness estimates: Proof of Theorem \ref{main}}\label{Sec:Comp}

In this section, we give the proof of Theorem \ref{main} together with some extensions.

\subsection{Proof of Theorem \ref{main}}

\begin{proof}[\unskip\nopunct]
By first and second derivative estimates for the $\sigma_k$-Yamabe equation (see e.g. \cite[Theorem 1.1]{GW03-IMRN}, \cite[Theorem 1.10]{Li09-CPAM}), it suffices to show that 
\[
v \leq C_1 \text{ for all positive $C_r^{2}$ solutions $v$ of \eqref{Eq:29X20-NP}}
\]
where $C_1$ depends only on $n, k$ and $K$. Suppose by contradiction that there exist positive functions $v_i \in C_{r}^{2}(\SSphere^{n})$ satisfying \eqref{Eq:29X20-NP} such that $\max v_i \rightarrow \infty$.

Let $u_i: \RR^n \rightarrow \RR$ be related to $v_i$ as in \eqref{Eq:19XI20-E1}. As $u_i$ is super-harmonic and rotationally symmetric, the maximum principle implies that $u_i(0)$ is the maximum of $u_i$ in any closed ball centered at the origin. Recalling \eqref{Eq:19XI20-E1}, we have
\begin{equation}
v_i(x) \leq v_i(S) \Big(\frac{2}{1 + \cos d_{\ringg}(x,S)}\Big)^{\frac{n-2}{2}} \text{ for all } x \in \SSphere^n \setminus \{N\},
	\label{Eq:17III21-E0}
\end{equation}
where $N$ and $S$ are respectively the north and south poles of $\SSphere^n$. In particular, $v_i \leq 2^{\frac{n-2}{2}} v_i(S)$ in the lower closed hemi-sphere. Likewise $v_i \leq 2^{\frac{n-2}{2}}v_i(N)$ in the upper closed hemi-sphere. As $\max v_i \rightarrow \infty$, this implies that
\begin{equation}
\max \{v_i(S),v_i(N)\} \rightarrow \infty.
	\label{Eq:17III21-E1}
\end{equation}

Throughout the proof, $C$ denotes some generic positive constant which may change from one line to another but will remain independent of $i$, $O(1)$ denotes a term which is bounded as $i \rightarrow \infty$, and $o(1)$ denotes a term which tends to zero as $i \rightarrow \infty$.

\medskip
\noindent \underline{Step 1:} We show that
\[
v_i(x) d_{\ringg}(x,\{N,S\})^{\frac{n-2}{2}} \leq C,
\]
and
\begin{equation}
|\nabla^\ell \ln v_i(x)|d_{\ringg}(x,\{N,S\})^\ell \leq C \text{ for } \ell = 1, 2.
\label{Eq:26XI20-Est2}
\end{equation}

These estimates follow from Theorem \ref{Prop:LocEst}(a) and \eqref{Eq:17III21-E0}.

In the next step, let $a_2$ and $\beta_2$ be as given in \eqref{K condition}, $t = \ln r$, $\xi_i$ be related to $u_i$ as in \eqref{Eq:12XII20-E2} and $\lambda_i := 2^{-\frac{1}{2}}\binom{n}{k}^{-\frac{1}{2k}}K(S)^{\frac{1}{2k}} u_i(0)^{\frac{2}{n-2}} = 2^{\frac{1}{2}}\binom{n}{k}^{-\frac{1}{2k}}K(S)^{\frac{1}{2k}} v_i(S)^{\frac{2}{n-2}}$.

\medskip
\noindent \underline{Step 2:} Making use of Pohozaev-type and mass-type identities, we show that if $v_i(S) \rightarrow \infty$, then $a_2 < 0$, and, for large $i$, there exist 
\begin{align}
\delta_i 
	&= e^{O(1)} \lambda_i^{-(1 - \frac{\beta_2}{n-2k})},\label{Eq:06III21-delta}\\
\nu_i 
	&= e^{O(1)} \lambda_i^{-(1 - \frac{2\beta_2}{n-2k})}\label{Eq:06III21-nu}
\end{align}
such that $\xi_i$ is strictly increasing in $(\ln\frac{R_i}{\lambda_i},\ln\delta_i)$, is strictly decreasing in $(\ln\delta_i,\ln\nu_i)$, has a strict local maximum at $\ln\delta_i$, and
\begin{align}
\xi_i(\ln\delta_i)
	&
	= \frac{\beta_2}{n-2k} \ln\lambda_i + p_2 + o(1),\label{Eq:21XII20-Est1}
\\
\xi_i(\ln\delta_i) 
	&= \ln\lambda_i + \ln \delta_i + q_2 + o(1),
\label{Eq:21II21-N0}\\
\xi_i(t) 
	&= \ln \lambda_i + t + O(1)   \text{ in } (\ln\frac{2}{\lambda_i}, \ln \delta_i),\label{Eq:21XII20-Est2}\\
\xi_i(t)
	&= - (1- \frac{2\beta_2}{n-2k})\ln \lambda_i  - t + O(1) \text{ in } (\ln \delta_i, \ln \nu_i).\label{Eq:21XII20-Est3}
\end{align}
where
\begin{align*}
p_2 
	&:=  - \frac{1}{n-2k} \ln\Big[ 2^{\beta_2 + \frac{n+2k}{2}}\binom{n}{k}^{\frac{n-2k}{2k}} \frac{\Gamma(\frac{n-\beta_2}{2})\Gamma(\frac{n+\beta_2}{2})}{2\Gamma(n)}  |a_2|  \beta_2  K(S)^{-\frac{n}{2k}}\Big],\\
q_2 
	&:= - \ln \Big[2^{\frac{n+2k}{2(n-2k)}}\binom{n}{k}^{\frac{1}{2k}}  K(S)^{-\frac{1}{2k}}\Big].
\end{align*}

If $\beta < n-2k$, the negativity of $a_2$, and estimates \eqref{Eq:06III21-delta}, \eqref{Eq:06III21-nu}, \eqref{Eq:21XII20-Est2} and \eqref{Eq:21XII20-Est3} follow from Theorem \ref{Prop:LocEst}(b) and (c). Using the fact that  $\xi_i$ is now defined on all of $\RR$ (rather than $(-\infty,\ln 2)$ in Theorem \ref{Prop:LocEst}), the same proof can be used to treat the case $n-2k \leq \beta < n$. Estimate \eqref{Eq:21XII20-Est1} will be obtained by using $\dot\xi_i(\ln\delta_i) = 0$ in the relevant Pohozaev identity. Estimate \eqref{Eq:21II21-N0} will be proved using a mass-type identity. Let us now give the details.

\medskip
\noindent\underline{Proof of \eqref{Eq:21XII20-Est2}.}

By Theorem \ref{Prop:LocEst}(a), for every $\varepsilon_i \rightarrow 0^+$ and every $R_i \rightarrow \infty$, after passing to a subsequence, we have for $0 \leq \ell \leq 2$ that
\begin{equation}
\Big|\frac{d^\ell}{dt^\ell}\Big[\xi_i(t) + \ln \frac{\lambda_i e^t}{1 + \lambda_i^2 e^{2t}} + \ln \Big(2^{\frac{1}{2}}\binom{n}{k}^{\frac{1}{2k}} K(S)^{-\frac{1}{2k}}\Big)\Big]\Big| \leq \varepsilon_i \lambda_i^\ell e^{\ell t} \text{ in } (-\infty, \ln\frac{R_i}{\lambda_i}).
	\label{Eq:19XI20-EstLCyl}
\end{equation}

Note that by \eqref{Eq:19XI20-EstLCyl}, $\dot\xi_i(t) \geq 0$ in $(\ln \frac{2}{\lambda_i}, \ln\frac{R_i}{\lambda_i})$. Let
\[
\delta_i = \sup\Big\{ s \geq \frac{2}{\lambda_i}: \dot\xi_i \geq 0 \text{ in } (\ln \frac{2}{\lambda_i}, \ln s)\Big\} \in [R_i\lambda_i^{-1},\infty].
\]
Clearly, if $\delta_i$ is finite, $\dot\xi_i(\ln\delta_i) = 0$. Furthermore, we have for $\ln\frac{R_i}{\lambda_i} < t < \ln\delta_i$ that
\begin{equation}
\xi_i(t)
	\geq \xi_i(\ln(R_i/\lambda_i))  \stackrel{\eqref{Eq:19XI20-EstLCyl}}{\geq} \ln R_i - O(1).
		\label{Eq:19XI20-E5}
\end{equation}

It follows from property \eqref{Eq:12III21-LMCrit} that $\ddot \xi_i < 0$ at every critical point of $\xi_i$ in $[\ln\frac{R_i}{\lambda_i}, \ln\delta_i]$ for large $i$. In particular, for large $i$, $\xi_i$ is strictly increasing in $[\ln\frac{R_i}{\lambda_i}, \ln\delta_i)$ and, if $\delta_i$ is finite, then as $\dot\xi_i(\ln\delta_i) = 0$, $\ddot\xi_i(\ln\delta_i) < 0$, and $\ln\delta_i$ is a strict local maximum of $\xi_i$.

Estimate \eqref{Eq:21XII20-Est2} follows from \eqref{Eq:19XI20-E5}, the monotonicity of $\xi_i$ and Lemma \ref{Lem:06XII20-FOCyl}.

\medskip
\noindent\underline{Proof of the negativity of $a_2$ and estimates \eqref{Eq:06III21-delta} and \eqref{Eq:21XII20-Est1}.}

As in the proof of Theorem \ref{Prop:LocEst} (see \eqref{Eq:30IV21-B3}), we have for $t \in (\ln\frac{R_i}{\lambda_i}, \ln \delta_i)$ that
\begin{align*}
&\Big[1 - \dot \xi_i(t)^2\Big]^k + o(1)\\
	&\quad= -(1 + o(1)) e^{(n-2k)\xi_i(t)} \lambda_i^{-\beta_2}  a_2\beta_2 2^{\beta_2 + \frac{n + 2k}{2}}\binom{n}{k}^{\frac{n-2k}{2k}} K(S)^{-\frac{n}{2k}}\int_0^\infty  \frac{r^{n + \beta_2 - 1}}{(1+r^2)^{n}}\,dr.
\end{align*}
Using Corollary \ref{Cor:IntegralId}, we get
\begin{equation}
\Big[1 - \dot \xi_i(t)^2\Big]^k + o(1)
	= -(1 + o(1)) e^{(n-2k)\xi_i(t)} \lambda_i^{-\beta_2} e^{-(n-2k)p_2}.
	\label{Eq:06III21-X1}
\end{equation}
Since the right hand side of \eqref{Eq:06III21-X1} is $- e^{O(1)}  a_2   e^{(n-2k) t}\lambda_i^{n-2k -\beta_2}$ (by \eqref{Eq:21XII20-Est2}) and \eqref{Eq:06III21-X1} holds for all $t \in (\ln \frac{R_i}{\lambda_i},\ln\delta_i)$,  it follows that $\delta_i$ is finite and, in view of the definition of $\delta_i$, $\dot\xi_i(\ln\delta_i) = 0$. In particular, we can also take $t = \ln \delta_i$ in \eqref{Eq:06III21-X1}, yielding the assertion $a_2 < 0$ and estimates \eqref{Eq:06III21-delta} and \eqref{Eq:21XII20-Est1}.

As $a_2 < 0$, item (i) of the hypotheses of the theorem gives $\beta_2 \geq \frac{n-2k}{2}$. 

\medskip
\noindent\underline{Proof of estimates \eqref{Eq:06III21-nu} and \eqref{Eq:21XII20-Est3}.} 

The proof is similar to the proof of Theorem \ref{Prop:LocEst}(c). We omit the details.

\medskip
\noindent\underline{Proof of estimate \eqref{Eq:21II21-N0}.}

We start by using the mass-type identity \eqref{Eq:21II21-Mass1} and the fact that $\dot\xi_i(\ln\delta_i) = 0$ to obtain
\begin{align}
e^{\frac{n-2k}{2}(-\xi_i(\ln\delta_i)+ \ln\delta_i)}
	&= 2^{k-1} n \binom{n}{k}^{-1} m(\ln\delta_i,\xi_i(\ln\delta_i),\dot\xi_i(\ln\delta_i))\nonumber\\
	&= 2^{k-1} n \binom{n}{k}^{-1} \int_{-\infty}^{\ln \delta_i} K_{cyl}(\tau) (1 - \dot\xi_i)^{-(k-1)} e^{-\frac{n+2k}{2}\xi_i} e^{ \frac{n-2k}{2}\tau}\,d\tau.
	\label{Eq:21II21-N1}
\end{align}

We proceed to estimate the integral on the right hand side of \eqref{Eq:21II21-N1}. The integration over $(-\infty,\ln\frac{R_i}{\lambda_i})$ can be estimated using the continuity of $K$ and \eqref{Eq:19XI20-EstLCyl} with $\varepsilon_i \ll R_i^{-3}$ and Corollary \ref{Cor:IntegralId}:
\begin{align}
&\int_{-\infty}^{\ln \frac{R_i}{\lambda_i}} K_{cyl}(\tau) (1 - \dot\xi_i)^{-(k-1)} e^{-\frac{n+2k}{2}\xi_i} e^{ \frac{n-2k}{2}\tau}\,d\tau\nonumber\\
	&\qquad= (1 + o(1)) 2^{\frac{n-2k + 4}{4}}\binom{n}{k}^{\frac{n+2k}{4k}}  K(S)^{-\frac{n-2k}{4k}} \lambda_i^{-\frac{n-2k}{2}} \int_0^\infty \frac{r^{n-1}}{(1 + r^2)^{\frac{n+2}{2}}}\,dr\nonumber\\
	&\qquad= (1 + o(1)) \frac{1}{n} 2^{\frac{n-2k+4}{4}}\binom{n}{k}^{\frac{n+2k}{4k}}  K(S)^{-\frac{n-2k}{4k}} \lambda_i^{-\frac{n-2k}{2}}.
	\label{Eq:21II21-N2}
\end{align}

To estimate the integration over $(\ln\frac{R_i}{\lambda_i},\ln\delta_i)$, we need to bound $(1 - \dot \xi_i)^{-(k-1)}$. Recall from Step 2 that $\dot\xi_i > 0$ in $(\ln\frac{R_i}{\lambda_i},\ln\delta_i)$. Let $X_i = e^{2\xi_i}(1 - \dot \xi_i^2) > 0$, which is, up to a harmless multiplicative constant, the repeated eigenvalue of the Schouten tensor of $g_{v_i}$. Note that \eqref{Eq:11XI20-E1} can be recast as
\[
X_i^{k-1} e^{2\xi_i}\ddot\xi_i + \frac{n-2k}{2k} X_i^k = 2^{k-1} \binom{n-1}{k-1}^{-1}K_{cyl}.
\]
Thus, there exists a small $\chi_0 > 0$ depending only on $n$, $k$ and a positive lower bound for $K$ such that $\ddot\xi_i(t) \geq 0$ whenever $X_i(t) < \chi_0$. As $\dot X_i = -2\dot \xi_i( e^{2\xi_i} \ddot \xi_i + X_i)$ and $\dot \xi_i > 0$ in $(\ln \frac{R_i}{\lambda_i},\ln\delta_i)$, this implies that $\dot X_i(t) \leq 0$ whenever $X_i(t) < \chi_0$ for $t \in (\ln \frac{R_i}{\lambda_i},\ln\delta_i)$. On the other hand, since $X_i(\ln \delta_i) = e^{O(1)} \lambda_i^{\frac{2\beta_2}{n-2k}} > \chi_0$ (in view of \eqref{Eq:21XII20-Est1} and $\dot\xi_i(\ln\delta_i) = 0$), we deduce that $X_i$ is nowhere less than $\chi_0$ in $(\ln \frac{R_i}{\lambda_i},\ln\delta_i)$, i.e.
\[
X_i \geq \chi_0 \text{ in } (\ln \frac{R_i}{\lambda_i},\ln\delta_i).
\]
It follows that $1 - \dot\xi_i  = \frac{X_i e^{-2\xi_i}}{1 + \dot\xi_i}\geq \frac{\chi_0}{2} e^{-2\xi_i}$ in $(\ln \frac{R_i}{\lambda_i},\ln\delta_i)$, and so, in view of \eqref{Eq:21XII20-Est2}, 
\begin{align}
&\int_{\ln \frac{R_i}{\lambda_i}}^{\ln\delta_i} K_{cyl}(\tau) (1 - \dot\xi_i)^{-(k-1)} e^{-\frac{n+2k}{2}\xi_i} e^{ \frac{n-2k}{2}\tau}\,d\tau\nonumber\\
	&\qquad\leq C \int_{\ln \frac{R_i}{\lambda_i}}^{\ln\delta_i} e^{-\frac{n-2k+4}{2}\xi_i} e^{ \frac{n-2k}{2}\tau}\,d\tau
		\leq CR_i^{-2} \lambda_i^{-\frac{n-2k}{2}}.
	\label{Eq:21II21-N4}
\end{align}

Putting \eqref{Eq:21II21-N2} and \eqref{Eq:21II21-N4} into \eqref{Eq:21II21-N1} we obtain \eqref{Eq:21II21-N0}, which concludes Step 2.

\medskip
\noindent\underline{Step 3:} We draw a contradiction.

By \eqref{Eq:17III21-E1}, we may assume without loss of generality that $v_i(S) \rightarrow \infty$. By Step 2 and point (i) of the hypotheses, we have that $\beta_2 \geq \frac{n-2k}{2}$. We consider the cases $\beta_2 \geq n -2k$ and $\frac{n-2k}{2} \leq \beta_2 < n-2k$ separately.

\medskip
\noindent
\underline{Case (a):} $\beta_2 \geq n - 2k$. We will show that $a_1 < 0$, $\frac{1}{\beta_1} + \frac{1}{\beta_2} = \frac{2}{n-2k}$ and that \eqref{Eq:BalCond} is violated, which amounts to a contradiction to our hypotheses.

We first prove that $v_i(N) \rightarrow \infty$. Indeed, by \eqref{Eq:21XII20-Est3}, the oscillation of $\xi_i(t) - t$ in $[0,\infty)$ tends to infinity as $i \rightarrow\infty$. This gives $\mathop{\textrm{osc}}_{\bar\SSphere^n_+} \ln v_i \rightarrow \infty$. Now, if $v_i(N)$ was bounded, we would have by \eqref{Eq:17III21-E0} that $v_i$ is uniformly bounded away from the south pole, and hence, by the Harnack estimate, $|\nabla \ln v_i| \leq C$ on $\bar\SSphere^n_+$, which is a contradiction to the above estimate on the oscillation of $\ln v_i$.

The rough idea of the proof is as follows: Let $\lambda_i = 2^{-\frac{k}{2}+1}\binom{n}{k}^{-1/2}K(S)^{\frac{1}{2}} v_i(S)^{\frac{2}{n-2}} \rightarrow \infty$ and  $\tilde\lambda_i = 2^{-\frac{k}{2}+1}\binom{n}{k}^{-1/2}K(N)^{\frac{1}{2}} v_i(N)^{\frac{2}{n-2}} \rightarrow \infty$. We apply Step 2 to both the north and the south poles to obtain that $\xi_i$ has exactly three critical points, is decreasing in $(-\infty,-\ln\lambda_i + o(1))$, increasing in $(-\ln\lambda_i + o(1),\ln\delta_i)$, decreasing in $(\ln\delta_i, \ln\tilde\lambda_i + o(1))$ and increasing in $(\ln\tilde\lambda_i + o(1),\infty)$, and that $\frac{1}{\beta_1} + \frac{1}{\beta_2} = \frac{2}{n-2k}$. We then show that the $4$-vector $V_i = (\ln\lambda_i, \ln\tilde\lambda_i, \xi_i(\ln\delta_i), \ln\delta_i)^T$ satisfies a linear equation of the form $MV_i = P + o(1)$ where the $4\times 4$-matrix $M$ and the $4$-vector $P$ are independent of $i$. It follows that $P$ is orthogonal to the kernel of $M^T$, which gives $C_{(1)}C_{(2)} = 1$ where $C_{(1)} = C_{n,k}(\beta_1,a_1,K(N))$ and $C_{(2)} = C_{n,k}(\beta_2,a_2,K(S))$.

Let us now give the details. Applying Step 2 to $S$, we have $a_2 < 0$ and there exist $\delta_i$ and $\nu_i$ satisfying \eqref{Eq:06III21-delta} and \eqref{Eq:06III21-nu}
such that $\xi_i$ is strictly increasing in $(\ln\frac{R_i}{\lambda_i},\ln\delta_i)$, is strictly decreasing in $(\ln\delta_i,\ln\nu_i)$, has a strict local maximum at $\ln\delta_i$, and \eqref{Eq:21XII20-Est1}--\eqref{Eq:21XII20-Est3} hold. Applying Step 2 to $N$, we have that $a_1 < 0$, $\beta_1 \geq \frac{n-2k}{2}$, and there exist
\begin{align}
\tilde\delta_i &= e^{O(1)} \tilde\lambda_i^{-(1 - \frac{\beta_1}{n-2k})},\label{Eq:06III21-deltaTilde}\\
\tilde\nu_i &= e^{O(1)}\tilde \lambda_i^{-(1 - \frac{2\beta_1}{n-2k})}\label{Eq:06III21-nuTilde}
\end{align}
 such that $\xi_i$ is strictly decreasing in $(-\ln\tilde\delta_i, -\ln \frac{R_i}{\tilde\lambda_i})$, strictly increasing in $(-\ln \tilde\nu_i, -\ln\tilde\delta_i)$, has a strict local maximum at $-\ln\tilde\delta_i$, 
\begin{align}
\xi_i(-\ln\tilde\delta_i)
	&=  \frac{\beta_1}{n-2k} \ln\tilde\lambda_i + p_1 + o(1),\label{Eq:21XII20-Est1Tilde}\\
\xi_i(-\ln\tilde\delta_i)
	&= \ln\tilde\lambda_i + \ln\tilde\delta_i + q_1 + o(1) ,	
	\label{Eq:21II21-N0Tilde}\\
\xi_i(t)
	&=  \ln\tilde\lambda_i - t + O(1) \text{ in } (-\ln\tilde\delta_i, \ln \frac{\tilde\lambda_i}{2}),
	\label{Eq:17II21-Est2}\\
\xi_i(t)
	&= - (1- \frac{2\beta_1}{n-2k})\ln \tilde\lambda_i  + t + O(1) \text{ in } (-\ln \tilde\nu_i, -\ln\tilde\delta_i),
	\label{Eq:17II21-Est3}
\end{align}
where
\begin{align*}
p_1
	& :=  - \frac{1}{n-2k} \ln\Big[ 2^{\beta_1 + \frac{n+2k}{2}}\binom{n}{k}^{\frac{n-2k}{2k}} \frac{\Gamma(\frac{n-\beta_1}{2})\Gamma(\frac{n+\beta_1}{2})}{2\Gamma(n)}  |a_1|  \beta_1  K(N)^{-\frac{n}{2k}}\Big],\\
q_1
	&:=- \ln\Big[2^{\frac{n+2k}{2(n-2k)}}\binom{n}{k}^{\frac{1}{2k}}  K(N)^{-\frac{1}{2k}}\Big].
\end{align*}

Comparing the value of $\xi_i(0)$ from \eqref{Eq:21XII20-Est2} and \eqref{Eq:17II21-Est3}, we have
\begin{equation}
\lambda_i = e^{O(1)} \tilde\lambda_i^{-1 +  \frac{2\beta_1}{n-2k}}.
\label{Eq:20I21-A4}
\end{equation}
This implies that $\beta_1 > \frac{n-2k}{2}$.

Note that, by \eqref{Eq:19XI20-EstLCyl} and the definition of $\delta_i$ and $\nu_i$, $\xi_i$ is strictly decreasing in $(-\infty, \ln\frac{1+o(1)}{\lambda_i})$, strictly increasing in $(\ln\frac{1+o(1)}{\lambda_i},\ln\delta_i)$, strictly decreasing in $(\ln\delta_i,\ln\nu_i)$, and has exactly two critical points in $(-\infty,\ln\nu_i)$ at $\ln\frac{1+o(1)}{\lambda_i}$ and $\ln\delta_i$. Now, since $\beta_2 \geq n-2k$ and $\beta_1 > \frac{n-2k}{2}$, we have by \eqref{Eq:06III21-delta} and \eqref{Eq:06III21-nuTilde} that 
\[
-\ln\tilde\nu_i = -(1 - \frac{2\beta_1}{n-2k})\ln\tilde \lambda_i \ll O(1) \leq \ln\delta_i  = -(1 - \frac{\beta_2}{n-2k}) \ln\lambda_i + O(1).
\] 
Since $\xi_i$ is strictly decreasing in $(-\ln\tilde\nu_i,-\ln\tilde\delta_i)$ and $\dot\xi_i(-\ln\tilde\delta_i) = \dot\xi_i(\ln\delta_i) = 0$, we have that $\ln\delta_i = -\ln\tilde\delta_i$, which implies (in view of \eqref{Eq:06III21-delta} and \eqref{Eq:06III21-deltaTilde})
\begin{equation}
\lambda_i^{-(1 - \frac{\beta_2}{n-2k})} = e^{O(1)} \tilde\lambda_i^{(1 - \frac{\beta_1}{n-2k})}.
	\label{Eq:18II21-B1}
\end{equation}
Substituting \eqref{Eq:20I21-A4} into \eqref{Eq:18II21-B1}, we obtain that
\begin{equation}
\frac{1}{\beta_1} + \frac{1}{\beta_2} = \frac{2}{n-2k}.
	\label{Eq:18II21-B2}
\end{equation}

Now, let $V_i = (\ln\lambda_i, \ln\tilde\lambda_i, \xi_i(\ln\delta_i), \ln\delta_i)^T$ and observe that \eqref{Eq:21XII20-Est1}, \eqref{Eq:21II21-N0}, \eqref{Eq:21XII20-Est1Tilde} and \eqref{Eq:21II21-N0Tilde} give a linear system of the form 
\[
MV_i = P + o(1), \text{ where }M = \begin{pmatrix} -\frac{\beta_2}{n-2k} & 0 & 1 & 0\\
-1 & 0 & 1 & -1\\
0 & - \frac{\beta_1}{n-2k} & 1& 0\\
0 & -1 & 1 & 1
\end{pmatrix}
\text{ and } P = \begin{pmatrix}p_2 \\q_2 \\p_1 \\ q_1\end{pmatrix}.
\]
A straightforward computation gives that $\det M = \frac{\beta_1\beta_2}{n-2k} \Big(-\frac{2}{n-2k} + \frac{1}{\beta_1} + \frac{1}{\beta_2}\Big) = 0$, and the kernel of $M^T$ is generated by $W_0 := (\frac{n-2k}{\beta_2},-1,\frac{n-2k}{\beta_1}, -1)^T$. The fact that $MV_i = P + o(1)$ implies that $P \cdot W_0 = 0$, i.e.
\[
\frac{n-2k}{\beta_2} p_2 - q_2 + \frac{n-2k}{\beta_1} p_1 -   q_1  = 0.
\]
Recalling the expression of $p_1, p_2, q_1, q_2$, we see that this is equivalent to $C_{(1)}C_{(2)} = 1$. However, since \eqref{Eq:18II21-B2} holds and $a_1, a_2 < 0$, we have by our hypotheses that \eqref{Eq:BalCond} holds, which is contradiction to the above identity. This finishes the proof when $\beta_2 \geq n-2k$.

\medskip
\noindent
\underline{Case (b):} $\frac{n-2k}{2} \leq \beta_2 < n - 2k$.

Take a point $p$ on the equator of $\SSphere^n$. Recall that $v_i(S) \rightarrow \infty$. By Step 2, we know that $v_i(p) \rightarrow 0$. Let $\check v_i = \frac{1}{v_i(p)} v_i$.  By the first and second derivatives estimates \eqref{Eq:26XI20-Est2}, after passing to a subsequence if necessary, we may assume that $\check v_i$ converges in $C^{1,\alpha}_{loc}(\SSphere^n \setminus \{S,N\})$ to some positive function $\check v_\infty \in C^{1,1}_{loc}(\SSphere^n \setminus \{S,N\})$ which satisfies
\begin{equation}
\lambda(A_{g_{\check v_\infty}})  \in \partial\Gamma_k \text{ in } \SSphere^n \setminus \{S,N\}
	\label{Eq:18III21-B0}
\end{equation}
in the viscosity sense. Note that as $\frac{n-2k}{2} \leq \beta_2 < n - 2k$, we have in Step 2 that $\delta_i \rightarrow 0$ and $\nu_i \geq \frac{1}{C}$. Hence, by estimate \eqref{Eq:21XII20-Est3} in Step 2, there exists  $r_i = O(\delta_i) \rightarrow 0$ such that $\frac{1}{C} \leq \check v_i \leq C$ in $\{x: r_i \leq d_{\ringg}(x,S) \leq \pi/2\}$. It follows that
\begin{equation}
\frac{1}{C} \leq \check v_\infty \leq C \text{ near $S$}.
	\label{Eq:18III21-B1}
\end{equation}

We proceed according to whether $v_i(N)$ is bounded or not. Suppose first that $v_i(N)$ is bounded. Then $\sup_{\SSphere^n_+} v_i$ is also bounded (see \eqref{Eq:17III21-E0} and the sentence following it). The estimates in Step 1 are thus improved to
\[
v_i(x) d_{\ringg}(x,\{S\})^{\frac{n-2}{2}} \leq C \text{ and } |\nabla^\ell \ln v_i(x)| d_{\ringg}(x,\{S\})^{\ell} \leq C \text{ for } \ell = 1, 2.
\]
It follows that the function $\check v_\infty$ satisfies
\[
\lambda(A_{g_{\check v_\infty}})  \in \partial\Gamma_k \text{ in } \SSphere^n \setminus \{S\}.
\]
In view of the Liouville-type theorem \cite[Theorem 1.3]{Li07-ARMA}, this is impossible: No such $\check v_\infty$ can satisfies \eqref{Eq:18III21-B1}.

Finally, consider the case that $N$ is a blow-up point. In view of Case (a) above, by exchanging the role of the north pole and the south pole, we may assume that $\frac{n-2k}{2} \leq \beta_1 < n-2k$. The proof of \eqref{Eq:18III21-B1} also applies near $N$ giving that
\[
\frac{1}{C} \leq \check v_\infty \leq C \text{ in } \SSphere^n \setminus \{S,N\}.
\]
By the classification result \cite[Theorem 1.6]{LiNgBocher}, no axisymmetric solution $\check v_\infty$ to \eqref{Eq:18III21-B0} satisfies the above inequality. This finishes the proof of Theorem \ref{main}.
\end{proof}

The following remark is easily seen from the above proof:

\begin{remark}\label{Rem:19III21-R1}
If $\max(a_{1},a_{2}) > 0$, the constant $C_1$ in Theorem \ref{main} depends only on an upper bound of  $|a_{1}|, |a_{2}|$, $|a_{1}|^{-1}, |a_{2}|^{-1}$, $(n- \beta_{1})^{-1}, (n- \beta_{2})^{-1}$, $\|\ln K\|_{C^{2,\alpha}_r(\SSphere^n)}$, and a non-negative function $\phi: [0,\pi/2)\rightarrow [0,\infty)$ such that $\phi(\theta)\rightarrow 0$ as $\theta \rightarrow 0$ and
\[
 \frac{|R_1(\theta)| + |\theta||R_1'(\theta)|}{|\theta|^{\beta_{1}}} \leq \phi(\theta) \text{ and }  \frac{|R_2(\theta)| + |\pi - \theta||R_2'(\theta)|}{|\pi - \theta|^{\beta_{2}}} \leq \phi(\pi-\theta).
\]

If $\frac{1}{\beta_{1}} + \frac{1}{\beta_{2}} \neq \frac{2}{n-2k}$ and $a_1, a_2 < 0$, the constant $C_1$ depends only on an upper bound of  $|a_{1}|, |a_{2}|, |a_{1}|^{-1}, |a_{2}|^{-1}$, $(n- \beta_{1})^{-1}, (n- \beta_{2})^{-1}$, $\|\ln K\|_{C^{2,\alpha}_r(\SSphere^n)}$, $|\frac{1}{\beta_{1}} + \frac{1}{\beta_{2}} - \frac{2}{n-2k}|^{-1}$, and a function $\phi$ as above.

If $\frac{1}{\beta_{1}} + \frac{1}{\beta_{2}} = \frac{2}{n-2k}$ and $a_1, a_2 < 0$, the constant $C_1$ depends only on an upper bound of  $|a_{1}|, |a_{2}|$, $|a_{1}|^{-1}, |a_{2}|^{-1}$, $(n- \beta_{1})^{-1}, (n- \beta_{2})^{-1}$, $\|\ln K\|_{C^{2,\alpha}_r(\SSphere^n)}$, $|C_{(1)}C_{(2)} - 1|^{-1}$, and a function $\phi$ as above.

 \end{remark}

\subsection{Some extensions of Theorem \ref{main}}\label{SSec:CptExt}

In many situations, we will often consider \eqref{Eq:29X20-NP} in a family of equations of the form
\begin{equation}
\sigma_k(\lambda(A_{g_v})) = K_\mu \text{ and } \lambda(A_{g_v}) \in \Gamma_k \text{ on } \SSphere^n
	\label{Eq:29X20-NPmu}
\end{equation}
where $K_\mu$ depends on a certain parameter $\mu$ in some index set $I$. Analogous to \eqref{K condition}, we will assume that there exist $a_{1,\mu}, a_{2\mu}\neq 0$ and $2 \leq \beta_{1,\mu}, \beta_{2,\mu} < n$ such that if we write
\begin{align*}
K_\mu(\theta) 
	=K_\mu(0)+a_{1,\mu} \theta^{\beta_{1,\mu}}+R_{1,\mu}(\theta) 
	=K_\mu(\pi)+a_{2,\mu}(\pi-\theta)^{\beta_{2,\mu}}+R_{2,\mu}(\theta)
\end{align*}
then
\begin{equation}
\lim_{\theta \rightarrow 0} \sup_{\mu \in I} \frac{|R_{1,\mu}(\theta)| + |\theta||R_{1,\mu}'(\theta)|}{|\theta|^{\beta_{1,\mu}}} = \lim_{\theta \rightarrow \pi} \sup_{\mu \in I} \frac{|R_{2,\mu} (\theta)| + |\pi - \theta||R_{2,\mu}'(\theta)|}{|\pi - \theta|^{\beta_{2,\mu}}}  = 0.
	\label{K condition mu}
\end{equation}

\begin{remark}\label{Rem:27II21-R1}
It is not hard to see from the proof of Theorem \ref{main} that if each $K_\mu$ satisfies the hypotheses of Theorem \ref{main}, $\frac{1}{C} \leq |a_{1,\mu}|, |a_{2,\mu}| \leq C$, $|n - \beta_{1,\mu}|, |n - \beta_{2,\mu}| \geq \frac{1}{C}$, $|\frac{1}{\beta_{1,\mu}} + \frac{1}{\beta_{2,\mu}} - \frac{2}{n-2k}| \geq \frac{1}{C}$, and $\|K_\mu\|_{C^{2,\alpha}_r(\SSphere^n)} \leq C$ for some constant $C$, then there exists a constant $C_1 > 0$ such that all positive solutions to \eqref{Eq:29X20-NPmu} with $\mu \in I$ satisfy
\[
 \|\ln v\|_{C^{4,\alpha}(\SSphere^{n})}<C_{1}.
\]
Furthermore, if $\max(a_{1,\mu},a_{2,\mu}) \geq \frac{1}{C}$, the assumption that $|\frac{1}{\beta_{1,\mu}} + \frac{1}{\beta_{2,\mu}} - \frac{2}{n-2k}| \geq \frac{1}{C}$ can be dropped.
 \end{remark}
 
To prove Theorem \ref{Thm:DegForm} later on, we embed $K$ in a family $\{K_\mu\}$ in two specific ways for which Remarks \ref{Rem:19III21-R1} and \ref{Rem:27II21-R1} do not apply. Let us now show how the proof of Theorem \ref{main} can be adapted to cover those situations.

\begin{theorem}\label{Thm:19XI20-M1}
Assume that $n\geq 5$, $2\leq k < n/2$, $0 < \alpha < 1$, $K  \in C_{r}^{2,\alpha}(\SSphere^{n})$ is positive and satisfies \eqref{K condition} for some $a_1, a_2 \neq 0$ and $2 \leq \beta_1,\beta_2 < n$. Assume further that $a_i > 0$ if $\beta_i < \frac{n-2k}{2}$ for some $i$, and $\max(a_1,a_2) > 0$ if  $\frac{1}{\beta_1} + \frac{1}{\beta_2} \geq \frac{2}{n-2k}$. For $\mu \in (0,1]$, let $K_\mu = \mu K + (1-\mu) 2^{-k} \binom{n}{k}$. Then there exists some positive constant $C_{1}$ such that all $C_{r}^{2}(\SSphere^{n})$ positive solutions to \eqref{Eq:29X20-NPmu} with $0 < \mu \leq 1$ satisfy
\[
 \|\ln v\|_{C^{4,\alpha}(\SSphere^{n})}<C_{1}.
\]
\end{theorem}

\begin{proof} The proof is almost identical to that of Theorem \ref{main}. We will only indicate the necessary changes. We suppose by contradiction that there exists $\mu_i \in (0,1]$ and positive functions $v_i \in C_{r}^{2}(\SSphere^{n})$ satisfying \eqref{Eq:29X20-NPmu} with $\mu = \mu_i$ such that
\[
\max\{ v_i(N),v_i(S)\} \rightarrow \infty.
\]

There is no change to Step 1.

Step 2 is modified as follows: One shows that if $v_i(S) \rightarrow \infty$, then $a_2 < 0$, $\beta_2 \geq \frac{n-2k}{2}$ and, for large $i$, there exist $\delta_i = e^{O(1)} \mu_i^{-\frac{1}{n-2k}}\lambda_i^{-(1 - \frac{\beta_2}{n-2k})}$ and $\nu_i = e^{O(1)}\mu_i^{-\frac{2}{n-2k}} \lambda_i^{-(1 - \frac{2\beta_2}{n-2k})}$ such that $\xi_i$ is strictly increasing in $(\ln\frac{R_i}{\lambda_i},\ln\delta_i)$, is strictly decreasing in $(\ln\delta_i,\ln\nu_i)$, has a strict local maximum at $\ln\delta_i$, and
\begin{align*}
&\xi_i(\ln\delta_i)
	= \frac{\beta_2}{n-2k}\ln \lambda_i - \frac{1}{n-2k}\ln\mu_i + p_2 + o(1),
\\
&\xi_i(t) 
	= \ln \lambda_i + t + O(1)   \text{ in } (\ln\frac{2}{\lambda_i}, \ln \delta_i),\\
&\xi_i(t)
	= - \frac{2}{n-2k} \ln \mu_i - (1- \frac{2\beta_2}{n-2k})\ln \lambda_i  - t + O(1) \text{ in } (\ln \delta_i, \ln \nu_i).
\end{align*}
The appearance of $\mu_i$ in the above is due to the fact that, in the present case, one needs to include a multiplicative factor of $\mu_i$ on the right hand side of \eqref{Eq:06III21-X1}.
	
Step 3 is modified as follows.  If $\frac{n-2k}{2} \leq \beta_2 < n-2k$, the proof remains unchanged. If $\beta_2 \geq n-2k$, one still has that $v_i(N) \rightarrow \infty$ and $v_i(S) \rightarrow 0$, $a_1, a_2 < 0$, $\beta_1 \geq \frac{n-2k}{2}$, $\delta_i = \tilde \delta_i^{-1}$ and $\lambda_i = e^{O(1)} \tilde \lambda_i^{-(1 - \frac{2\beta_1}{n-2k})}$. Recalling that $\delta_i = e^{O(1)} \mu_i^{-\frac{1}{n-2k}}\lambda_i^{-(1 - \frac{\beta_2}{n-2k})}$ and $\tilde\delta_i = e^{O(1)} \mu_i^{-\frac{1}{n-2k}}\tilde\lambda_i^{-(1 - \frac{\beta_1}{n-2k})}$, one obtains that
\[
\tilde\lambda_i^{-\frac{\beta_1 + \beta_2}{n-2k} + \frac{2\beta_1\beta_2}{n-2k}} = e^{O(1)}\mu_i^{\frac{2}{n-2}} \leq C.
\]
This implies $\frac{1}{\beta_1} + \frac{1}{\beta_2} \geq \frac{2}{n-2k}$. By our hypotheses on the sign of $a_1$ and $a_2$, we thus have $\max(a_1,a_2) > 0$, contradicting the earlier conclusion that $a_1$ and $a_2$ are both negative.
\end{proof}

\begin{theorem}\label{Thm:27II21-T1}
Assume that $n\geq 5$, $2\leq k < n/2$, $0 < \alpha < 1$, $0 < \varepsilon_0 < 1$, $\{K_{\mu}\}$ is a bounded sequence of positive functions in $ C_{r}^{2,\alpha}(\SSphere^{n})$ which satisfies \eqref{K condition mu} for some $- \varepsilon_0^{-1} < a_{1,\mu}, a_{2,\mu} < - \varepsilon_0 < 0$ and $\frac{n-2k}{2} \leq \beta_{1,\mu},\beta_{2,\mu} \leq n - \varepsilon_0$, $\frac{1}{\beta_{1,\mu}} + \frac{1}{\beta_{2,\mu}} \rightarrow \frac{2}{n-2k}$. Let $C_{(1),\mu} = C_{n,k}(\beta_{1,\mu},a_{1,\mu},K_{\mu}(0))$ and $C_{(2),\mu} =  C_{n,k}(\beta_{2,\mu},a_{2,\mu},K_{\mu}(\pi))$ and assume further that either
\begin{enumerate}[(i)]
\item $\frac{1}{\beta_{1,\mu}} + \frac{1}{\beta_{2,\mu}} \geq \frac{2}{n-2k}$ and $C_{(1),\mu}C_{(2),\mu} < 1 - \varepsilon_0$,
\end{enumerate}
or
\begin{enumerate}[(i)]
\setcounter{enumi}{1}
\item $\frac{1}{\beta_{1,\mu}} + \frac{1}{\beta_{2,\mu}} \leq \frac{2}{n-2k}$ and $C_{(1),\mu}C_{(2),\mu} > 1 + \varepsilon_0$.
\end{enumerate}

Then there exists a constant $C_{1} > 0$ such that all $C_{r}^{2}(\SSphere^{n})$ positive solutions to \eqref{Eq:29X20-NP} satisfy
\[
 \|\ln v\|_{C^{4,\alpha}(\SSphere^{n})}<C_{1}.
\]
\end{theorem}

\begin{proof} We amend the proof of Theorem \ref{main}, and we will indicate only the necessary changes. We suppose by contradiction that the conclusion fails. By passing to a subsequence, we may assume that there exist $\mu_i \rightarrow \infty$ and positive functions $v_i \in C_{r}^{2}(\SSphere^{n})$ satisfying \eqref{Eq:29X20-NPmu} with $\mu = \mu_i$ such that
\[
\max\{ v_i(N),v_i(S)\} \rightarrow \infty.
\]

Let $\gamma_i = \frac{n-2k}{\beta_{1,\mu_i}} + \frac{n-2k}{\beta_{2,\mu_i}} - 2$. Passing again to a subsequence, we may further assume that $\gamma_i \geq 0$ for all $i$ or $\gamma_i \leq 0$ for all $i$.

Steps 1 and 2 remain unchanged. In Step 3, we again have that both the north and the south poles are blow-up points, $\delta_i = \tilde\delta_i^{-1}$, and 
\begin{align}
\xi_i(\ln\delta_i)
	&
	= \frac{\beta_{2,\mu_i}}{n-2k} \ln\lambda_i + p_{2,\mu_i} + o(1),\label{Eq:21XII20-Est1mu}
\\
\xi_i(\ln\delta_i) 
	&= \ln\lambda_i + \ln \delta_i + q_{2,\mu} + o(1),
\label{Eq:21II21-N0mu}\\
\xi_i(-\ln\tilde\delta_i)
	&
	= \frac{\beta_{1,\mu_i}}{n-2k} \ln\tilde\lambda_i + p_{1,\mu_i} + o(1),\label{Eq:21XII20-Est1muTilde}
\\
\xi_i(-\ln\tilde\delta_i) 
	&= \ln\tilde\lambda_i + \ln \tilde\delta_i + q_{1,\mu_i} + o(1),
\label{Eq:21II21-N0muTilde}
\end{align}
where
\begin{align*}
p_{2,\mu}
	&:=  - \frac{1}{n-2k} \ln\Big[ 2^{\beta_{2,\mu} + \frac{n+2k}{2}}\binom{n}{k}^{\frac{n-2k}{2k}} \frac{\Gamma(\frac{n-\beta_{2,\mu}}{2})\Gamma(\frac{n+\beta_{2,\mu}}{2})}{2\Gamma(n)}  |a_{2,\mu}|  \beta_{2,\mu}  K_\mu(S)^{-\frac{n}{2k}}\Big],\\
q_{2,\mu} 
	&:= - \ln \Big[2^{\frac{n+2k}{2(n-2k)}}\binom{n}{k}^{\frac{1}{2k}}  K_\mu(S)^{-\frac{1}{2k}}\Big],\\
p_{1,\mu}
	&:=  - \frac{1}{n-2k} \ln\Big[ 2^{\beta_{1,\mu} + \frac{n+2k}{2}}\binom{n}{k}^{\frac{n-2k}{2k}} \frac{\Gamma(\frac{n-\beta_{1,\mu}}{2})\Gamma(\frac{n+\beta_{1,\mu}}{2})}{2\Gamma(n)}  |a_{1,\mu}|  \beta_{1,\mu}  K_\mu(N)^{-\frac{n}{2k}}\Big],\\
q_{1,\mu} 
	&:= - \ln \Big[2^{\frac{n+2k}{2(n-2k)}}\binom{n}{k}^{\frac{1}{2k}}  K_\mu(N)^{-\frac{1}{2k}}\Big].
\end{align*}

Now, adding \eqref{Eq:21II21-N0mu} and \eqref{Eq:21II21-N0muTilde} gives
\[
2\xi_i(\ln \delta_i) - \ln \lambda_i - \ln \tilde\lambda_i
	= q_{1,\mu_i} + q_{2,\mu_i} + o(1).
\]
Multiplying \eqref{Eq:21XII20-Est1mu} by $\frac{n-2k}{\beta_{2,\mu_i}}$ and \eqref{Eq:21XII20-Est1muTilde} by $\frac{n-2k}{\beta_{1,\mu_i}}$ and adding the resulting identities together give
\[
(2+\gamma_i)\xi_i(\ln \delta_i) - \ln \lambda_i - \ln\tilde\lambda_i
	= \frac{n-2k}{\beta_{1,\mu_i}} p_{1,\mu_i} + \frac{n-2k}{\beta_{2,\mu_i}} p_{2,\mu_i} + o(1).
\]
Recalling that $\xi_i(\ln\delta_i) = \frac{\beta_{2,\mu_i}}{n-2k} \ln\lambda_i \rightarrow \infty$, we thus have in the case of non-negative $\gamma_i$'s that 
\[
0 
	\leq \liminf_{i \rightarrow \infty} \Big[\frac{n-2k}{\beta_{1,\mu_i}} p_{1,\mu_i} + \frac{n-2k}{\beta_{2,\mu_i}} p_{2,\mu_i} - q_{1,\mu_i} - q_{2,\mu_i}\Big] 
	= \liminf_{i \rightarrow \infty} \ln[C_{(1),\mu_i}C_{(2),\mu_i}],
\]
and in the case of non-positive $\gamma_i$'s that
\[
0 
	\geq \limsup_{i \rightarrow \infty} \Big[\frac{n-2k}{\beta_{1,\mu_i}} p_{1,\mu_i} + \frac{n-2k}{\beta_{2,\mu_i}} p_{2,\mu_i} - q_{1,\mu_i} - q_{2,\mu_i}\Big] 
	= \limsup_{i \rightarrow \infty} \ln[C_{(1),\mu_i}C_{(2),\mu_i}].
\]
These contradict our hypotheses.
\end{proof}

\subsection{Second proof of Theorem \ref{main} in the case $\frac{1}{\beta_1} + \frac{1}{\beta_2} = \frac{2}{n-2k}$}

In this subsection, we give an alternative proof of Theorem \ref{main} in the case $1/\beta_1 + 1/\beta_2 = 2/(n-2k)$ and $a_{1}$, $a_{2}<0$.

\begin{proof}[\unskip\nopunct]

By the assumptions on $\beta_{1}$ and $\beta_{2}$,
\begin{equation*}
(n(n-2k))/(n+2k)<\min\{\beta_{1},\beta_{2}\}\leq n-2k\leq \max\{\beta_{1},\beta_{2}\}<n.
\end{equation*}

By the first and second derivative estimates for the $\sigma_{k}$-Yamabe equation, it suffices to show that
\begin{equation*}
v\leq C_{1}\mbox{ for all positive $C_{r}^{2}$ solutions $v$ of (\ref{Eq:29X20-NP})}.
\end{equation*}
Suppose by contradiction that there exist positive functions $v_{i}\in C_{r}^{2}(\SSphere^{n})$ satisfying (\ref{Eq:29X20-NP}) such that $\max v_{i}\rightarrow\infty$. Let $N$ and $S$ denote respectively the north and south poles of $\SSphere^{n}$.

Throughout the proof, $C$ denotes some generic positive constant which may change from one line to another but will remain independent of $i$,

\medskip
\noindent \underline{Step 1:} Making the same argument as in the beginning of the proof of Theorem \ref{main}, we can conclude that (\ref{Eq:17III21-E1})-(\ref{Eq:26XI20-Est2}) still hold.

\medskip
\noindent \underline{Step 2:} We show that
\begin{equation}\label{step2}
\min\{v_{i}(S),v_{i}(N)\}\rightarrow\infty.
\end{equation}
This follows from (\ref{Eq:17III21-E1}) and the following lemma (which does not use $1/\beta_{1}+1/\beta_{2}=2/(n-2k)$).

\begin{lemma}\label{Lem:05VI21-L1}
Assume that $n\geq 5$, $2\leq k < n/2$, $0 < \alpha < 1$, $K  \in C_{r}^{2,\alpha}(\SSphere^{n})$ is positive and satisfies \eqref{K condition} for some $a_1, a_2 < 0$ and $(n(n-2k))/(n+2k)< \beta_1,\beta_2 < n$. Assume that $\{v_{i}\}\subset C_{r}^{2}(\SSphere^{n})$ is a sequence of positive solutions of \eqref{Eq:29X20-NP} satisfying (\ref{Eq:17III21-E1}). Then we have (\ref{step2}).
\end{lemma}

\begin{proof}
Assume $v_{i}(S)\rightarrow\infty$. Let $u_{i}$ be related to $v_{i}$ as in (\ref{Eq:19XI20-E1}). By Theorem \ref{Prop:LocEst}(a), for every $\varepsilon_{i}\rightarrow0^{+}$ and every $R_{i}\rightarrow\infty$,
\begin{equation}\label{behave}
|u_{i}(0)^{-1}u_{i}(r)-(1+\lambda_{i}^{2}r^{2})^{\frac{2-n}{2}}|\leq\varepsilon_{i} \quad  \text{ in } \quad \{0\leq r\leq r_{i}:=\lambda_{i}^{-1}R_{i}\},
\end{equation}
where $\lambda_{i}=2^{-\frac{1}{2}}\binom{n}{k}^{-\frac{1}{2k}}K(S)^{\frac{1}{2k}} u_i(0)^{\frac{2}{n-2}}$. In particular, we can choose $R_{i}$ such that $R_{i}u_{i}(0)^{-\frac{\beta_{2}}{n(n-2)}}\rightarrow0^{+}$ and $\varepsilon_{i}R_{i}^{n-2}\rightarrow0^{+}$. By Theorem \ref{Prop:LocEst} (a)-(c), we have that
\begin{equation}\label{rate2}
u_{i}(r)=e^{O(1)}u_{i}(0)^{-1}r^{2-n} \quad \text{ in } \quad \{r_{i}\leq r\leq\bar{r}_{i}\},
\end{equation}
and
\begin{equation}\label{rate}
u_{i}(r)=e^{O(1)}u_{i}(0)^{-\min\{\frac{2\beta_{2}}{n-2k}-1,1\}} \quad\text{ in } \{\bar{r}_{i}\leq r\leq 1\},
\end{equation}
where $\bar{r}_{i}=e^{O(1)}u_{i}(0)^{-\frac{2}{n-2}\max\{1-\frac{\beta_{2}}{n-2k},0\}}$. 

We will prove by contradiction that $v_{i}(N)\rightarrow\infty$. Suppose the contrary, then by (\ref{Eq:26XI20-Est2}) and (\ref{rate}), for any $0<\varepsilon<1$, we have for large $i$ that
\begin{equation}\label{lilin}
u_{i}(r)\leq Cu_{i}(0)^{-\min\{\frac{2\beta_{2}}{n-2k}-1,1\}}r^{2-n},\quad\forall~r\geq \varepsilon.
\end{equation}

On one hand, the  Kazdan--Warner-type identity (see \eqref{Eq:13XI20-M4}) gives
\begin{equation}\label{Kz2}
\int_{0}^{\infty}r^{n}K'_{Euc}(r)u_{i}^{\frac{2n}{n-2}}\,dr =0.
\end{equation}

On the other hand, by  (\ref{lilin}), we have for large $i$ that
\begin{equation*}
\Big|\int_{\varepsilon}^{\infty}r^{n}K'_{Euc}(r)u_{i}^{\frac{2n}{n-2}}\,dr \Big|\leq C\int_{\varepsilon}^{\infty}r^{n}u_{i}^{\frac{2n}{n-2}}\,dr \leq C(\varepsilon)u_{i}(0)^{-\frac{2n}{n-2}\min\{\frac{2\beta_{2}}{n-2k}-1,1\}}.\label{zz}
\end{equation*}
For some $\varepsilon>0$ sufficiently small so that $K_{Euc}' < 0$ in $(0,\varepsilon)$ (see (\ref{K condition})), we deduce from (\ref{behave}) that
\begin{multline*}
-\int_{0}^{\varepsilon}r^{n}K'_{Euc}(r)u_{i}^{\frac{2n}{n-2}}\,dr
	 \geq -\int_{0}^{r_{i}}r^{n}K'_{Euc}(r)u_{i}^{\frac{2n}{n-2}}\,dr\\
	 \geq C\int_{0}^{r_{i}}r^{n+\beta_{2}-1}u_{i}^{\frac{2n}{n-2}}\,dr =Cu_{i}(0)^{-\frac{2\beta_{2}}{n-2}}.
\end{multline*}
Multiplying the above two inequalities by $u_{i}(0)^{\frac{2\beta_{2}}{n-2}}$, letting $i\rightarrow\infty$ and using the fact $n(n-2k)/(n+2k)<\beta_{2}<n$, we have
\begin{equation*}
\liminf\limits_{i\rightarrow\infty}u_{i}(0)^{\frac{2\beta_{2}}{n-2}}\Big(-\int_{0}^{\infty}r^{n}K'_{Euc}(r)u_{i}^{\frac{2n}{n-2}}\,dr\Big)\geq C>0,
\end{equation*}
which is a contradiction with (\ref{Kz2}).
\end{proof}

\medskip
\noindent \underline{Step 3:} We show that, for a fixed $x_{0}\in\SSphere^{n}$ and $d_{\ringg}(x_{0},S)=\pi/2$,
\begin{equation}\label{step3}
v_{i}(x_{0})v_{i}(S)^{\min\{\beta_{1},\beta_{2}\}/\beta_{1}}=e^{O(1)}
\end{equation}
and
\begin{equation}\label{step3'}
v_{i}(x_{0})v_{i}(N)^{\min\{\beta_{1},\beta_{2}\}/\beta_{2}}=e^{O(1)}.
\end{equation}
The above two estimates follow from Theorem \ref{Prop:LocEst} (a)-(c) and the facts $1/\beta_{1}+1/\beta_{2}=2/(n-2k)$, $a_{1}$, $a_{2}<0$.

\medskip
\noindent \underline{Step 4:} By (\ref{step2}) and (\ref{step3}), we have that, $v_{i}(x_{0})=e^{O(1)}v_{i}(S)^{-\frac{\min\{\beta_{1},\beta_{2}\}}{\beta_{1}}}\rightarrow0$. By (\ref{Eq:26XI20-Est2}), after passing to a subsequence if necessary, $v_{i}(x_{0})^{-1}v_{i}(x)$ converges in $C_{loc}^{1,\alpha}(\SSphere^{n}\texttt{\symbol{'134}}\{N,S\})$ to some positive axisymmetric function $G\in C_{loc}^{1,1}(\SSphere^{n}\texttt{\symbol{'134}}\{N,S\})$ which satisfies
\[
\lambda(A_{g_{G}})\in\partial\Gamma_{k}~~\mbox{in }\SSphere^{n}\texttt{\symbol{'134}}\{N,S\}.
\]
By the classification result \cite[Theorem 1.6]{LiNgBocher}, we have that $c_{1}:=\lim\limits_{x\rightarrow S}d_{\ringg}(x,S)^{n-2}G(x)\in[0,\infty)$, $c_{2}:=\lim\limits_{x\rightarrow N}d_{\ringg}(x,N)^{n-2}G(x)\in[0,\infty)$, $\max\{c_{1},c_{2}\}>0$, and in the stereographic projection coordinates as at the beginning of Section 2,
\begin{equation}\label{explicitly}
G(x)=2^{2-n}(1+r^{2})^{\frac{n-2}{2}}\Big(c_{1}^{\frac{n-2k}{k(n-2)}}r^{-\frac{n-2k}{k}}+c_{2}^{\frac{n-2k}{k(n-2)}}\Big)^{\frac{k(n-2)}{n-2k}}.
\end{equation}
By (\ref{step3}) and (\ref{step3'}) and after passing to a subsequence if necessary, we have that $v_{i}(x_{0})v_{i}(S)^{\min\{\beta_{1},\beta_{2}\}/\beta_{1}}$ and $v_{i}(x_{0})v_{i}(N)^{\min\{\beta_{1},\beta_{2}\}/\beta_{2}}$ converge respectively to two positive constants $c_{3}$ and $c_{4}$. Therefore,
\begin{equation}\label{qiqi}
v_{i}(S)^{\min\{\beta_{1},\beta_{2}\}/\beta_{1}}v_{i}(x)\rightarrow c_{3}G(x),~~~~\mbox{in }C_{loc}^{1,\alpha}(\SSphere^{n}\texttt{\symbol{'134}}\{N,S\}),
\end{equation}
and
\begin{equation*}
v_{i}(N)^{\min\{\beta_{1},\beta_{2}\}/\beta_{2}}v_{i}(x)\rightarrow c_{4}G(x),~~~~\mbox{in }C_{loc}^{1,\alpha}(\SSphere^{n}\texttt{\symbol{'134}}\{N,S\}).
\end{equation*}

Next we show that
\begin{equation}\label{a val}
c_{1}c_{3}=
\begin{cases}
\Big(\frac{K(S)}{2^{k}\binom{n}{k}}\Big)^{-\frac{n-2}{2k}} &\text{ if } ~\beta_{2}\geq \beta_{1},\\
0 &\text{ if } ~\beta_{2}<\beta_{1},
\end{cases}
\end{equation}
\begin{equation}\label{b val+}
c_{2}c_{3}=
\begin{cases}
2^{n-2}\Big(-a_{2}\beta_{2} \frac{\Gamma(\frac{n-\beta_2}{2})\Gamma(\frac{n+\beta_2}{2})}{2\Gamma(n)}\Big)^{\frac{n-2}{n-2k}}\Big(\frac{2^{k\beta_{2}}\binom{n}{k}^{\beta_{2}}}{K(S)^{2k+\beta_{2}}}\Big)^{\frac{n-2}{2k(n-2k)}} &\text{ if } \beta_{2}\leq\beta_{1}\\
0 &\text{ if } \beta_{2}>\beta_{1},
\end{cases}
\end{equation}
\begin{equation}\label{b val}
c_{2}c_{4}=
\begin{cases}
\Big(\frac{K(N)}{2^{k}\binom{n}{k}}\Big)^{-\frac{n-2}{2k}} & \text{ if } ~\beta_{2}\leq \beta_{1},\\
0 &\text{ if } ~\beta_{2}>\beta_{1},
\end{cases}
\end{equation}
and
\begin{equation}\label{a val+}
c_{1}c_{4}=
\begin{cases}
2^{n-2}\Big(-a_{1}\beta_{1}\frac{\Gamma(\frac{n-\beta_1}{2})\Gamma(\frac{n+\beta_1}{2})}{2\Gamma(n)}\Big)^{\frac{n-2}{n-2k}}\Big(\frac{2^{k\beta_{1}}\binom{n}{k}^{\beta_{1}}}{K(N)^{2k+\beta_{1}}}\Big)^{\frac{n-2}{2k(n-2k)}} & \text{ if } ~\beta_{2}\geq\beta_{1},\\
0 & \text{ if } ~\beta_{2}<\beta_{1}.
\end{cases}
\end{equation}
We will only need to prove (\ref{a val}) and (\ref{b val+}), since (\ref{b val}) and (\ref{a val+}) follow by switching the roles of $S$ and $N$.

Let $u_{i}$ be related to $v_{i}$ as in (\ref{Eq:19XI20-E1}) and let $w_{i}=u_{i}^{(n-2k)/k(n-2)}$. Fix a small $\sigma>0$ such that $K'_{Euc}(r)<0$ on $[0,\sigma]$. To prove (\ref{a val}) and (\ref{b val+}), we first establish the following two identities.
\begin{equation}
\sigma^{\frac{n-k}{k}}w_{i}'(\sigma)=-\frac{n(n-2k)}{2k^{2}\binom{n}{k}^{1/k}}\int_{0}^{\sigma}K_{Euc}(r)^{1/k}r^{(n-k)/k}w^{\frac{n+2k}{n-2k}}_{i}(1-\rho_{i}(r))^{\frac{1-k}{k}}\,dr.\label{juncheng}
\end{equation}
and
\begin{equation}\label{guiqiuliu}
\frac{E(\sigma,w_{i},w_{i}')}{\sigma^{(n-k)/k}w_{i}'(\sigma)}=-\frac{2k^{2}}{n(n-2k)}\frac{(K_{Euc}(\sigma)(1-\rho_{i}(\sigma))^{1/k}}{\int_{0}^{\sigma}K_{Euc}(r)^{1/k}r^{(n-k)/k}w^{\frac{n+2k}{n-2k}}_{i}(1-\rho_{i}(r))^{\frac{1-k}{k}}\,dr},
\end{equation}
where $\rho_{i}(r)=K_{Euc}(r)^{-1}r^{-n}w_{i}(r)^{-\frac{2nk}{n-2k}}\int_{0}^{r}K_{Euc}'(s)s^{n}w_{i}(s)^{\frac{2nk}{n-2k}}\,ds$ and $$E(r,w_{i},w_{i}')=\frac{2k}{n-2k}w_{i}^{-\frac{2n}{n-2k}}\Big[-\frac{w_{i}w_{i}'}{r}-\frac{k}{n-2k}(w_{i}')^{2}\Big].$$

Equations (\ref{Eq:19XI20-E2}) and (\ref{Eq:13XI20-M4}) can be rewritten in terms of $w_{i}$ as 
\begin{equation}\label{kfc}
\left\{\begin{array}{rcl}
w_{i}''+\frac{n-k}{k}\frac{w_{i}'}{r} &=&-\frac{n(n-2k)}{2k^{2}\binom{n}{k}}K_{Euc}(r)w_{i}^{\frac{n+2k}{n-2k}}E(r,w_{i},w_{i}')^{1-k} ~~~~\mbox{in }[0,\infty),\\
E(r,w_{i},w_{i}') &>& 0 ~~~~\mbox{in }[0,\infty),
\end{array}\right.
\end{equation}
and
\begin{equation}\label{huifen}
E(r,w_{i},w_{i}')^{k}=\frac{K_{Euc}(r)}{\binom{n}{k}}(1-\rho_{i}(r)),
\end{equation}
respectively. Raising (\ref{huifen}) to the power of $\frac{1-k}{k}$ and then inserting it into (\ref{kfc}), we have
\begin{equation*}
w_{i}''+\frac{n-k}{k}\frac{w_{i}'}{r}=-\frac{n(n-2k)}{2k^{2}\binom{n}{k}^{1/k}}K_{Euc}(r)^{1/k}w_{i}^{\frac{n+2k}{n-2k}}(1-\rho_{i}(r))^{\frac{1-k}{k}}.
\end{equation*}
Multiplying the above identity by $r^{\frac{n-k}{k}}$ and then integrating it on $[0,\sigma]$ give (\ref{juncheng}). Raising (\ref{huifen}) to the power of $1/k$, evaluating it at $r=\sigma$, and then dividing it by (\ref{juncheng}), we obtain (\ref{guiqiuliu}).

Now we use identities (\ref{juncheng}) and (\ref{guiqiuliu}) to obtain (\ref{a val}) and (\ref{b val+}). By (\ref{qiqi}) and (\ref{explicitly}), we have
\begin{align*} u_{i}(0)^{\min\{\beta_{1},\beta_{2}\}/\beta_{1}}u_{i}(r)&=(2^{\frac{n-2}{2}}v_{i}(S))^{\min\{\beta_{1},\beta_{2}\}/\beta_{1}}((\frac{2}{1+r^{2}})^{\frac{n-2}{2}}v_{i}(x))\\
&=2^{\frac{n-2}{2}(\frac{\min\{\beta_{1},\beta_{2}\}}{\beta_{1}}-1)}c_{3}(c_{1}^{\frac{n-2k}{k(n-2)}}r^{-\frac{n-2k}{k}}+c_{2}^{\frac{n-2k}{k(n-2)}})^{\frac{k(n-2)}{n-2k}}+o(1)
\end{align*}
in $C_{loc}^{1,\alpha}(\RR^{n})$. It follows that
\begin{equation}\label{lhs1}
w_{i}(0)^{\min\{\beta_{1},\beta_{2}\}/\beta_{1}}\times\mbox{LHS of }(\ref{juncheng})= -\frac{n-2k}{k}2^{\frac{n-2k}{2k}(\frac{\min\{\beta_{1},\beta_{2}\}}{\beta_{1}}-1)}(c_{1}c_{3})^{\frac{n-2k}{k(n-2)}}+o(1),
\end{equation}
and
\begin{align}
&\quad\frac{(n-2k)^{2}}{2k^{2}}w_{i}(0)^{\min\{\beta_{1},\beta_{2}\}/\beta_{1}}w_{i}(\sigma)^{\frac{2n}{n-2k}}\sigma^{\frac{n}{k}}\times\mbox{LHS of }(\ref{guiqiuliu})\nonumber\\
&\quad\quad\quad\quad=-\frac{n-2k}{k}2^{\frac{n-2k}{2k}(\frac{\min\{\beta_{1},\beta_{2}\}}{\beta_{1}}-1)}(c_{2}c_{3})^{\frac{n-2k}{k(n-2)}}+o(1).\label{lhs2}
\end{align}

Before estimating the left hand sides of the above identities, we will first give the following estimates:
\begin{align}
I_{1}
	&:=\int_{0}^{\sigma}K_{Euc}(r)^{1/k}r^{(n-k)/k}w^{\frac{n+2k}{n-2k}}_{i}(1-\rho_{i}(r))^{\frac{1-k}{k}}\,dr\nonumber\\
	&= (1 + o(1)) 2^{\frac{n}{2k}}\binom{n}{k}^{\frac{n}{2k^{2}}} \frac{k}{n}  K(S)^{-\frac{n-2k}{2k^{2}}}w_{i}(0)^{-1},
	\label{I1}
\end{align}
and 
\begin{align}
I_{2}
	&:=K_{Euc}(\sigma)\sigma^{n}w_{i}(\sigma)^{\frac{2nk}{n-2k}}
	-\int_{0}^{\sigma}K_{Euc}'(s)s^{n}w_{i}^{\frac{2nk}{n-2k}}\,ds\nonumber\\
	&= (1 + o(1))a_{2}\beta_{2}2^{\frac{n+3\beta_{2}}{2}}\binom{n}{k}^{\frac{n+\beta_{2}}{2k}} \frac{\Gamma(\frac{n-\beta_2}{2})\Gamma(\frac{n+\beta_2}{2})}{2\Gamma(n)} K(S)^{-\frac{n+\beta_{2}}{2k}}w_{i}(0)^{-\frac{2k\beta_{2}}{n-2k}}.
	\label{I2}
\end{align}

Recall $r_{i}=\lambda_{i}^{-1}R_{i}$ as in (\ref{behave}) and write $I_{1} =I_{1,1}+I_{1,2}$ where $I_{1,1}$ and $I_{1,2}$ correspond to the integrals over $[0,r_{i}]$ and $[r_{i},\sigma]$ respectively. By (\ref{behave}), we have
\begin{align*}
I_{1,1}&=(1+o(1))\int_{0}^{r_{i}}K_{Euc}^{1/k}(r)r^{(n-k)/k}w_{i}^{\frac{n+2k}{n-2k}}\,dr\\
&=(1+o(1))K_{Euc}^{1/k}(0)\int_{0}^{r_{i}}r^{(n-k)/k}w_{i}^{\frac{n+2k}{n-2k}}\,dr\\
&=(1+o(1))K_{Euc}^{1/k}(0)\lambda_{i}^{-n/k}w_{i}(0)^{\frac{n+2k}{n-2k}}\int_{0}^{\infty}s^{(n-k)/k}(1+s^{2})^{-\frac{n+2k}{2k}}\,ds\\
&=(1+o(1))\Big(2^{k}\binom{n}{k}\Big)^{\frac{n}{2k^{2}}}K(S)^{-\frac{n-2k}{2k^{2}}}w_{i}(0)^{-1}\int_{0}^{\infty}s^{(n-k)/k}(1+s^{2})^{-\frac{n+2k}{2k}}\,ds,
\end{align*}
where, in the first equality, we have used the fact that for any $0<r\leq r_{i}$,
\begin{align*}
|\rho_{i}(r)|	
	&\stackrel{(\ref{K condition})}{\leq}Cw_{i}(r)^{\frac{-2nk}{n-2k}}\int_{0}^{r}s^{\beta_{2}-1}w_{i}(s)^{\frac{2nk}{n-2k}}\,ds
		\stackrel{w_{i}(r)\geq w_{i}(r_{i})}{\leq} Cw_{i}(r_{i})^{\frac{-2nk}{n-2k}}\int_{0}^{r_{i}}s^{\beta_{2}-1}w_{i}(s)^{\frac{2nk}{n-2k}}\,ds\\
	&\stackrel{(\ref{behave})}{\leq} Cw_{i}^{-\frac{2nk}{n-2k}}(0)(1+\lambda_{i}^{2}r_{i}^{2})^{n}\int_{0}^{r_{i}}s^{\beta_{2}-1}w_{i}^{\frac{2nk}{n-2k}}ds
		\stackrel{(\ref{behave})}{\leq} Cw_{i}(0)^{-\frac{2k\beta_{2}}{n-2k}}R_{i}^{2n}=o(1).
\end{align*}
Using the fact that $K'<0$ on $[r_{i},\sigma]$, and estimate (\ref{rate2}) in the interval $[r_i, \bar r_i]$ and estimate (\ref{rate}) in the interval $[\bar r_i, \sigma]$, we have 
\begin{equation*}
I_{1,2}\stackrel{K'<0}{\leq} \int_{r_{i}}^{\sigma} r^{(n-k)/k}w_{i}^{\frac{n+2k}{n-2k}}\,dr \stackrel{(\ref{rate2}), (\ref{rate})}{=}o(1)w_{i}(0)^{-1}.
\end{equation*}
Combining the above estimates of $I_{1,1}$ and $I_{1,2}$ and using Corollary \ref{Cor:IntegralId} give (\ref{I1}).

Now we estimate (\ref{I2}). We write $I_{2}$ as
\[
I_{2}=K_{Euc}(\sigma)\sigma^{n}w_{i}(\sigma)^{\frac{2nk}{n-2k}}-(\int_{0}^{r_{i}}+\int_{r_{i}}^{\sigma})(K_{Euc}'(s)s^{n}w_{i}^{\frac{2nk}{n-2k}}\,ds)=: I_{2,1}+I_{2,2}+I_{2,3}.
\]
By (\ref{rate}) and the fact $1/\beta_{1}+1/\beta_{2}=2/(n-2k)$, $|I_{2,1}|\leq Cw_{i}(0)^{-\frac{2nk}{n-2k}\frac{\min\{\beta_{1},\beta_{2}\}}{\beta_{1}}}$. By (\ref{K condition}) and (\ref{behave}), 
\begin{align*}
I_{2,2}
	&=(-a_{2}\beta_{2}2^{\beta_{2}}+o(1))\int_{0}^{r_{i}}s^{n+\beta_{2}-1}w_{i}^{\frac{2nk}{n-2k}}\,ds\\
	&=(-a_{2}\beta_{2}2^{\beta_{2}}+o(1))\lambda_{i}^{-n-\beta_{2}}w_{i}(0)^{\frac{2nk}{n-2k}}\int_{0}^{\infty}r^{n+\beta_{2}-1}(1+r^{2})^{-n}\,dr\\
&=(-a_{2}\beta_{2}2^{\beta_{2}}+o(1))(2^{k}\binom{n}{k})^{\frac{n+\beta_{2}}{2k}}K(S)^{-\frac{n+\beta_{2}}{2k}}w_{i}(0)^{-\frac{2k\beta_{2}}{n-2k}}\int_{0}^{\infty}r^{n+\beta_{2}-1}(1+r^{2})^{-n}.
\end{align*}
By (\ref{K condition}), (\ref{rate2}) and (\ref{rate}), 
\begin{align*}
|I_{2,3}|
	&\leq C(\int_{r_{i}}^{\bar{r}_{i}}+\int_{\bar{r}_{i}}^{\sigma})s^{n+\beta_{2}-1}w_{i}^{\frac{2nk}{n-2k}}\\
	&=o(1)w_{i}(0)^{-\frac{2k\beta_{2}}{n-2k}}+O(1)w_{i}(0)^{-\frac{2nk}{n-2k}\frac{\min\{\beta_{1},\beta_{2}\}}{\beta_{1}}}
		=o(1)w_{i}(0)^{-\frac{2k\beta_{2}}{n-2k}}.
\end{align*}
Combining the above estimates of $I_{2,1}$, $I_{2,2}$ and $I_{2,3}$ together and using Corollary \ref{Cor:IntegralId} give (\ref{I2}).

Now, by (\ref{I1})
\begin{align}
&\quad w_{i}(0)^{\min\{\frac{2\beta_{2}}{n-2k}-1,1\}}\times\mbox{RHS of }(\ref{juncheng})=-\frac{n(n-2k)}{2k^{2}\binom{n}{k}^{1/k}}w_{i}(0)^{\frac{\min\{\beta_{1},\beta_{2}\}}{\beta_{1}}}I_{1}\nonumber\\
&\quad\quad\quad\quad=- (1 + o(1)) \frac{n-2k}{k}\Big(\frac{2^{k}\binom{n}{k}}{K(S)}\Big)^{\frac{n-2k}{2k^{2}}}w_{i}(0)^{\frac{\min\{\beta_{1},\beta_{2}\}}{\beta_{1}}-1}
.\label{rhs1}
\end{align}
By (\ref{I1}) and (\ref{I2}), 
\begin{align}
&\frac{(n-2k)^{2}}{2k^{2}}w_{i}(0)^{\min\{\frac{2\beta_{2}}{n-2k}-1,1\}}w_{i}(\sigma)^{\frac{2n}{n-2k}}\sigma^{\frac{n}{k}}\times\mbox{RHS of }(\ref{guiqiuliu})\nonumber\\
	&\qquad=-\frac{n-2k}{n}w_{i}(0)^{\frac{\min\{\beta_{1},\beta_{2}\}}{\beta_{1}}}(I_{2}^{1/k}/I_{1})\nonumber\\
	&\qquad
=-(1+o(1))\frac{n-2k}{k}w_{i}(0)^{\frac{\min\{\beta_{1},\beta_{2}\}-\beta_{2}}{\beta_{1}}}\frac{\Big(-a_{2}\beta_{2}2^{\beta_{2}} \frac{\Gamma(\frac{n-\beta_2}{2})\Gamma(\frac{n+\beta_2}{2})}{2\Gamma(n)}\Big)^{\frac{1}{k}}(2^{k}\binom{n}{k})^{\frac{\beta_{2}}{2k^{2}}}}{K(S)^{\frac{2k+\beta_{2}}{2k^{2}}}}.
\label{rhs2}\end{align}

Inserting (\ref{lhs1}) and (\ref{rhs1}) into (\ref{juncheng}), passing to limit, and raising to the power of $\frac{k(n-2)}{n-2k}$ give (\ref{a val}). Inserting (\ref{lhs2}) and (\ref{rhs2}) into (\ref{guiqiuliu}), passing to limit, and raising to the power of $\frac{k(n-2)}{n-2k}$ give (\ref{b val+}).

\medskip
\noindent \underline{Step 5:} We make use of the  Kazdan--Warner-type identity to show that
\begin{equation}\label{c5}
\left(\frac{K(S)^{k+\beta_{2}}}{K(N)^{k+\beta_{1}}}\right)^{\frac{n-2}{4k}}\left(\lim\limits_{i\rightarrow\infty}\frac{v_{i}(S)^{\beta_{2}}}{v_{i}(N)^{\beta_{1}}}\right)=\left(\frac{2^{\frac{\beta_{2}}{2}}a_{2}\beta_{2}\binom{n}{k}^{\frac{\beta_{2}}{2k}}\Gamma(\frac{n-\beta_2}{2})\Gamma(\frac{n+\beta_2}{2})}{2^{\frac{\beta_{1}}{2}}a_{1}\beta_{1}\binom{n}{k}^{\frac{\beta_{1}}{2k}}\Gamma(\frac{n-\beta_1}{2})\Gamma(\frac{n+\beta_1}{2})}\right)^{\frac{n-2}{2}}.
\end{equation}
Indeed, the  Kazdan--Warner-type identity (see \eqref{Eq:13XI20-M4}) gives
\begin{equation*}
\int_{\SSphere^{n}}\langle\nabla K(x),\nabla x_{n+1}\rangle v_{i}^{\frac{2n}{n-2}}=0,
\end{equation*}
or equivalently,
\begin{equation}\label{kz}
\int_{d_{\ringg}(x,S)\leq \pi/2}\langle\nabla K(x),\nabla x_{n+1}\rangle v_{i}^{\frac{2n}{n-2}}=\int_{d_{\ringg}(x,N)\leq \pi/2}\langle\nabla K(x),\nabla(-x_{n+1})\rangle v_{i}^{\frac{2n}{n-2}}.
\end{equation}
We will show that
\begin{equation}\label{lhs of kw}
\mbox{LHS of (\ref{kz})}= ( 1 + o(1)) 2^{\frac{n+\beta_{2}}{2}}a_{2}\beta_{2}\binom{n}{k}^{\frac{n+\beta_{2}}{2k}} \frac{\Gamma(\frac{n-\beta_2}{2})\Gamma(\frac{n+\beta_2}{2})}{2\Gamma(n)}K(S)^{-\frac{k+\beta_{2}}{2k}}v_{i}(S)^{-\frac{2\beta_{2}}{n-2}},
\end{equation}
and
\begin{equation}\label{rhs of kw}
\mbox{RHS of (\ref{kz})} =  (1 + o(1))2^{\frac{n+\beta_{1}}{2}}a_{1}\beta_{1}\binom{n}{k}^{\frac{n+\beta_{1}}{2k}} \frac{\Gamma(\frac{n-\beta_1}{2})\Gamma(\frac{n+\beta_1}{2})}{2\Gamma(n)} K(N)^{-\frac{k+\beta_{1}}{2k}}v_{i}(N)^{-\frac{2\beta_{1}}{n-2}}.
\end{equation}
Inserting (\ref{lhs of kw}) and (\ref{rhs of kw}) into (\ref{kz}) gives
\begin{equation*}
\frac{K(S)^{\frac{k+\beta_{2}}{2k}}v_{i}(S)^{\frac{2\beta_{2}}{n-2}}}{K(N)^{\frac{k+\beta_{1}}{2k}}v_{i}(N)^{\frac{2\beta_{1}}{n-2}}}=\frac{2^{\frac{\beta_{2}}{2}}a_{2}\beta_{2}\binom{n}{k}^{\frac{\beta_{2}}{2k}}\Gamma(\frac{n-\beta_2}{2})\Gamma(\frac{n+\beta_2}{2})}{2^{\frac{\beta_{1}}{2}}a_{1}\beta_{1}\binom{n}{k}^{\frac{\beta_{1}}{2k}}\Gamma(\frac{n-\beta_1}{2})\Gamma(\frac{n+\beta_1}{2})}+o(1).
\end{equation*}
Raising to the power of $(n-2)/2$ and letting $i\rightarrow\infty$ give (\ref{c5}). 

We will only need to prove (\ref{lhs of kw}), since (\ref{rhs of kw}) follows by switching the roles of $S$ and $N$.

Now we prove (\ref{lhs of kw}). Let $u_{i}$ be related to $v_{i}$ as in (\ref{Eq:19XI20-E1}), then $\mbox{LHS of (\ref{kz})}=\int_{0}^{1}K'_{Euc}(r)r^{n} u_{i}^{\frac{2n}{n-2}}$. In order to estimate this integral, we divide the integral into two parts: $I_{4}$, the integral on $[0,r_{i}]$, and $I_{5}$, the integral on $[r_{i},1]$. By (\ref{K condition}) and (\ref{behave}),
\begin{align*}
I_{4}
	&\stackrel{(\ref{K condition})}{=}2^{\beta_{2}}a_{2}\beta_{2}(1+o(1))\int_{0}^{r_{i}}r^{n+\beta_{2}-1}u_{i}^{\frac{2n}{n-2}}\,dr\\
&\stackrel{(\ref{behave})}{=} (1 + o(1)) 2^{\beta_{2}}a_{2}\beta_{2}2^{\frac{n+\beta_{2}}{2}}\binom{n}{k}^{\frac{n+\beta_{2}}{2k}}K(S)^{-\frac{n+\beta_{2}}{2k}}u_{i}(0)^{-\frac{2\beta_{2}}{n-2}} \int_{0}^{\infty}s^{n+\beta_{2}-1}(1+s^{2})^{-n}\,ds\\
&= (1 + o(1)) a_{2}\beta_{2}2^{\frac{n+\beta_{2}}{2}}\binom{n}{k}^{\frac{n+\beta_{2}}{2k}}K(S)^{-\frac{n+\beta_{2}}{2k}}v_{i}(S)^{-\frac{2\beta_{2}}{n-2}}\int_{0}^{\infty}s^{n+\beta_{2}-1}(1+s^{2})^{-n}\,ds.
\end{align*}
By (\ref{rate2}) and (\ref{rate}),
\begin{align*}
|I_{5}|
	&\leq C\int_{r_{i}}^{1}r^{n}u_{i}^{\frac{2n}{n-2}}\leq C\Big[\int_{r_{i}}^{\bar{r}_{i}}r^{n}(u_{i}(0)(1+\lambda_{i}^{2}r^{2}))^{\frac{2n}{n-2}}\,dr +u_{i}(0)^{-\frac{2n}{n-2}\min\{\frac{2\beta_{2}}{n-2k}-1,1\}}\Big]\\
	&=o(u_{i}(0)^{-\frac{2\beta_{2}}{n-2}}).
\end{align*}
Combining the above two estimates together and using Corollary \ref{Cor:IntegralId} give (\ref{lhs of kw}).

\medskip
\noindent \underline{Step 6:} We reach a contradiction. Let $c_{5}:=\lim\limits_{i\rightarrow\infty}(v_{i}(S)^{\beta_{2}}/v_{i}(N)^{\beta_{1}})$. Then it is easy to see that
\begin{equation}\label{relation}
c_{5}=(c_{3}/c_{4})^{\frac{\beta_{1}\beta_{2}}{\min\{\beta_{1},\beta_{2}\}}}
\end{equation}

If $\beta_{2}\geq\beta_{1}$, by (\ref{a val}), $c_{1}>0$. Dividing (\ref{a val}) by (\ref{a val+}), inserting it into the right hand side of (\ref{relation}) and inserting (\ref{c5}) into the left hand side of (\ref{relation}), we have 
\begin{equation}\label{contra}
C_{n,k}(a_{1},K(N))C_{n,k}(a_{2},K(S))= 1 ,
\end{equation}
which is a contradiction to (\ref{Eq:BalCond}). If $\beta_{1}\geq\beta_{2}$, by (\ref{b val+}), $c_{2}>0$. Dividing (\ref{b val+}) by (\ref{b val}), inserting it into the right hand side of (\ref{relation}) and inserting (\ref{c5}) into the left hand side of (\ref{relation}), we have (\ref{contra}), which is again a contradiction to (\ref{Eq:BalCond}).
\end{proof}


\section{The total degree: Proof of Theorem \ref{Thm:DegForm}}\label{Sec:Degree}

The computation of the degree is a direct adaptation of the computation in \cite{Li95-JDE, LiNgWang-NPkLarge} to the case of axisymmetry. For completeness, we present a sketch.

Fix some $0 < \alpha' \leq \alpha < 1$. By Theorem \ref{main} and first and second derivative estimates for the $\sigma_k$-Yamabe equation (see \cite{Chen05, GW03-IMRN}, \cite[Theorem 1.10]{Li09-CPAM}, \cite[Theorem 1.20]{LiLi03}, \cite{Wang06}), we can select $C_*$ sufficiently large such that all axisymmetric positive solutions to \eqref{Eq:29X20-NP} belong to the set
\[
\mcO = \Big\{\tilde v \in C_r^{4,\alpha'}(\SSphere^n): \|\ln \tilde v\|_{C^{4,\alpha'}(\SSphere^n)} <  C_*, \lambda(A_{g_{\tilde v}}) \in \Gamma_k\Big\}.
\]
Consider the nonlinear operator $F: \mcO \rightarrow C_r^{2,\alpha'}(\SSphere^{n})$ defined by
\[
F[v]:=\sigma_{k}(\lambda(A_{g_{v}}))-K,\quad\forall~v\in \mcO .
\]
By \cite{Li89-CPDE}, the degree $\deg(F, \mcO, 0)$ is well-defined and is independent of $\alpha' \in (0,\alpha]$ (see \cite[Theorem B.1]{Li95-JDE}). 

If $a_1, a_2 < 0$ and $\frac{1}{\beta_1} + \frac{1}{\beta_2} > \frac{2}{n-2k}$, we have in view of the homotopy invariance property of the degree, the non-existence result Theorem \ref{nonexist} and Remark \ref{Rem:27II21-R1} that $\deg(F, \mcO, 0) = 0$.

If $a_1, a_2 < 0$, $\frac{1}{\beta_1} + \frac{1}{\beta_2} = \frac{2}{n-2k}$ and $C_{n,k}(\beta_1,a_1,K(0)) C_{n,k}(\beta_2,a_2,K(\pi)) < 1$, it follows the compactness estimate Theorem \ref{Thm:27II21-T1} and the above statement that $\deg(F, \mcO, 0) = 0$.

In all remaining cases, in view of the compactness estimate Theorem \ref{Thm:27II21-T1} and Remark \ref{Rem:27II21-R1}, we may assume without loss of generality that $\beta_1, \beta_2 > n-2k$.

We continue by deforming $K$ to a constant. For $\mu\in(0,1]$, we let $K_\mu = \mu K + (1-\mu) 2^{-k} \binom{n}{k}$ and consider the equation \eqref{Eq:29X20-NPmu}. By Theorem \ref{Thm:19XI20-M1}, we may assume that all axisymmetric positive solutions to \eqref{Eq:29X20-NPmu} for $\mu \in (0,1]$ belong to the set $\mcO$.

Let $F_{\mu}: \mcO \rightarrow C_r^{2,\alpha'}(\SSphere^{n})$ be defined by
\begin{equation}
F_{\mu}[v]:=\sigma_{k}(\lambda(A_{g_{v}}))-K_{\mu},\quad\forall~v\in \mcO .
	\label{Eq:25XII20-FmuDef}
\end{equation}
Then the degree $\deg(F_\mu, \mcO, 0)$ is well-defined and is independent of $\mu \in (0,1]$ and of $\alpha' \in (0,\alpha]$. We would now like to compute this degree for small $\mu$ and some $\alpha' \in (0,\alpha)$, using the Lyapunov-Schmidt reduction.

We parametrize $C^{4,\alpha'}_r(\SSphere^n)$ as $\mcS_0 \times \RR$ where the $\RR$-factor takes into account the the action of the M\"obius group on $\SSphere^n$ on axisymmetric functions and where the element $1 \in \mcS_0$ corresponds to the so-called axisymmetric standard bubbles on $\SSphere^n$. To this end, for $t \in \RR$, let $\varphi_{t}$ be the M\"obius transformation on $\SSphere^{n}$ which, under stereographic projection with respect to the north pole, sends $y$ to $ty$. For function $v$ defined on $\SSphere^n$, we let 
\[
T_{t}v:=v\circ\varphi_t|\det d\varphi_t|^{\frac{n-2}{2n}}
\]
where $d\varphi_t$ denotes the Jacobian of $\varphi_t$. In particular, the pull-back metric of $g_v = v^{\frac{4}{n-2}}\ringg$ under $\varphi_t$ is given by $\varphi_t^* (g_v) = g_{T_t v}$.

Let
\begin{equation*}
\mcS_{0} =\Big\{v\in C^{4,\alpha'}_r(\SSphere^{n}):\int_{\SSphere^{n}}x^{n+1} |v(x)|^{\frac{2n}{n-2}}\,dv_{\ringg}(x) =0\Big\}.
\end{equation*}
For $w \in \mcS_0$ and $t \in \RR$, let $\pi(w,t)$ be defined by $\pi(w,0) = w$ and 
\[
\pi(w, t) = T_t^{-1}(w).
\]
It can be checked that the map $\pi: \mcS_0 \times \RR \mapsto C^{4,\alpha'}_r(\SSphere^n)$ is a $C^2$ diffeomorphism.

As in \cite{LiNgWang-NPkLarge}, the Theorem \ref{Thm:19XI20-M1} and Liouville-type theorem give
\begin{lemma}\label{Lem:25XII20-wConv}
Let $n\geq 5$, $2\leq k < n/2$, and $0 < \alpha' < \alpha < 1$. Suppose that $K \in C^{2,\alpha}_r(\SSphere^n)$ be as in Theorem \ref{Thm:DegForm} with $\beta_1, \beta_2 > n-2k$. If $v_{\mu_j} = \pi(w_{\mu_j}, t_{\mu_j})$ solves \eqref{Eq:29X20-NPmu} for some sequence $\mu_j \rightarrow 0^+$, then $t_{\mu_j}$ stays in a compact interval of $\RR$ and 
\[
\lim_{j \rightarrow \infty} \|w_{\mu_j} - 1\|_{C^{4,\alpha'}(\SSphere^n)} = 0.
\]
\end{lemma}

The linearized operator of $F_\mu[\pi(\cdot,t)]$ at $\bar w \equiv 1$ is readily found to be
\[
\mcL := D_w (F_\mu \circ \pi)(w,\xi)]\Big|_{w = \bar w} = - d_{n,k}(\Delta + n) \quad \text{ with } \quad d_{n,k} := \frac{2^{2-k}}{n-2} \Big(\begin{array}{c}n\\k\end{array}\Big)
\]
and with domain $D(\mcL)$ being the tangent plane to $\mcS_0$ at $w = \bar w$:
\[
D(\mcL) := T_1(\mcS_0) = \Big\{\eta \in C^{4,\alpha'}_r(\SSphere^n): \int_{\SSphere^n} x \eta(x)\,dv_{\ringg}(x) = 0\Big\}.
\]
It is well-known that $\mcL$ is an isomorphism from $D(\mcL)$ to
\[
R(\mcL) := \Big\{f \in C^{2,\alpha'}_r(\SSphere^n): \int_{\SSphere^n} x^{n+1} f(x)\,dv_{\ringg}(x) = 0\Big\}.
\]
Let $\Pi$ be a projection from $C^{2,\alpha'}_r(\SSphere^n)$ onto $R(\mcL)$ defined by
\[
\Pi f(x) = f(x) - \frac{n+1}{|\SSphere^n|} x^{n+1} \int_{\SSphere^n} y^{n+1} f(y)\,dv_{\ringg}(y).
\]

As in \cite{LiNgWang-NPkLarge}, we have:
\begin{proposition}\label{Prop:25XIII20-LSchmidt}
Let $n\geq5$, $2\leq k < n/2$, and $0 < \alpha' < \alpha < 1$. Suppose that $K \in C^{2,\alpha}_r(\SSphere^n)$ is positive and let $F_\mu$ be defined by \eqref{Eq:25XII20-FmuDef}. Then for every $s_0 \geq 1$, there exists a constant $\mu_0 \in (0,1]$ and a neighborhood $\mcN$ of $1$ in $\mcS_0$  such that, for every $\mu \in (0,\mu_0]$ and $\frac{1}{s_0} \leq t \leq s_0$, there exists a unique $w_{t,\mu}\in \mcN$, depending smoothly on $(t,\mu)$, such that
\begin{equation}
\Pi(F_{\mu}[\pi(w_{t,\mu},\xi)])=0.\label{25XIII20-A1}
\end{equation}
Furthermore, there exists some $C>0$ such that, for $\mu \in (0,\mu_0]$ and $\frac{1}{s_0} \leq t, t' \leq s_0$,
\begin{align*}
\|w_{t,\mu}-1\|_{C^{4,\alpha'}(\SSphere^{n})}
	&\leq C\mu \Big\|K - 2^{-k}\binom{n}{k}\Big\|_{C^{2,\alpha}(\SSphere^{n})} ,\\
\|w_{t,\mu}-w_{t',\mu}\|_{C^{4,\alpha'}(\SSphere^{n})}
	&\leq C\mu |t - t'| \Big\|K - 2^{-k}\binom{n}{k}\Big\|_{C^{2,\alpha}(\SSphere^{n})}.
\end{align*}
\end{proposition}

Note that equation (\ref{25XIII20-A1}) can be equivalently rewritten as
\[
\sigma_{k}(\lambda(A_{g_{w_{t,\mu}}}))=K_{\mu}\circ\varphi_t(x)-\Lambda_{t,\mu} x^{n+1}\quad\mbox{on }\SSphere^{n},
\]
where $\Lambda_{t,\mu} \in \RR$ is given by
\begin{equation}
\Lambda_{t,\mu}=-\frac{n+1}{|\SSphere^{n}|}\int_{\SSphere^{n}}F_{\mu}[\pi(w_{t,\mu},\xi)](x) x^{n+1}\,dv_{\ringg}(x).
	\label{Eq:25XII20-LambdaDef}
\end{equation}
Furthermore, for $\mu$ sufficiently close to $0$, $v_\mu$ solves \eqref{Eq:29X20-NPmu} if and only if $v_\mu = \pi(w_{t_\mu,\mu},t_\mu)$ and $\Lambda_{t_\mu,\mu} = 0$ for some $t_\mu$. 

Note that, in view of the Kazdan--Warner-type identity \eqref{Eq:13XI20-M4}, $\Lambda_{t,\mu}$ can be expressed more directly in terms of $K$ as
\begin{equation}
\frac{1}{\mu}\Lambda_{t,\mu}\int_{\SSphere^{n}} |\nabla x^{n+1}|^2 w_{t,\mu}^{\frac{2n}{n-2}}dv_{\ringg}(x) = \int_{\SSphere^{n}}\langle \nabla(K\circ\varphi_{t}),\nabla x^{n+1} \rangle w_{t,\mu}^{\frac{2n}{n-2}}dv_{\ringg}(x).
	\label{Eq:20I21-F1}
\end{equation}

The degree of the function $t \mapsto \Lambda_{t,\mu}$ can be computed in the same way as in \cite{LiNgWang-NPkLarge}:

\begin{lemma}\label{Lem:25XII20-Deg}
Let $n\geq5$, $2\leq k < n/2$, $\alpha \in (0,1)$ and $K \in C^{2,\alpha}_+(\SSphere^n)$ be as in Theorem \ref{Thm:DegForm}  with $\beta_1, \beta_2 > n-2k$. Let $\Lambda_{t,\mu}$ be defined as in \eqref{Eq:25XII20-LambdaDef}. Then there exist $\mu_0 \in (0,1]$ and $s_0 \in (1,\infty)$ such that, for all $\mu \in (0,\mu_0]$ and $s \in (1,s_0]$, the Brouwer degrees $\deg(\Lambda_{t,\mu},[s^{-1},s],0)$ is well-defined and
\[
\deg(\Lambda_{t,\mu},[s^{-1},s],0) = -\frac{1}{2}(-1)^{n}[\sign(a_1) + \sign(a_2)].
\]
\end{lemma}

\begin{proof}[Proof of Theorem \ref{Thm:DegForm}]
As explained at the beginning of the section, we only need to consider the case $\beta_1, \beta_2 > n-2k$. In this case, as in \cite{LiNgWang-NPkLarge}, there exist $\mu_0 \in (0,1]$ and $s_0 > 1$ such that
\[
\deg(F_\mu, \mcO, 0) =  (-1)^{n}\deg(\Lambda_{t,\mu},[s^{-1},s],0) \text{ for all } \mu \in (0,\mu_0], s \in (1,s_0).
\]
The conclusion follows from Lemma \ref{Lem:25XII20-Deg}.
\end{proof}


\section{Perturbation method: Proof of Theorem \ref{Thm:Perturb}}\label{Sec:Perturb}

\begin{proof}[Proof of Theorem \ref{Thm:Perturb}]
After a renaming of $K$ to $K_\mu$, it suffices to exhibit a function $K$ satisfying \eqref{K condition} such that $\textrm{sign}(a_i) = \varepsilon_i $ and that the equation \eqref{Eq:29X20-NPmu} has a solution for some sufficiently small $\mu$. 

Fix some $s_0 > 1$ for the moment, and let $\Lambda_{t,\mu}$ be as in Proposition \ref{Prop:25XIII20-LSchmidt}. Then \eqref{Eq:29X20-NPmu} has a positive solution if the map $t \mapsto \Lambda_{t,\mu}$ has a zero in $[s_0^{-1},s_0]$.

Prompted by formula \eqref{Eq:20I21-F1} for $\Lambda_{t,\mu}$ and the fact that $w_{t,\mu} \approx 1$ for small $\mu$, we consider the function
\[
H_K(t) = \int_{\SSphere^{n}}\langle \nabla(K\circ\varphi_{t}),\nabla x^{n+1} \rangle dv_{\ringg}(x) = n \int_{\SSphere^{n}} K\circ\varphi_{t}\, x^{n+1}  dv_{\ringg}(x). 
\]
Clearly, if $K$ and $s_0$ are such that $H_K(1)$ and $H_K(s_0)$ are of opposite signs, then for all sufficiently small $\mu$, $\Lambda_{1,\mu}$ and $\Lambda_{s_0,\mu}$ are also of opposite signs and the conclusion will follow.

We now proceed to construct $K$ and $s_0$. Let $K_\#(x) = (x^{n+1})^{2m}$ for some large $m > \beta_1, \beta_2$. Then $H_{K_\#}(1) = 0$ and $H_{K_\#}'(1) > 0$. In particular, there exists $s_0 > 1$ such that $H_{K_\#}(s_0) > 0$.

Take a function $K_*$ satisfying \eqref{K condition} with $\textrm{sign}(a_i) =\varepsilon_i $. By considering the behavior of $H_{K_*}(t)$ as $t \rightarrow 0$, we have that $H_{K_*} \not\equiv 0$. Replacing $K_*$ with $K_* \circ \varphi_t$ for some suitable $t$, we may assume also that $H_{K_*}(1) \neq 0$.

Since $H_{K_\#}(1) = 0$ and $H_{K_\#}(s_0) > 0$, there exists some $\gamma \in \RR$ such that $H_{K_* + \gamma K_\#}(1)$ and $H_{K_* + \gamma K_\#}(s_0)$ are of opposite signs.

The desired function $K$ then take the form $C + K_* + \gamma K_\#$ for some sufficiently large $C$ such that $K$ is positive.
\end{proof}

\section{Non-existence: Proof of Theorem \ref{nonexist}}\label{Sec:NonExist}

\begin{proof}[Proof of Theorem \ref{nonexist}]
Let us first prove the non-existence of positive axisymmetric solutions for some suitable $K$ with the declared properties. The fact that this implies the theorem will be dealt with at the last stage.

It is more convenient to work in cylindrical coordinates. Fix $2 \leq \beta_1, \beta_2 < n$ such that $\frac{1}{\beta_1} + \frac{1}{\beta_2} \geq \frac{2}{n-2k}$. For small $0 < \varepsilon \ll 1$ and large $T \geq 1$, fix a positive function $\hat K_{\varepsilon,T} \in C^\infty(\RR)$ such that 
\begin{align}
\hat K_{\varepsilon,T}(t) &= 1 - \frac{1}{2} e^{\beta_2 (t + T + 1)} 
	\text{ for } t \leq - T - 1,\label{Eq:27XII20-A1}\\
-2 \leq \frac{d}{dt}\hat K_{\varepsilon,T}(t) &\leq 0 \text{ for } -T-1 < t < -T,
	\label{Eq:27XII20-A2}\\ 
\hat K_{\varepsilon,T}(t) &= 1 - \frac{1}{2} e^{-\beta_1 (t - T - 1)} 
	\text{ for } t \geq T + 1,
	\label{Eq:27XII20-A3}\\
2 \geq \frac{d}{dt}\hat K_{\varepsilon,T}(t) &\geq 0 \text{ for } T < t < T+1,
	\label{Eq:27XII20-A4}\\
\hat K_{\varepsilon,T}(t) 
	&= \varepsilon \text{ for } -T \leq t \leq T.
	\label{Eq:27XII20-A5}
\end{align}
Let $t = \ln \cot\frac{\theta}{2}$ and $K_{\varepsilon,T}(\theta) = \hat K_{\varepsilon,T}(t)$. We will show that there exists $N \gg 1$ such that, whenever $T \geq N$ and $\varepsilon e^{(n+2k)T} \leq \frac{1}{N}$, there is no positive axisymmetric solution of \eqref{Eq:29X20-NP} with $K = K_{\varepsilon,T}$.

Suppose by contradiction that there exists $T_i$ and $\varepsilon_i$ with $T_i \geq i$ and $\varepsilon_i e^{(n+2k)T_i} \leq \frac{1}{i}$ such that the problem \eqref{Eq:29X20-NP} with $K = K_{\varepsilon_i,T_i}$ has a solution $v_i$. 

Let $r = e^t$ and let $u_i: \RR^n \rightarrow \RR$ and $\xi_i: \RR \rightarrow \RR$ be related to $v_i$ as in \eqref{Eq:19XI20-E1} and \eqref{Eq:12XII20-E2}. In particular, $\xi_i$ satisfies
\begin{equation}
F_k[\xi_i] = \hat K_{\varepsilon_i,T_i} \text{ and }  |\dot \xi_i| < 1 \text{ in } (-\infty,\infty),
	\label{Eq:27XII20-B1}
\end{equation}
and
\[
\xi_i(t) - |t| \text{ is bounded as } |t| \rightarrow \infty.
\]

In the sequel, we use $C$ to denote some positive generic constant which is always independent of $i$, $O(1)$ to denote a term which is bounded as $i \rightarrow \infty$, and $o(1)$ to denote a term which tends to $0$ as $i \rightarrow \infty$.

Observe that the arguments in the proof of Theorem \ref{Prop:LocEst}(a) give $\xi_i(t) \geq - C$ for $|t| > T_i + 2$. Since $|\dot\xi| < 1$, this implies that, for every $m \geq 0$
\begin{equation}
\xi_i(t) \geq - C(m) \text{ for } |t| \geq T_i - m.
	\label{Eq:27XII20-C0}
\end{equation}
Applying first and second derivative estimates for the $\sigma_k$-Yamabe equation to \eqref{Eq:27XII20-B1}, we have, for every $m > 0$, 
\begin{equation}
|\dot{\xi}_i(t)| + |\ddot{\xi}_i(t)| \leq C(m) \text{ for } |t| \geq T_i - m.
	\label{Eq:27XII20-C0X}
\end{equation}

\noindent
\underline{Step 1:} Let $Y_i = e^{-\frac{n-2k}{2k}\xi_i}$. We show that
\begin{align}
Y_i(- T_i) + \frac{2k}{n-2k} \dot Y_i(- T_i)
	&= o(1) Y_i( T_i),
	\label{Eq:23III21-S1M}\\
Y_i(T_i) - \frac{2k}{n-2k}  \dot Y_i( T_i)
	&= o(1) Y_i(-T_i).
	\label{Eq:23III21-S1}
\end{align}

We start by rewriting the equation $F_k[\xi_i] = \hat K_{\varepsilon_i,T_i}$ in the form
\[
e^{2\xi_i} \Big(\ddot \xi_i + \frac{n-2k}{2k}(1 - \dot \xi_i^2)\Big)
 = 2^{k-1}\binom{n-1}{k-1}^{-1} \frac{\hat K_{\varepsilon_i,T_i}}{e^{2(k-1)\xi_i}(1-\dot\xi_i^2)^{k-1}}.
\]
We proceed by estimating the term on the denominator on the right hand side. Recall the function $H$ defined in \eqref{Eq:HDef}, and note that, by the Pohozaev identity \eqref{Eq:13XI20-M2pre} and the monotonicity of $\hat K_{\varepsilon_i,T_i}$ in $(-\infty, T_i)$, $H(t,\xi_i,\dot \xi_i)$ is non-decreasing in $(-\infty, T_i)$. Also, since $|\dot \xi_i| < 1$, $\xi_i(t) + t$ is bounded as $t \rightarrow - \infty$ and $k < \frac{n}{2}$, $H(t,\xi_i,\dot \xi_i)  \rightarrow 0$ as $t \rightarrow -\infty$. Therefore $H(t,\xi_i,\dot \xi_i) \geq 0$, i.e.
\[
\frac{1}{2^k} \binom{n}{k} e^{2k\xi_i}(1- \dot \xi_i^2)^k  \geq  \hat K_{\varepsilon_i,T_i} (t) > 0 \text{ in } (-\infty,T_i).
\]
Inserting this into the previous equation, we obtain 
\[
0 \leq e^{2\xi} \Big(\ddot \xi_i + \frac{n-2k}{2k}(1 - \dot \xi_i^2)\Big)
  \leq a_{n,k} \hat K_{\varepsilon_i,T_i}^{1/k}, \qquad  a_{n,k} =  \binom{n}{k}^{\frac{k-1}{k}}\binom{n-1}{k-1}^{-1}.
\]
Multiplying this equation by $e^{\pm \frac{n-2k}{2k}t - \frac{n+2k}{2k}\xi}$, we get
\begin{equation}
0 \leq \pm \frac{d}{dt} \Big[ e^{\frac{n-2k}{2k}(\pm t - \xi_i)}( 1 \pm \dot \xi_i)\Big] \leq C \varepsilon_i^{\frac{1}{k}} e^{\pm\frac{n-2k}{2k}t - \frac{n+2k}{2k}\xi_i}.
	\label{Eq:25III21-E1}
\end{equation}
Using the fact that $\hat K_{\varepsilon_i,T_i} = \varepsilon_i$ in $(-T_i,T_i)$, $\xi_i(t) \geq \xi_i(\pm T_i) - (t \mp T_i)$ in $(-T_i,T_i)$ (since $|\xi_i| <1$), $T_i \rightarrow \infty$ and $\varepsilon_i e^{(n+2k)T_i} \rightarrow 0$, we can integrate \eqref{Eq:25III21-E1} to obtain
\begin{align*}
e^{-\frac{n-2k}{2k}\xi_i(T_i)}\big(1 + \dot\xi_i(T_i)\big)
	&=  o(1) e^{-\frac{n-2k}{2k}\xi_i(-T_i)} + o(1)e^{-\frac{n+2k}{2k}\xi_i(-T_i)},\\
e^{-\frac{n-2k}{2k}\xi_i(-T_i)}\big(1 -  \dot\xi_i(-T_i)\big) 
	&= o(1)e^{-\frac{n-2k}{2k}\xi_i(T_i)} +  o(1) e^{-\frac{n+2k}{2k}\xi_i(T_i)}.
\end{align*}
In view of \eqref{Eq:27XII20-C0} and the expression of $Y_i$, \eqref{Eq:23III21-S1M} and \eqref{Eq:23III21-S1} follows. Step 1 is finished.

\medskip
\noindent\underline{Step 2:} We show that \footnote{In fact, it can be seen from the proof that, when $\beta_1 \neq n-2k$, $\alpha_+ = 1$, and when $\beta_2 \neq n-2k$, $\alpha_- = -1$.}
\[
\alpha_- := \limsup_{i\rightarrow \infty} \dot\xi_i(- T_i) < 1 \text{ and } \alpha_+ = \liminf_{i \rightarrow 0} \dot\xi_i(T_i) > -1.
\]
Once this is done, after passing to a subsequence, we have
\[
\dot Y_i(- T_i) = -\frac{n-2k}{2k} (\alpha_- + o(1)) Y_i(-T_i) 
\text{ and } \dot Y_i( T_i) = -\frac{n-2k}{2k} (\alpha_+ + o(1)) Y_i(T_i),
\]
which together with \eqref{Eq:23III21-S1M} and \eqref{Eq:23III21-S1} gives
\[
0 < Y_i(T_i) = o(1)Y_i(-T_i) = o(1) Y_i(T_i),
\]
which yields a contradiction and finishes the proof of Theorem \ref{nonexist}.

We will only prove that $\alpha_- < 1$. The proof of $\alpha_+ > -1$ is similar.

Let us first show that $\xi_i(-T_i) \rightarrow \infty$ as $i \rightarrow \infty$. Indeed, by \eqref{Eq:23III21-S1M}, $Y_i(-T_i)(1 + \dot\xi_i(-T_i)) = o(1)$. Using \eqref{Eq:27XII20-C0} and the expression of $H$ (see \eqref{Eq:HDef}), we have that $H(t, \xi_i,\dot \xi_i)\Big|_{t = -T_i} = o(1)$. Recalling the Pohozaev identity \eqref{Eq:13XI20-M3}, the fact that $ \frac{d}{dt}\dot{\hat K}_{\varepsilon_i,T_i} \leq -\frac{1}{C} e^{-\beta_2(t + T_i)}$ in $(-\infty,-T_i)$ (see \eqref{Eq:27XII20-A1} and \eqref{Eq:27XII20-A2}), and $\xi_i(t) \leq \xi_i(-T_i) - (t + T_i)$ in $(-\infty, -T_i)$ (since $|\dot \xi_i| < 1$), we thus have
\[
o(1) = \int_{-\infty}^{-T_i} |\frac{d}{dt} \hat K_{\varepsilon_i,T_i}| e^{-n\xi_i} dt 
	\geq \frac{1}{C} e^{-n\xi_i(-T_i)}.
\]
which gives $\xi_i(-T_i) \rightarrow \infty$ as $i \rightarrow \infty$ as wanted.

Let $\hat\xi_i(t) := \xi_i(t-T_i) - \xi_i(-T_i)$. Using \eqref{Eq:27XII20-C0}--\eqref{Eq:27XII20-C0X} and the fact that $\hat \xi_i(0) = 0$, we have, after passing to a subsequence, that $\hat \xi_i$ converges in $C^{1,\alpha}_{loc}(\RR)$ to a function $\hat \xi_\infty \in C^{1,1}_{loc}(\RR)$. Also, in view of \eqref{Eq:27XII20-B1}, $F_k[\hat \xi_i](t) = e^{-2k\xi_i(-T_i)} \hat K_{\varepsilon_i,T_i}(t - T_i)$. Hence, since $\xi_i(-T_i) \rightarrow \infty$ as $i \rightarrow \infty$, $\hat \xi_\infty$ satisfies in the viscosity sense the equation
\[
F_k[\hat\xi_\infty] = 0 \text{ and } |\dot{\hat \xi}_\infty| \leq 1 \text{ in } (-\infty,\infty).
\]
By the classification result \cite[Theorem 1.6]{LiNgBocher}, $\tilde\xi_\infty$ takes the form
\begin{equation}
\hat\xi_\infty(t) =  - \frac{2k}{n-2k} \ln (a e^{-\frac{n-2k}{2k}t} + be^{\frac{n-2k}{2k}t}) \text{ for some } a, b \geq 0 \text{ with } a + b > 0.
	\label{Eq:26III21-XF1}
\end{equation}
Now, as $\dot\xi_i(-T_i) \rightarrow \dot{\hat \xi}_\infty(0) = \frac{a  - b}{a  + b}$, in order to conclude Step 2 (and therefore the proof of the theorem), it suffices to show that $b > 0$.\footnote{In fact, it will be seen from the proof below that, when $\beta_2 \neq n - 2k$, we also have $a = 0$.}

\medskip
\noindent\underline{Claim:} The following statements holds.
\begin{enumerate}[(i)]
\item Either $\{e^{-\frac{n-2}{2}T_i}v_i(S)\}$ is bounded, or $e^{-\frac{n-2}{2}T_i}v_i(S)\rightarrow \infty$ and $\beta_2 \leq n-2k$.
\item Either $\{e^{-\frac{n-2}{2}T_i}v_i(N)\}$ is bounded, or $e^{-\frac{n-2}{2}T_i}v_i(N)\rightarrow \infty$ and $\beta_1 \leq n-2k$.

\end{enumerate}

Before proving the claim, let us remark that statement (i) implies that $b > 0$ (and hence $\alpha_-  < 1$) as follows. (Likewise, (ii) implies that $\alpha_+ > -1$.) Let 
\[
\check u_i(r) = e^{-\frac{n-2}{2}(t + \xi_i(t - T_i))} = e^{-\frac{n-2}{2}T_i} u_i(e^{-T_i} r).
\]
Note that $\check u_i$ satisfies $\sigma_k(A^{\check u_i}(r)) = \hat K_{\varepsilon_i,T_i}(\ln r - T_i)$ on $\RR^n$, and, by the claim, either $\{\check u_i(0)\}$ is bounded, or $\check u_i(0) \rightarrow \infty$ and $\beta_2 \leq n-2k$. In the case that $\{\check u_i(0)\}$ is bounded, as $\check u_i(0)$ is the maximum of $\check u_i$ on $\RR^n$ (by the super-harmonicity of $\check u_i$), the first derivative estimates for the $\sigma_k$-Yamabe equation give that $\frac{1}{C}\check u_i(1) \leq \check u_i(r) \leq C\check u_i(1)$ for $r \leq 1$, i.e. $|\hat\xi_i(t) + t - \hat\xi_i(0)| \leq C$ in $(-\infty,0)$. In particular, $\hat \xi_\infty(t) + t$ is bounded as $t \rightarrow -\infty$. Clearly this is true in \eqref{Eq:26III21-XF1} if and only if $a = 0$ and $b > 0$. In the case that $\check u_i(0) \rightarrow \infty$ and $\beta_2 < n-2k$, we have by Theorem \ref{Prop:LocEst}(c) and (d) that there exists an exponent $\varkappa = \varkappa(\beta_2) > 0$ such that  
\[
	\frac{1}{C}\check u_i(1) \leq \check u_i(r) \leq C\check u_i(1) \text{ in } (\check u_i(0)^{-\varkappa},1).
\]
As $\check u_i(0)^{-\varkappa} \rightarrow 0$, this again implies that $\hat \xi_\infty(t) + t$ is bounded as $t \rightarrow -\infty$ and so $a = 0$ and $b > 0$. In the case that $\check u_i(0) \rightarrow \infty$ and $\beta_2 = n-2k$, we can apply Step 2 of the proof of Theorem \ref{main} to $\check \xi_i(t) := \xi_i(t - T_i)$ to obtain that $\hat \xi_i(t) = \check\xi_i(t) - \xi_i(-T_i)$ has a critical point at some $\ln\check\delta_i = O(1)$. It follows that $\hat \xi_\infty$ has at least one critical point, which implies that $a, b > 0$.

It remains to prove the claim. Note that the claim clearly holds if $\beta_1, \beta_2 \leq n - 2k$. Therefore, we may assume without loss of generality that $\beta_2 = \max(\beta_1, \beta_2) > n-2k$. As $\frac{1}{\beta_1} + \frac{1}{\beta_2} \geq \frac{2}{n-2k}$, we have that $\beta_1 < n-2k$, and so (ii) clearly holds. In particular, $\alpha_+ > -1$. It remains to prove (i).

Assume by contradiction that (i) does not hold. Then, since $\beta_2 > n-2k$, $e^{-\frac{n-2}{2}T_i}v_i(S) \rightarrow \infty$ and $\check u_i(0) \rightarrow \infty$. The proof builds upon the identity 
\begin{equation}
H(-T_i, \xi_i(-T_i),\dot\xi_i(-T_i)) = H(T_i, \xi_i(T_i),\dot\xi_i(T_i)),
	\label{Eq:23III21-Est2X}
\end{equation}
which holds in view of the Pohozaev identity \eqref{Eq:13XI20-M2pre} and \eqref{Eq:27XII20-A5}.

As $\alpha_+ > -1$, we have $\dot\xi_i(T_i) = \alpha_+ + o(1) > -1 + \frac{1}{C}$. In particular, by \eqref{Eq:23III21-S1}
\begin{equation}
e^{-\xi_i(T_i)} = o(1) e^{-\xi_i(-T_i)}.
	\label{Eq:23III21-Est2Y}
\end{equation}

Let $\lambda_i := 2^{-\frac{1}{2}}\binom{n}{k}^{-\frac{1}{2k}}\check u_i(0)^{\frac{2}{n-2}} = 2^{-\frac{1}{2}}\binom{n}{k}^{-\frac{1}{2k}} e^{-\frac{n-2}{2}T_i} u_i(0)^{\frac{2}{n-2}} \rightarrow \infty$. Applying Step 2 of the proof of Theorem \ref{main} (see \eqref{Eq:21XII20-Est2}  and \eqref{Eq:06III21-X1}) to $\check \xi_i(t) := \xi_i(t - T_i)$, we obtain
\begin{equation}
H(-T_i, \xi_i(-T_i),\dot\xi_i(-T_i))
	\stackrel{\eqref{Eq:06III21-X1}}{=} e^{O(1)} \lambda_i^{-\beta_2} \stackrel{\eqref{Eq:21XII20-Est2}}{=} e^{O(1)} e^{-\beta_2 \xi_i(-T_i)}.
	\label{Eq:23III21-Est2}
\end{equation}

To estimate $H(T_i, \xi_i(T_i),\dot\xi_i(T_i))$, consider
\[
\tilde u_i(r) = e^{-\frac{n-2}{2}(t + \xi_i(-t + T_i))}  = \Big(\frac{e^{T_i/2}}{r}\Big)^{n-2}u_i\big(\frac{e^{T_i}}{r}\big)
\]
and let us treat separately the case $\{\tilde u_i(0)\}$ is bounded and the case $\tilde u_i(0) \rightarrow \infty$.

Let us start with the case that $\{\tilde u_i(0)\}$ is bounded. As seen earlier, this implies that $\frac{1}{C}\tilde u_i(1) \leq \tilde u_i(r) \leq C\tilde u_i(1)$ for $r\leq 1$, and so
\[
\xi_i(t + T_i) = \xi_i(T_i) + t + O(1) \text{ for } t \geq 0.
\]
By the Pohozaev identity \eqref{Eq:13XI20-M3}, \eqref{Eq:27XII20-A3} and \eqref{Eq:27XII20-A4},
\[
H(T_i, \xi_i(T_i),\dot\xi_i(T_i)) = \int_{T_i}^\infty \dot{\hat K}_{\varepsilon_i,T_i}(t) e^{-n\xi_i}\,dt = e^{-n\xi_i(T_i) + O(1)} .
\]
Using \eqref{Eq:23III21-Est2Y} in the above gives
\[
H(T_i, \xi_i(T_i),\dot\xi_i(T_i)) = o(1) e^{-n\xi_i(-T_i)},
\]
which gives a contradiction to \eqref{Eq:23III21-Est2X} and \eqref{Eq:23III21-Est2}, since $\beta_2 < n$.

Let us now turn to the case $\tilde u_i(0) \rightarrow \infty$. Note that by the same argument that gives $\xi_i(-T_i) \rightarrow \infty$, we also have that $\xi_i(T_i) \rightarrow \infty$. This implies that $\tilde u_i(1)  = o(1)$. By Theorem \ref{Prop:LocEst}(c) and (d), we thus have $\beta_1 > \frac{n-2k}{2}$. Apply Step 2 of the proof of Theorem \ref{main} (see \eqref{Eq:21XII20-Est3}  and \eqref{Eq:06III21-X1}) to $\tilde \xi_i(t) := \xi_i(-t+T_i)$, we can find $\tilde \delta_i =  e^{O(1)} \check\lambda_i^{-(1 - \frac{\beta_1}{n-2k})}$ such that 
\begin{align}
&\xi_i(t)
	= \xi_i(T_i)  - T_i + t + O(1) \text{ in } (T_i, T_i - \ln \tilde \delta_i),
	\label{Eq:23III20-Est3}\\
&H(t, \xi_i , \dot \xi_i)\Big|_{t = T_i - \ln\tilde \delta_i}
	= e^{O(1)} \check\lambda_i^{-\beta_1} = e^{O(1)} e^{-\beta_1(\frac{2\beta_1}{n-2k} - 1)^{-1}\xi_i(T_i)}.
	\label{Eq:23III20-Est4}
\end{align}

Using \eqref{Eq:27XII20-A3}, \eqref{Eq:27XII20-A4} and \eqref{Eq:23III20-Est3} in the Pohozaev identity \eqref{Eq:13XI20-M2pre}, we have
\[
H(t, \xi_i,\dot\xi_i)\Big|_{t = T_i}^{t = T_i - \ln\tilde \delta_i} 
	= \int_{T_i}^{T_i - \ln\tilde \delta_i} \dot{\hat K}_{\varepsilon_i,T_i}(t) e^{-n\xi_i}\,dt
	= e^{O(1)} e^{-n\xi_i(T_i)}.
\]
Putting this and \eqref{Eq:23III20-Est4} together and then using \eqref{Eq:23III21-Est2Y}, we get
\begin{align*}
H(T_i, \xi_i(T_i),\dot\xi_i(T_i)) 
	&= e^{O(1)} e^{-n\xi_i(T_i)} + e^{O(1)} e^{-\beta_1(\frac{2\beta_1}{n-2k} - 1)^{-1}\xi_i(T_i)}\\
	&= o(1) e^{-n\xi_i(-T_i)} + o(1) e^{-\beta_1(\frac{2\beta_1}{n-2k} - 1)^{-1}\xi_i(-T_i)}.
\end{align*}
As $\beta_2 < n$, this together with \eqref{Eq:23III21-Est2} and \eqref{Eq:23III21-Est2X} implies that $\beta_2 > \beta_1(\frac{2\beta_1}{n-2k} - 1)^{-1}$, which contradicts the hypothesis that $\frac{1}{\beta_1} + \frac{1}{\beta_2} \geq \frac{2}{n-2k}$.

Finally, to conclude, we show that with $T \geq N$ and $\varepsilon e^{(n+2k)T} \leq \frac{1}{N}$ as above, \eqref{Eq:29X20-NP} with $K = K_{\varepsilon,T}$ has no positive solution, with or without axisymmetry. This follows from Proposition \ref{Prop:Sym} below.
\end{proof}

\begin{proposition}\label{Prop:Sym}
Suppose $K \in C^1_r(\SSphere^n)$ is positive, non-constant and satisfies $\frac{d}{d\theta} K(\theta) \leq 0$ in $(0,\pi/2)$ and $\frac{d}{d\theta} K(\theta) \geq 0$ in $(\pi/2,\pi)$. Then every positive solution $v \in C^2(\SSphere^n \setminus\{\theta = 0,\pi\})$ to 
\[
\sigma_k(\lambda(A_{g_v})) = K \text{ and } \lambda(A_{g_v}) \in \Gamma_k \text{ on } \SSphere^n\setminus\{\theta = 0,\pi\}
\]
is axisymmetric.
\end{proposition}

\begin{remark}
The conclusion remains valid if $K(x)$ is replaced by $K(x)u^{-a}$ for any constant $a \geq 0$, and/or if $(\sigma_k,\Gamma_k)$ is replaced by more general operators $(f,\Gamma)$ as in \cite{JinLiXu08-ADE}.
\end{remark}

\begin{proof} 
Let $u: \RR^n \setminus \{0\} \rightarrow \RR$ be related to $v$ by \eqref{Eq:19XI20-E1}. Then $u$ is super-harmonic and positive in $\RR^n \setminus \{0\}$. It follows that $\liminf_{y \rightarrow 0} u(y) > 0$ and so
\[
\liminf_{d(x,S) \rightarrow 0} v(x) > 0.
\]
Likewise,
\[
\liminf_{d(x,N) \rightarrow 0} v(x) > 0.
\]
Note that by \cite[Theorem 1.1]{CafLiNir11}, it holds in the viscosity sense that
\begin{equation}
\sigma_k(\lambda(A_{g_v})) \geq K \text{ and } \lambda(A_{g_v}) \in \Gamma_k \text{ on } \SSphere^n.
	\label{Eq:05VI21-A1}
\end{equation}

We can now use the method of moving spheres as in the proof of \cite[Theorem 1.5]{JinLiXu08-ADE} to reach the conclusion. For readers' convenience, we give here a sketch: For any point on $p$ on the equator of $\SSphere^n$ and any $\lambda \in (0,\pi)$, let $\varphi_{p,\lambda}: \SSphere^n \rightarrow \SSphere^n$ be the M\"obius transformation that reflects about the sphere $\partial B_\lambda(p)$ centered at $p$ and of radius $\lambda$ and let $v_{p,\lambda} = |Jac(\varphi_{p,\lambda})|^{\frac{n-2}{2n}}v \circ \varphi_{p,\lambda}$. By the conformal invariance of the equation \eqref{Eq:29X20-NP} and the monotonicity property of $K$ with respect to $\theta$, 
\begin{equation}
\sigma_k(\lambda(A_{g_{v_{p,\lambda}}})) = K \circ \varphi_{p,\lambda} \leq K \text{ in } \SSphere^n \setminus B_\lambda(p).
	\label{Eq:05VI21-A2}
\end{equation}
Using \cite[Lemmas 3.5 and 3.6]{JinLiXu08-ADE}, the number
\[
\bar\lambda_{p} := \sup\Big\{\lambda \in (0,\pi): v_{p,\mu} \leq v \text{ in } \SSphere^n \setminus B_\lambda(p)\Big\}
\]
is well-defined and belongs to $(0,\pi]$. One can then imitate the proof of \cite[Lemma 3.3]{JinLiXu08-ADE} using \eqref{Eq:05VI21-A1}, \eqref{Eq:05VI21-A2} and the strong maximum principle \cite[Theorem 3.1]{CafLiNir11} to show that $\bar\lambda_{p} \geq \pi/2$. Since this holds for every $p$ on the equator, we have that $v$ is axisymmetric.
\end{proof}


\section{Non-compactness: Proof of Theorem \ref{noncompact}}\label{Sec:NonComp}

\begin{proof}[Proof of Theorem \ref{noncompact}] 
We will work in cylindrical coordinates. Fix $2 \le \beta < \frac{n - 2k}{2}$. Consider $\hat K_\varepsilon = 2^{-k}\binom{n}{k} + \varepsilon J$ where $\varepsilon$ is sufficiently small and $J \in C^\infty(\RR)$ is a fixed even function satisfying
\begin{align*}
J(t) 
	&= - e^{\beta t} 
	\text{ for } t \leq -1,\\
\dot J(t) 
	&\leq 0 \text{ for } t \leq 0.
\end{align*}

For $j \geq 0$, let $X_j$ denote the Banach space of functions $\eta \in C^j((-\infty,0])$ such that
\[
\|\eta\|_{j} := \sup_{t \in (-\infty,0]} e^{-(2+\beta)t} \sum_{\ell = 0}^j |\frac{d^\ell}{dt^\ell} \eta| < \infty.
\]

We will show that, for a suitably small but fixed $\varepsilon > 0$, the equation
\begin{equation}
F_k[\xi] = \hat K_\varepsilon, \text{ and }  |\dot \xi| < 1 \text{ in } (-\infty,\infty)
	\label{Eq:03I21-A3}
\end{equation}
has a sequence of even solutions $\xi_i \in C^2(\RR)$ such that
\begin{equation}
\Big(\xi_i - \log\cosh(t + T_i)\Big)\Big|_{(-\infty,0]} \in X_2
	\label{Eq:03I21-A4}
\end{equation}
where $T_i \rightarrow \infty$ as $i \rightarrow \infty$. Once this is done, the conclusion of the theorem follows from Corollary \ref{Cor:31XII20-InftyEnergy}.

\medskip
\noindent\underline{Step 1:} We prove that there exists some small $\varepsilon_0 > 0$ such that for every $0 < \varepsilon < \varepsilon_0$ and $T \ge 1$ there exists $\xi  = \xi(\cdot;\eps,T)\in C^2((-\infty,0])$ which satisfies 
\begin{align}
F_k[\xi]
	&= \hat K_\varepsilon, \text{ and }  |\dot \xi| < 1 \text{ in } (-\infty,0),
	\label{Eq:03I21-A5}\\
&\xi - \log\cosh(t + T) \in X_2
	\label{Eq:03I21-A6}
\end{align}
and the family $\xi(\cdot;\varepsilon, T)$ depends continuously on $(\varepsilon,T)$ in the sense that $(\varepsilon,T) \mapsto \xi(\cdot;\varepsilon,T) - \log \cosh(\cdot + T)$ belongs to $C^1((0,\varepsilon_0) \times [1,\infty);X_2)$.

We claim that it is enough to find $\varepsilon_0 > 0$ such that for every $0 < \varepsilon < \varepsilon_0$ and $T \ge 1$ there exists $\xi = \xi(\cdot;\varepsilon,T) \in C^2((-\infty,-T])$ such that $F_k[\xi] = \hat K_\varepsilon$ in $(-\infty,-T )$ and the function $\eta(t;\eps,T) := \xi(t - T;\varepsilon,T) - \log \cosh t$ belongs to $X_2$ and that $(\eps,T) \mapsto \eta(\cdot;\eps,T)$ belongs to $C^1((0,\varepsilon_0) \times [1,\infty);X_2)$. Indeed, let $(-\infty,T_{\max}) \subset (-\infty,0)$ be the maximal such that $\xi$ satisfies the equation $F_k[\xi]= \hat K_\varepsilon$ in $(-\infty,T_{\max})$, then by the Pohozaev identity \eqref{Eq:13XI20-M3} and the monotonicity of $\hat K_\varepsilon$, we have
\begin{multline}
H(t,\xi,\dot \xi) = \frac{1}{2^k} \binom{n}{k} e^{(2k-n)\xi} (1- \dot \xi^2)^k - \hat K_\varepsilon(t) e^{-n\xi} \\
	= -\int_{-\infty}^t \dot K_{\varepsilon}(\tau) e^{-n\xi(\tau)}\,d\tau > 0 \text{ for } t \in (-\infty,T_{\max}).
		\label{Eq:03I21-A7}
\end{multline}
This implies that, $1 - \dot \xi^2 > 0$ in $(-\infty,T_{\max})$ and $\limsup_{t \rightarrow T_{\max}} |\xi(t)| < \infty$. Standard results on local existence, uniqueness and continuous dependence for ODEs imply that $T_{\max} = 0$ and the claim follows.

By considering $\tilde \xi  = \xi(\cdot - T)$ and using the claim, to finish Step 1, we need to show the existence of some $\varepsilon_0 > 0$ such that, for $0 < \varepsilon < \varepsilon_0$ and $T \geq 1$, there is a solution $\tilde \xi = \tilde \xi(\cdot;\eps, T)$ to
\begin{align}
F_k[\tilde\xi]
	&= \hat K_{\varepsilon e^{-\beta T}}  \text{ in } (-\infty,0),
	\label{Eq:03I21-A5X}\\
& \tilde\xi - \log\cosh t \in X_2
	\label{Eq:03I21-A6X}
\end{align}
and that the map $T \mapsto \eta(\cdot; \eps, T) =  \tilde\xi(\cdot; \eps,T) - \log\cosh t$ belongs to $C^1((0,\eps_0) \times [1,\infty);X_2)$.

Using $\eta$, we recast \eqref{Eq:03I21-A5X}--\eqref{Eq:03I21-A6X} as
\[
\mcA[\eta] = -2^{k-1}\binom{n-1}{k-1}^{-1} \varepsilon\,e^{-\beta T}\,\textrm{sech}^2 t\,e^{\beta t}
\]
where $\mcA: X_2 \rightarrow X_0$ is given by
\begin{align*}
\mcA[\eta] 
	&:=  2^{k-1}\binom{n-1}{k-1}^{-1} \textrm{sech}^2 t\,\Big\{F_k[\log\cosh t + \eta] - \frac{1}{2^k}\binom{n}{k}\Big\}\\
	&= \textrm{sech}^2 t \Big\{  e^{2k\eta} \Big(1 - 2 \cosh t \,\sinh t \,\dot \eta - \cosh^2 t\,\dot\eta^2\Big)^{k-1} \times\\
		&\qquad \times \Big(\frac{n}{2k} + \cosh^2 t \,\ddot \eta - \frac{n-2k}{k} \cosh t \,\sinh t\, \dot \eta - \frac{n-2k}{2k} \cosh^2 t\,\dot\eta^2\Big) - \frac{n}{2k}\Big\}\\
	&=: P(t,\eta,\dot\eta,\ddot\eta).
\end{align*}
Note that for every $\eta \in X_2$, $P(t,\eta,\dot\eta,\ddot\eta)$, $P_\eta(t,\eta,\dot\eta,\ddot\eta)$, $P_{\dot\eta}(t,\eta,\dot\eta,\ddot\eta)$ and $P_{\ddot\eta}(t,\eta,\dot\eta,\ddot\eta)$ are continuous and bounded in $(-\infty,0)$. It follows that $\mcA$ is $C^1$ with derivative
\[
D\mcA[\eta][\varphi] = P_{\ddot\eta}(t,\eta,\dot\eta,\ddot\eta)\ddot\varphi + P_{\dot\eta}(t,\eta,\dot\eta,\ddot\eta)\dot\varphi + P_\eta(t,\eta,\dot\eta,\ddot\eta) \varphi.
\]
Since $\mcA[0] = 0$, by the implicit function theorem (see e.g. \cite[Theorem 2.7.2]{NirenbergLectNotes}), it suffices to check that $\mcL := D\mcA[0]$ is invertible. A direct computation gives
\[
\mcL[\varphi] = \ddot \varphi - (n-2) \tanh t\, \dot \varphi + n\, \textrm{sech}^2 t\,\varphi.
\]
The homogeneous equation $\mcL[\varphi] = 0$ has two linearly independent solutions $\varphi_1(t) = \tanh t$ and $\varphi_2(t)  = e^{(n-2)|t|}(1 + O(e^{t}))$ as $t \rightarrow -\infty$. (For example, we can choose $\varphi_2(t) = \tanh t \int_{c}^t \frac{\cosh^n \tau}{\sinh^2 \tau}\,d\tau$ for some $c < 0$.) 
 In particular, the only solution to $\mcL[\varphi] = 0$ in $X_2$ is the trivial element. Furthermore, for every $\zeta \in X_0$, the solution to $\mcL[\varphi] = \zeta$ in $X_2$ is given by
\[
\varphi(t) = -\varphi_1(t) \int_{-\infty}^t \frac{\zeta(\tau)\varphi_2(\tau)}{\cosh^{n-2}\tau}\,d\tau + \varphi_2(t)\int_{-\infty}^t \frac{\zeta(\tau)\varphi_1(\tau)}{\cosh^{n-2}\tau}d\tau \text{ for } t \in (-\infty,0].
\]
We thus have that $\mcL$ is a bijection from $X_2$ onto $X_0$. This completes Step 1.

\medskip
\noindent\underline{Step 2:} Since $\hat K_\varepsilon$ is even, to show the existence of even solutions to \eqref{Eq:03I21-A3}--\eqref{Eq:03I21-A4}, it suffices to show that, after possibly shrinking $\varepsilon_0$, for every $\varepsilon \in (0,\eps_0)$ there exists a sequence $T_i \rightarrow \infty$ such that the solution $\xi(\cdot;\eps,T_i)$ obtained in Step 1 satisfies in addition that $\dot\xi(0;\eps,T_i) = 0$.

Claim: By shrinking $\eps_0$ if necessary, we have that if  $\dot\xi(t;\eps,T) = 0$ for some $t \in (-\infty,0]$, $\eps \in (0,\eps_0)$ and $T \geq 1$, then $|\ddot\xi(t; \eps, T)| \neq 0$.

Arguing by contradiction, we assume that there exist $\varepsilon_i \rightarrow 0$, $\xi_i  = \xi(\cdot; \eps_i,T_i)$ and $s_i \in (-\infty,0]$ such that $\dot\xi_i(s_i) = 0$ and $\ddot \xi_i(s_i) \rightarrow 0$. From the expression of $F_k[\xi_i]$ and \eqref{Eq:03I21-A5}, we have that $\{\xi_i(s_i)\}$ is bounded. Furthermore, the argument in Section \ref{SSec:LE(a)} (see \eqref{Eq:30IV21-A1}), we have $\xi_i \geq - C$ in $(-\infty,s_i]$ for some $C$ independent of $i$. Recalling \eqref{Eq:03I21-A7}, we have
\[
\lim_{i\rightarrow \infty}H(s_i,\xi_i(s_i),\dot{\xi}_i(s_i)) = \lim_{i\rightarrow \infty} \varepsilon_i \beta \int_{-\infty}^{s_i} e^{\beta \tau} e^{-n\xi_i(\tau)}\,d\tau  = 0.
\]
By Lemma \ref{Eq:23IV21-M1}, we then have $\lim_{i\rightarrow \infty} \ddot \xi_i(s_i) = \ddot\Xi(0) > 0$, which is a contradiction.

We now fix an arbitrary $0 < \eps < \eps_0$. Let $m(T)$ be the number of solutions to $\dot\xi(\cdot;\eps, T) = 0$ in $(-\infty,0]$. Note that by \eqref{Eq:03I21-A6}, $\dot \xi(t;\eps,T) \neq 0$ for large negative $t$. Thus, by the claim, $m(T)$ is finite for every $T \geq 1$. Since $T \mapsto \xi(\cdot;\eps,T) - \log\cosh(\cdot + T)$ belongs to $C^0([1,\infty);X_2)$, we deduce again from the claim that if an interval $(c,d) \subset [1,\infty)$ is such that $\dot\xi(0;\eps,T) \neq 0$ for $T \in (c,d)$, then $m(T)$ is constant for $T \in (c,d)$.  On the other hand, by Theorem \ref{Prop:LocEst}(d), $m(T) \rightarrow \infty$ as $T \rightarrow \infty$. The conclusion is readily seen.
\end{proof}

\appendix
\section{The values of certain integrals}\label{App}

\begin{lemma}
For $0 < b < 2a$, it holds that
\[
\int_0^\infty (1+r^2)^{-a}\, r^{b-1}\,dr
	= \frac{\Gamma(a - \frac{b}{2})\Gamma(\frac{b}{2})}{2\Gamma(a)} .
\]
\end{lemma}

\begin{corollary}\label{Cor:IntegralId}
Suppose $n > 0$. We have
\begin{align*}
\int_0^\infty \frac{r^{n+\beta-1}}{(1 + r^2)^n} \,dr 
	&= \frac{\Gamma(\frac{n-\beta}{2})\Gamma(\frac{n+\beta}{2})}{2\Gamma(n)}
	\text{ for } -n < \beta < n,\\
\int_0^\infty \frac{r^{n-1}}{(1 + r^2)^{\frac{n+2}{2}}} \,dr
	&= \frac{1}{n}.
\end{align*}
\end{corollary}

\begin{proof} We perform the change of variable $x = \frac{1}{1 + r^2}$. Noting that $r^2 = \frac{1-x}{x}$ and $2rdr = -\frac{dx}{x^2}$, we have
\[
\int_0^\infty (1+r^2)^{-a}\, r^{b-1}\,dr = \frac{1}{2}\int_0^1 x^{a - \frac{b}{2} - 1} (1-x)^{\frac{b}{2} - 1}\,dx = \frac{1}{2} B(a-\frac{b}{2},\frac{b}{2}),
\]
where $B$ is the beta function. The conclusion follows from a well-known relation between beta and Gamma functions.
\end{proof}

\newcommand{\noopsort}[1]{}

\end{document}